\documentclass[10pt]{article}
\usepackage[english]{babel}
\usepackage[toc,page]{appendix}
\usepackage{bbm}
\usepackage{fullpage}
\usepackage[top=1in, bottom=1.25in, left=1in, right=1in]{geometry}
\usepackage{etaremune}
\usepackage{enumitem}
\usepackage{graphicx}
\usepackage[labelformat=empty]{caption}
\usepackage{subcaption}
\usepackage{tikz}

\usepackage{needspace}
\newcommand{\conditionalpagebreak}{%
  \needspace{3\baselineskip} % Checks if there is space for 3 lines (2 blank lines + 1 line for 'Bibliography' title)
  \ifdim\pagegoal=\maxdimen % True if the page is empty
  \else % Page is not empty
    \ifdim\pagetotal>\dimexpr\pagegoal-3\baselineskip % Check if the space is less than 3 lines
      \newpage
    \fi
  \fi
}

\usepackage{framed}

\usepackage{amssymb,amsmath,amsthm}

\usepackage{hyperref}

\usetikzlibrary{decorations.pathreplacing,calligraphy}

\DeclareSymbolFont{rsfs}{U}{rsfs}{m}{n}
\DeclareSymbolFontAlphabet{\mathscrsfs}{rsfs}
\usepackage{mathrsfs}

\usepackage{url}

\hypersetup{
  colorlinks   = true, %Colours links instead of ugly boxes
  urlcolor     = blue, %Colour for external hyperlinks
  linkcolor    = blue, %Colour of internal links
  citecolor   = red %Colour of citations
}

\newtheorem{theorem}{Theorem}[section]
\newtheorem{lemma}[theorem]{Lemma}

\newtheorem{proposition}[theorem]{Proposition}
\newtheorem{corollary}[theorem]{Corollary}

\theoremstyle{definition}
\newtheorem{definition}{Definition}

\newtheorem{remark}[theorem]{Remark}

\numberwithin{equation}{section}

\usepackage{times}

\newcommand{\bea}{\begin{eqnarray}}
\newcommand{\eea}{\end{eqnarray}}
\newcommand{\<}{\langle}
\renewcommand{\>}{\rangle}

\newcommand{\wt}{\widetilde}
\newcommand{\op}{\text{op}}
\newcommand{\wh}{\widehat}

\newcommand\eg{{\text{\eg~}}}

\def\eps{{\varepsilon}}

\def\supp{{\rm supp}}

\def\ind{{\mathbbm 1}}

\def\hx{\hat{x}}

\def\hxi{\hat{\xi}}

\def\btau{{\boldsymbol{\tau}}}

\def\hq{{\widehat{q}}}

\def\bg{{\boldsymbol{g}}}

\def\bx{{\boldsymbol{x}}}

\def\bzero{{\mathbf 0}}

\def\cT{{\mathcal T}}
\def\cC{{\mathcal C}}
\def\cE{{\mathcal E}}

\def\op{{\rm op}}
\def\Band{{\rm Band}}

\def\vq{\vec q}

\def\wtZ{\wt{Z}}

\def\sig{\sigma}
\def\bsig{{\boldsymbol {\sigma}}}
\def\vbsig{{\vec \bsig}}

\def\brho{{\boldsymbol \rho}}

\def\vbrho{{\vec \brho}}

\def\by{{\boldsymbol y}}

\def\bv{{\boldsymbol{v}}}

\def\bx{{\boldsymbol{x}}}

\def\bA{\boldsymbol{A}}

\def\de{{\rm d}}

\def\<{\langle}
\def\>{\rangle}

\def\cH{{\cal H}}

\def\by{{\boldsymbol{y}}}

\def\cE{{\mathcal E}}

\def\txi{\widetilde{\xi}}
\def\tx{\widetilde{x}}

\def\b0{{\boldsymbol{0}}}

\def\hH{\widehat{H}}

\def\rd{{\mathrm {rad}}}

\def\bG{{\boldsymbol G}}

\def\Pl{{\mathsf{Pl}}}

\DeclareMathOperator*{\plim}{p-lim}
\DeclareMathOperator*{\pliminf}{p-lim\,inf}
\DeclareMathOperator*{\plimsup}{p-lim\,sup}

\def\cS{{\mathcal S}}

\def\cB{{\mathcal B}}

\def\bzero{\boldsymbol{0}}

\renewcommand{\b}{\mathbf{b}}

\def\fr{\frac}
\def\lt{\left}
\def\rt{\right}

\def\la{\langle}
\def\ra{\rangle}

\def\eps{\varepsilon}

\def\bbE{{\mathbb{E}}}

\def\bbL{{\mathbb{L}}}
\def\bbN{{\mathbb{N}}}
\def\bbP{{\mathbb{P}}}
\def\bbR{{\mathbb{R}}}

\def\bbT{{\mathbb{T}}}

\def\cB{{\mathcal{B}}}

\def\cP{{\mathcal{P}}}
\def\cQ{{\mathcal{Q}}}
\def\cR{{\mathcal{R}}}

\def\Crt{{\mathsf{Crt}}}
\def\LMAX{{\mathsf{LMAX}}}

\def\sH{{\mathscr{H}}}

\def\bg{{\mathbf{g}}}

\DeclareMathOperator*{\argmax}{\arg\max}

\def\sph{\mathrm{sp}}

\newcommand{\norm}[1]{{\lt\|#1\rt\|}}
\newcommand{\tnorm}[1]{{\|#1\|}}

\newcommand{\Span}{\mathrm{span}}

\def\oH{\overline H}

\def\cal{\mathcal}

\def\wtCrt{\wt \Crt}

\def\fR{{\mathfrak{R}}}

\usepackage{mathtools} % for vcentcolon
\usepackage{dsfont} % for mathds
\usepackage{stmaryrd} % for llbracket, rrbracket

\newcommand{\halpha}{{\widehat \alpha}}

\author{
    Brice Huang\thanks{Department of Electrical Engineering and Computer Science, MIT. Email: \texttt{bmhuang@mit.edu}}
    \and
    Mark Sellke\thanks{Department of Statistics, Harvard University.
    Email: \texttt{msellke@fas.harvard.edu}}
}

\title{
A Constructive Proof of the Spherical Parisi Formula
}

\date{}

\begin{document}

\maketitle

\begin{abstract}
    The Parisi formula for the free energy is among the crown jewels in the theory of spin glasses.
    We present a simpler proof of the lower bound in the case of the spherical mean-field model.
    Our method follows the TAP approach developed recently in e.g. \cite{subag2018free}: we obtain an ultrametric tree of pure states, each with approximately the same free energy as the entire model, which are hierarchically arranged in accordance with the Parisi ansatz.
    We construct this tree ``layer by layer'' given the minimizer to Parisi's variational problem.
    On overlap intervals with full RSB, the tree is built by an optimization algorithm due to Subag. On overlap intervals with finite RSB, the tree is constructed by a new truncated second moment argument; a similar argument also characterizes the free energy of the resulting pure states.
    Notably we do not use the Aizenman--Sims--Starr scheme, and require interpolation bounds only up to the 1RSB level.
    Our methods also yield results for large deviations of the ground state, including the entire upper tail rate function for all 1RSB models without external field.
\end{abstract}

\setcounter{tocdepth}{1}
\tableofcontents

\section{Introduction}

We consider the mixed $p$-spin Hamiltonian
% \mscomment{Gaussian external field for Kac-Rice/LDP}
\begin{equation}
    \label{eq:def-HN}
    H_N(\bsig) =
    \sum_{p\ge 1}
    \fr{\gamma_p}{N^{(p-1)/2}} \sum_{i_1,\ldots,i_p=1}^N g_{i_1,\ldots,i_p} \sigma_{i_1}\cdots \sigma_{i_p}.
\end{equation}
% \begin{align*}
%     H_N(\bsig) &= \wtH_N (\bsig) + \bh, \qquad \text{where} \\
%     \wtH_N(\bsig) &=
%     \sum_{p\ge 2}
%     \fr{\gamma_p}{N^{(p-1)/2}} \sum_{i_1,\ldots,i_p=1}^N g_{i_1,\ldots,i_p} \sigma_{i_1}\cdots \sigma_{i_p}.
% \end{align*}
Here $g_{i_1,\ldots,i_p}$ are i.i.d. standard Gaussians, and the coefficients $\gamma_p$ satisfy $\sum_{p\ge 1} 2^p \gamma_p^2 < \infty$.\footnote{For notational convenience, we have written the model's external field as the degree-$1$ term of $H_N$, rather than the traditional $h \sum_{i=1}^N \sigma_i$. In the spherical models we consider, these are of course equivalent by rotational invariance.}
This model is described by the \textbf{mixture function} $\xi(t) = \sum_{p=1}^P \gamma_p^2 t^p$.
% Meanwhile $\bh$ has i.i.d. centered Gaussian coordinates with standard deviation $h\geq 0$.
% This model is described by the pair $(\xi,h)$, where $\xi(t) = \sum_{p=2}^P \gamma_p^2 t^p$ is the \textbf{mixture function}.
For $\bsig^1,\bsig^2 \in \bbR^N$, define the \textbf{overlap} $R(\bsig^1,\bsig^2) = \la \bsig^1,\bsig^2 \ra / N$.
Then $H_N$ is the Gaussian process with covariance
\[
    \bbE H_N(\bsig^1)H_N(\bsig^2) = N \xi(R(\bsig^1,\bsig^2)).
\]
Since the introduction of this model by Sherrington and Kirkpatrick \cite{sherrington1975solvable}, a central question has been to understand the free energy, defined as follows.
In this paper we consider the case of spherical spins.
Let
\[
    \cS_N = \lt\{\bsig \in \bbR^N : \norm{\bsig}_2 = \sqrt{N}\rt\}
\]
be the sphere of radius $\sqrt{N}$.
The partition function and free energy density are defined by
\begin{align}
    \label{eq:def-Z}
    Z_N &= \int_{\cS_N} \exp H_N(\bsig) ~\de \bsig, \\
    \label{eq:def-F}
    F_N &= \frac{1}{N}\log Z_N,
\end{align}
where in \eqref{eq:def-Z} the integration is with respect to the uniform probability measure on $\cS_N$.

The in-probability limit of the free energy was first predicted by Parisi in \cite{parisi1979infinite}, and proved in the breakthrough works of Talagrand \cite{talagrand2006parisi, talagrand2006spherical} and Panchenko \cite{panchenko2013parisi} following decades of progress in the probability and statistical physics communities.
In the equivalent formulation due to Crisanti and Sommers \cite{crisanti1992spherical}, the limiting free energy is described as follows.
Let $x:[0,1] \to [0,1]$ be a right-continuous non-decreasing function such that $x(\hq) = 1$ for some $\hq < 1$ (which may depend on $x$).
Let
\[
    \hx(q) = \int_q^1 x(s)~\de s.
\]
Define the Crisanti-Sommers functional
\begin{equation}
    \label{eq:cs-functional}
    \cP(x;\xi)
    = \fr12 \lt\{
        \xi'(0) \hx(0)
        + \int_0^1 \xi''(q)\hx(q) ~\de q
        + \int_0^\hq \fr{\de q}{\hx(q)}
        + \log (1-\hq)
    \rt\}.
\end{equation}
Note that $\hx(q) = 1-q$ for $q > \hq$, so this functional is independent of $\hq$.
Finally define
\begin{equation}
    \label{eq:cs-functional-inf}
    \cP(\xi) = \inf_x \cP(x;\xi).
\end{equation}
\begin{theorem}[\cite{talagrand2006spherical,chen2013aizenman}]
    \label{thm:parisi}
    The limiting free energy exists and equals
    \[
        \plim_{N\to\infty} F_N
        = \cP(\xi).
    \]
\end{theorem}

\subsection{Main Result}

The purpose of this paper is to give a new constructive proof of the (more difficult) lower bound for $\plim_{N\to\infty} F_N$ in the Parisi formula.
In fact, we will construct an ultrametric tree of pure states, each with the same free energy as the entire model, taking all overlaps in the model's overlap distribution (in fact a slight extension thereof, see \eqref{eq:def-S} below) as predicted by Parisi's ultrametric ansatz \cite{parisi1979infinite, parisi1983order}.

We will use the following characterization of the unique minimizer of \eqref{eq:cs-functional-inf}.
We emphasize that this description (including existence and uniqueness) is a comparatively elementary fact about the variational problem, and as yet says nothing about the free energy $F_N$.
For given $x$, define
\begin{align*}
    F(q) &= \xi'(q) - \int_0^q \fr{\de s}{\hx(s)^2}, &
    f(s) &= \int_0^s F(q)~\de q,
\end{align*}
and
\begin{align}
    \label{eq:def-S}
    S &= \lt\{s \le 1: f(s) = f_{\max}\rt\}, &
    f_{\max} &= \sup\{f(q): q \in [0,1)\}.
\end{align}

\begin{lemma}[{\cite[Proposition 2.1]{talagrand2006spherical}}]
    \label{lem:cs-extremality}
    There is a unique $x$ attaining the infimum \eqref{eq:cs-functional-inf}, which is characterized as follows.
    Let $\nu$ be the probability measure on $[0,1]$ such that $x(q) = \nu([0,q])$.
    Then $\nu(S)=1$.
\end{lemma}
\begin{remark}
    \label{rmk:zero-in-S}
    The measure $\nu$ is the \textbf{overlap distribution} of the model $\xi$.
    Namely in the \emph{generic} models where $\sum_{p~\text{even}:\gamma_p>0}1/p=\sum_{p~\text{odd}:\gamma_p>0}1/p=\infty$,
    one has $\lim_{N\to\infty}\bbE G^{\otimes 2}(f(R(\bsig^1,\bsig^2)))=\int f(x)\de\nu(x)$ for all continuous $f:[-1,1]\to\bbR$, where $G$ is the Gibbs measure of the model and $\bsig^1,\bsig^2$ are independent samples from $G$.
    The same holds for arbitrary $\xi$ modulo small ``generic perturbations'' that do not affect the free energy; see \cite[Chapter 3]{panchenko2013sherrington}.
    % \mscomment{I imagine changing the degree to $P=\infty$ is no problem, but we can also mention only generic perturbations if easier.}
    % \bhcomment{Yeah I don't see any problems with $P=\infty$}
\end{remark}
\begin{remark}
    \label{rmk:ghost-atoms}
    It is possible for $\supp(\nu) \subseteq S$ to be a strict inclusion, and one may think of overlaps $q \in S \setminus \supp(\nu)$ as ``atoms of mass zero" in the overlap distribution.
    Indeed, \cite[Theorem 10]{subag2018free} showed that (for generic models) all overlaps in $S$ are \emph{multi-samplable}, meaning that the Gibbs probability of sampling several points with this pairwise overlap is not exponentially small.
\end{remark}
% When there is ambiguity, we will write $F_{\xi,h}$, $f_{\xi,h}$ $x_{\xi,h}$ and $S_{\xi,h}$ for the above quantities corresponding to model $(\xi,h)$, and will drop the subscript $H_N$ when $h=0$.
% \bhcomment{going to write for $k$ constant first, then modify so $k=e^{cN}$}
The following two definitions describe the geometry of the pure states that our main result will construct.
\begin{definition}
    For $k,D \in \bbN$, let $\bbT = \bbT(k,D)$ be the tree with vertices $\{\emptyset\} \cup [k] \cup [k]^2 \cup \cdots \cup [k]^D$ rooted at $\emptyset$, where $u\in [k]^d$ is the parent of $v \in [k]^{d+1}$ if $u$ is the length-$d$ prefix of $v$.
    For $u,v\in \bbT$, write $|u|$ for the length of $u$ and $u\wedge v$ for the length of the least common ancestor of $u,v$.
    Let $\bbL = \bbL(k,D) = [k]^D$ be the leaf set of $\bbT$.
    % Further, let $u\sim v$ if $u=v$, or $u,v$ are siblings, or one of $u,v$ is the parent of the other.
\end{definition}

\begin{definition}
    \label{def:ultrametric-tree}
    Let $k,D \in \bbN$, $0 \le q_0 < \cdots < q_D \le 1$, $\vq = (q_0,\ldots,q_D)$, and $\delta > 0$.
    A $(k,D,\vq,\delta)$-ultrametric tree is a collection of points $(\bsig^u)_{u \in \bbT}$ such that
    \begin{equation}
        \label{eq:ultrametric}
        |R(\bsig^u,\bsig^v) - q_{u\wedge v} | \le \delta, \qquad u,v\in \bbT.
    \end{equation}
    % A $(k,D,\vq,\delta)$-locally ultrametric tree is such a collection where \eqref{eq:ultrametric} is only required to hold for $u\sim v$.
\end{definition}
For $q\in [0,1)$, define
\[
    E(q) = \fr12 \lt\{
        \xi'(0) \hx(0)
        + \int_0^q \xi''(s)\hx(s) ~\de s
        + \int_0^q \fr{\de s}{\hx(s)}
    \rt\}.
\]
For $k\in \bbN$, $\delta > 0$, and $\norm{\bsig}_2 \le \sqrt{qN}$, let
\begin{align*}
    &\Band_{k,q,\delta} (\bsig)=
    \Big\{
        \vbrho = (\brho^1, \ldots, \brho^k) :
        \norm{\brho^i}_2 = \sqrt{qN},
        \\
        &\quad\quad\quad\quad
        |R(\brho^i - \bsig, \bsig)| \le \delta,
        |R(\brho^i - \bsig, \brho^j - \bsig)| \le \delta,~~~\forall 1\leq i<j\leq k
    \Big\}.
\end{align*}

\begin{theorem}
    \label{thm:main}
    For any $\delta, \eps > 0$, $D\in \bbN$ and increasing $q_0,\ldots,q_D \in S$ with $q_D = \sup(S)$, there exists $c>0$ such that the following holds for any $k\le e^{cN}$.
    With probability $1-e^{-cN}$ there is a $(k,D,\vq,\delta)$-ultrametric tree $(\bsig^u)_{u \in \bbT}$ with the following properties.
    \begin{enumerate}[label=(\roman*)]
        \item \label{itm:thm-main-energy} Energy of tree nodes: for each $u\in \bbT$, $\fr1N H_N(\bsig^u) \ge E(q_{|u|}) - \eps$.
        \item \label{itm:thm-main-pure} Free energy of pure states: for each $u\in \bbL$,
        \begin{equation}
            \label{eq:pure-fe-lower-bd}
            \fr{1}{kN} \log \int_{\Band_{k,1,\delta}(\bsig^u)}
            \exp \lt(\sum_{i=1}^k H_N(\brho^i) \rt)
            ~\de \vbrho
            \ge \cP(\xi) - \eps.
        \end{equation}
        % \[
        %     \inf_{\bv^1,\ldots,\bv^{k-1} \in \cS_N}
        %     \fr{1}{N} \log \int_{\substack{Band(\bsig^u) \\ |R(\bsig,\bv^i)| \le \delta}}
        %     \exp(H_N(\bsig)) ~\de \bsig
        %     \ge \cP(\xi,h) - \eps
        %     % \fr12 \lt(\xi(1) - \xi(q_D) - \xi'(q_D)(1-q_D)\rt) + \fr12 \log (1-q_D) o(1)
        % \]
        % \bhcomment{this is a bit spitballed for now, may fix when $k=e^{cN}$}
    \end{enumerate}
\end{theorem}
The free energy lower bound \eqref{eq:pure-fe-lower-bd} holds even in a ``$k$-replicated" sense, where we average over $k$ replicas $\brho^i$ constrained to be nearly orthogonal.
This of course lower bounds the free energy of a single replica, as
\begin{equation}
    \label{eq:pure-fe-lower-bd-no-replica}
    \int_{\Band_{k,1,\delta}(\bsig^u)}
    \exp \lt(\sum_{i=1}^k H_N(\brho^i) \rt)
    ~\de \vbrho
    \le
    \lt(
        \int_{\Band_{1,1,\delta}(\bsig^u)}
        \exp H_N(\brho)~\de \brho
    \rt)^k,
\end{equation}
and this shows there is no free energy cost to taking $k$ approximately orthogonal replicas.
In our proof of Theorem~\ref{thm:main}, we derive an analogous $k$-replicated lower bound on the energy increment from any $\bsig^u$, where $u \in \bbT \setminus \bbL$, to its children $\bsig^{u1},\ldots,\bsig^{uk}$, see Theorem~\ref{thm:main-avg-case}\ref{itm:main-avg-case-energy}; this allows us to construct the ultrametric tree $(\bsig_u)_{u\in \bbT}$.

As a consequence of Theorem~\ref{thm:main}, we obtain the lower bound in the Parisi formula.
\begin{corollary}
\label{cor:parisi-LB}
    We have $\pliminf_{N\to\infty} F_N \ge \cP(\xi)$.
\end{corollary}
\begin{proof}
    Equations \eqref{eq:pure-fe-lower-bd} and \eqref{eq:pure-fe-lower-bd-no-replica} imply
    \begin{align*}
        F_N = \fr1N \log \int_{S_N} \exp H_N(\brho) ~\de \brho
        &\ge
        \fr1N \log \int_{\Band_{1,1,\delta}(\bsig^u)} \exp H_N(\brho) ~\de \brho \\
        &\ge
        \fr{1}{kN} \log \int_{\Band_{k,1,\delta}(\bsig^u)}
        \exp \lt(\sum_{i=1}^k H_N(\brho^i) \rt)
        ~\de \vbrho
        \ge
        \cP(\xi) - \eps.
    \end{align*}
    Since this holds for any $\eps>0$, the result follows.
\end{proof}

% \begin{proof}
%     Since
%     \[
%         \int_{q_D}^1 \xi''(q)\hx(q)~\de q
%         = \int_{q_D}^1 \xi''(q)(1-q) ~\de q
%         = \xi(1) - \xi(q_D) - \xi'(q_D)(1-q_D),
%     \]
%     the free energy is lower bounded by
%     \begin{align*}
%         &E(q_D) + \fr12 \int_{q_D}^1 \xi''(q)\hx(q)~\de q + \fr12 \log (1-q_D) \\
%         &= \fr12 \lt\{
%             h^2 \hx(0)
%             + \int_0^1 \xi''(q)\hx(q)~\de q
%             + \int_0^{q_D} \fr{\de q}{\hx(q)} ~\de q
%             + \log (1-q_D)
%         \rt\} \\
%         &= \fr12 \lt\{
%             h^2 \hx(0)
%             + \int_0^1 \xi'(q)x(q)~\de q
%             + \int_0^{q_D} \fr{\de q}{\hx(q)} ~\de q
%             + \log (1-q_D)
%         \rt\} = \cP(x;\xi,q),
%     \end{align*}
%     where in the last line we integrate by parts.

% \end{proof}

Taking the temperature to zero, we also obtain the following consequence on near-ground states.

\begin{corollary}
\label{cor:many-orthogonal}
    Let $\nu_\infty$ be the zero-temperature overlap measure defined in \eqref{eq:def-nu-infty}, and let $q_1<q_2<\dots<q_D=1$ lie in $\supp(\nu_\infty)$.
    Then for any $\delta,\eps>0$, there exists $c>0$ such that for all $k\leq e^{cN}$, with probability $1-e^{-cN}$ there exists a $(k,D,\vq,\delta)$-ultrametric tree $\bbT\subseteq \cS_N$ such that:
    \begin{equation}
    \label{eq:near-ground-states-tree}
    \min_{\bsig\in \bbL} H_N(\bsig)/N\geq \max_{\bsig\in\cS_N} H_N(\bsig)/N -\eps.
    \end{equation}
    In fact the same holds with $\supp(\nu_\infty)$ replaced by $T$ from \eqref{eq:T-def}, which is the zero-temperature analog of $S$.
\end{corollary}

Finally in Section~\ref{sec:LDP}, we study the large deviations of the ground state energy $GS_N=\max_{\bsig\in\cS_N} H_N(\bsig)/N$.
Confirming predictions of \cite{fyodorov2023replica}, we determine the upper tail rate function for all 1RSB $\xi$ with $\gamma_1=0$, and identify a sharp phase transition in the speed from $O(N)$ to $\Omega(N^2)$ for general $\xi$.
The former follows by the methods of Section~\ref{sec:moment}; the latter uses Corollary~\ref{cor:many-orthogonal}, and in particular the fact that $k$ can be taken exponentially large in $N$.
% \mscomment{Just one paragraph early on so readers know this part exists.}

\begin{remark}
    Related ultrametric decompositions for spin glasses have appeared in several previous works, including \cite{jagannath2017approximate,subag2018free,chatterjee2021average}.
    Our work follows the approach of \cite{subag2018free}, and in particular uses a uniform concentration idea introduced therein.
    % which implies that a spin glass can be decomposed into smaller parts, such that the free energy of the original spin glass is lower bounded by the sum of ground state and free energies of the parts (\bhcomment{ref}).
    In the aforementioned works, this idea is used with previously established properties of Gibbs measures and free energies to construct ultrametric decompositions.

    Our work proceeds in the opposite direction, using this idea to prove the lower bound in the Parisi formula.
    Uniform concentration reduces the proof of the general lower bound to four special cases, which we term fundamental types (see Section~\ref{sec:model-types}).
    For two of these cases, elementary proofs of the lower bound are known that do not depend on the full Parisi formula.
    Our main contribution is to provide such a proof for the two remaining cases, and thereby complete an independent proof of the lower bound.
    As a consequence, we are able to construct the tree in Theorem~\ref{thm:main} one layer at a time ``by hand."

    % Our approach to the construction is also quite different, going one layer at a time ``by hand''
    % , of the two of these cases
    % The aforementioned works proceed in the opposite direction, using previously established properties of Gibbs measures and/or free energies to construct their decompositions.
    % This idea reduces the proof of the Parisi formula lower bound in general models to four special cases, which we term fundamental types.
    % Our approach to the construction is also quite different, going one layer at a time ``by hand'', which is why we can use it as a proof technique to study free energies.

    A notable aspect of Theorem~\ref{thm:main} is that it gives an ultrametric tree with exponentially large branching factor at each level.
    At zero temperature, with $\gamma_1=0$ so that $0\in \supp(\nu_\infty)$, the existence of many approximately orthogonal near ground states is closely related to disorder chaos; see \cite{chatterjee2014superconcentration,ding2015multiple,auffinger2018energy,chen2018energy,eldan2020simple}.
    Corollary~\ref{cor:many-orthogonal} is the first to show
    $e^{cN}$ approximately orthogonal near ground states exist without additional assumptions on $\xi$.
\end{remark}

% \begin{remark}
%     Our proof of Theorem~\ref{thm:main} is in our opinion easier than previous ones \cite{talagrand2006parisi,talagrand2006spherical,panchenko2013parisi}.
%     We take a quite direct approach, and the most technically sophisticated inputs required are arguably RS and 1RSB interpolation upper bounds (see Remark~\ref{rem:interpolation-and-positivity} and above).
%     However despite being elementary, the approaches we combine are conceptually powerful and were developed relatively recently (see Subsection~\ref{subsec:background}).
%     We believe this work further demonstrates the strength and versatility of these ideas.
% \end{remark}

\subsection{Previous Approaches to the Parisi Formula}

Mean-field spin glasses were introduced in \cite{sherrington1975solvable,derrida1981random} to model disordered magnetic materials.
Soon after, \cite{thouless1977solution,de1978stability} observed that the replica-symmetric ansatz made by Sherrington and Kirkpatrick could not be correct at low temperatures.
This was resolved by Parisi's ground-breaking replica symmetry breaking solution, yielding a formula for the free energy at any temperature \cite{parisi1979infinite,parisi1980sequence,parisi1983order}.
Several mysterious, fascinating features were present in this highly non-rigorous ansatz, including the prediction of ultrametricity \cite{mezard1984nature,mezard1984replica,mezard1987spin} which led to the introduction of Ruelle cascades \cite{ruelle1987mathematical}.
\emph{Spherical} spin glasses were also introduced in \cite{crisanti1992spherical}, where it was observed that a similar replica ansatz should apply and lead to simpler formulas.
Despite this, rigorous results were for some time mainly restricted to high-temperature settings with similar behavior to classical spin systems \cite{aizenman1987some,comets1995sherrington}.

A crucial breakthrough was made in \cite{guerra2002thermodynamic}, which proved the existence of a limiting free energy at all temperatures using the interpolation method.
Then in \cite{guerra2003broken}, Guerra gave an inspired interpolation upper bound for the free energy, which matched the conjectural Parisi ansatz.
Finally, Talagrand used a difficult interpolation scheme (analyzing its intermediate-time behavior using another interpolation) to prove the Parisi formula at all temperatures for both Ising and spherical models, with the slight restriction that $\gamma_p=0$ for all $p$ odd \cite{talagrand2006parisi,talagrand2006spherical}.

While Talagrand's solution was rather complicated, it was realized in \cite{aizenman2003extended} that Guerra's upper bound allows one to transparently deduce an extended variational formula for the free energy over a space of ``random overlap structures'', relaxing the ultrametricity condition in the Parisi ansatz.
Combined with asymptotic ultrametricity of the Gibbs measures shown by \cite{panchenko2013parisi}, this led to an alternate proof of the Parisi formula with no parity restriction \cite{panchenko2014parisi,chen2013aizenman}.
More recently, the intrinsic behavior of the associated variational formula was clarified via connection to stochastic Hamilton-Jacobi equations \cite{auffinger2015parisi,jagannath2016dynamic}, and a limiting zero temperature formula was obtained for the ground state energy \cite{auffinger2017parisi,chen2017parisi}.

\section{Fundamental Model Types}
\label{sec:model-types}

In this subsection we define four types of models, which we term \textbf{topologically trivial}, \textbf{strictly RS}, \textbf{strictly 1RSB}, and \textbf{strictly FRSB}.
We state lower bounds on the free energy of strictly RS models and the ground state energies of the other three model types.
These models will serve as the basic building blocks for any overlap distribution.
The proof of Theorem~\ref{thm:main}, carried out in Section~\ref{sec:building-model}, will decompose a model $\xi$ into several sub-models of these types  and apply these results.
These lower bounds are then combined back together via a uniform concentration lemma of \cite{subag2018free}.

\begin{definition}
    The model $\xi$ is \textbf{strictly RS} if $S = \{0\}$.
\end{definition}
The remaining three types of models will be defined using a zero-temperature version of the Crisanti-Sommers formula introduced in \cite{chen2017parisi}, which is obtained as a limit of \eqref{eq:cs-functional}.
For $\alpha : [0,1] \to [0,+\infty)$ and $L > \int_0^1 \alpha(s)~\de s$, let
\[
    \halpha(q) = L - \int_0^1 \alpha(s)~\de s.
\]
Then define
\begin{equation}
    \label{eq:cs-functional-0temp}
    \cQ(L,\alpha;\xi) = \fr12 \lt\{
        \xi'(0) L +
        \int_0^1 \xi''(q) \halpha(q) ~\de q
        + \int_0^1 \fr{\de q}{\halpha(q)}
    \rt\}
\end{equation}
and
\begin{equation}
    \label{eq:cs-functional-0temp-inf}
    \cQ(\xi) = \inf_{L,\alpha} \cQ(L,\alpha;\xi).
\end{equation}
\begin{theorem}[{\cite[Theorem 1]{chen2017parisi}}]
\label{thm:chen-sen-spherical}
    The limiting ground state energy of the model $\xi$ is
    \[
        \plim_{N\to\infty} \fr{1}{N} \max_{\bsig \in \cS_N} H_N(\bsig) = \cQ(\xi).
    \]
\end{theorem}

The minimizer of \eqref{eq:cs-functional-0temp-inf} has a similar characterization to Lemma~\ref{lem:cs-extremality} above.
For given $L,\alpha$, define
\begin{align}
    \label{eq:def-G}
    G(q) &= \xi'(q) - \int_0^q \fr{\de s}{\halpha(s)^2}, &
    g(s) &= \int_s^1 G(q)~\de q.
\end{align}
Similarly to before, we let $\nu_{\infty}$ be the finite Borel measure on $[0,1]$ defined by
\begin{equation}
    \label{eq:def-nu-infty}
    \nu_\infty([0,q]) = \alpha(q) \qquad \forall q\in [0,1]
\end{equation}
and define the set
\begin{equation}
\label{eq:T-def}
    T = \{q \in [0,1] : g(q) = 0\}.
\end{equation}
Note that we always have $1\in T$.
% \bhcomment{a little annoying that $g$ is not defined consistently with $f$ above, flip one of these}
\begin{lemma}[{\cite[Theorem 2]{chen2017parisi}}]
    \label{lem:cs-extremality-0temp}
    There is a unique $(L,\alpha)$ attaining the infimum \eqref{eq:cs-functional-0temp-inf}, which is characterized by the following properties:
    \[
        G(1)=0;\quad\quad
        \min_{q\in [0,1]} g(q) = 0;\quad\quad
        \nu_{\infty}(T^c) = 0.
    \]
\end{lemma}

% We similarly write $G_{\xi,h}$, $g_{\xi,h}$, $L_{\xi,h}$, $\alpha_{\xi,h}$, and $T_{\xi,h}$ when there is ambiguity in $\xi,h$, and drop the subscript $H_N$ when $h=0$.
\begin{definition}
    The model $\xi$ is \textbf{topologically trivial} if $T = \{1\}$, \textbf{strictly 1RSB} if $T = \{0,1\}$, and \textbf{strictly FRSB} if $T = [0,1]$.
\end{definition}
\begin{remark}
    \label{rmk:need-no-field}
    Note that $\xi$ can only be strictly RS, strictly 1RSB, or strictly FRSB if $\xi'(0)=\gamma_1^2=0$, i.e. there is no external field.
    Indeed if $\xi'(0)>0$, then $F(q), G(q) > 0$ for $q$ in a neighborhood of $0$, so we cannot have $0\in S, T$.
    % Henceforth, we say $\xi$ is strictly RS, strictly 1RSB, or strictly FRSB if $(\xi,0)$ is.
    Conversely, by Lemma~\ref{lem:type-characterization}\ref{itm:char-topologically-trivial}, $\xi$ can only be topologically trivial if $\xi'(1) \ge \xi''(1)$, which implies $\xi'(0) > 0$ except in the simple case that $\xi$ is quadratic.
\end{remark}

\subsection{Proof Outline for Theorem~\ref{thm:main}}
\label{subsec:proof-outline}

\paragraph{Decomposition into Fundamental Types}
We will construct the ultrametric tree in Theorem~\ref{thm:main} layer by layer, as follows.
Let $x$ attain the infimum in \eqref{eq:cs-functional-inf}; it is known from \cite{jagannath2018bounds} (see Lemma~\ref{lem:intervals}) that the associated $S$ \eqref{eq:def-S} is a finite union of intervals (including possibly atoms).
We may assume without loss of generality that the sequence $q_0,\ldots,q_D$ contains all endpoints of these intervals, so that $q_0 = \inf S$, $q_D = \sup S$, and each interval $[q_d,q_{d+1}]$ (where we take as convention $q_{-1}=0$, $q_{D+1}=1$) either is contained in $S$ or intersects $S$ at exactly its endpoints.
Recall that (modulo Remark~\ref{rmk:ghost-atoms}) $S$ is the support of the overlap distribution $\nu([0,q]) = x(q)$, so these intervals comprise the overlap support and overlap gaps of the model $\xi$.
% overlap distribution $\nu$, given by $\nu([0,q]) = x(q)$,

Following a construction of \cite{subag2018free}, we define a sub-model of $\xi$ for each interval $[q_d,q_{d+1}]$ (see \eqref{eq:def-xi-d}), which represents the landscape of $H_N$ on an orthogonal band of radius $\sqrt{(q_{d+1}-q_d)N}$ around a point of radius $\sqrt{q_d N}$.
Due to the choice of $q_0,\ldots,q_D$, the sub-model for $[0,q_0]$ will be topologically trivial, that for $[q_D,1]$ strictly RS, and the remaining sub-models either strictly 1RSB or strictly FRSB (see Figure~\ref{fig:decomposition}).
We then prove sharp lower bounds for the free energy of each strictly RS component and the ground state energy of the remaining components.
Furthermore, for all but topologically trivial components, this lower bound will hold in a $k$-replicated sense.
This allows us to combine the bounds and construct the tree described by Theorem~\ref{thm:main} in Section~\ref{sec:building-model}.

\begin{figure}
    \centering
    \begin{tikzpicture}[scale=0.6]
        \node at (0,0) {{$\bullet$}};
        \draw[thick,->] (0,0) -- (21,0);
        \draw[thick,->] (0,0) -- (0,10.5);
        \draw[dashed] (5,0) -- (5,2);
        \draw[dashed] (9,0) -- (9,2);
        \draw[dashed] (12,0) -- (12,8);
        \draw[dashed] (17,0) -- (17,10);
        \draw[dashed] (20,0) -- (20,10);
        \draw[thick] (20,-.3) -- (20,.3);
        \draw[thick] (-.3,10) -- (.3,10);
        \node at (21.5,0) {{\small $q$}};
        \node at (0,11) {{\small $x(q)$}};
        \node at (20,-.7) {{\small $1$}};
        \node at (-.7,10) {{\small $1$}};
        \draw (5,2) -- (9,2);
        \draw plot [smooth] coordinates {(9,2) (9.25,2.0625) (9.5,2.25) (9.75,2.5625) (10,3) (11,5) (11.25,5.4375) (11.5,5.75) (11.75, 5.9375) (12,6)};
        \draw (12,8) -- (17,8);
        \draw (17,10) -- (20,10);
        \draw [thick,decorate,decoration = {calligraphic brace}] ( 5-.3,-.2) --  ( 0+.3,-.2);
        \draw [thick,decorate,decoration = {calligraphic brace}] ( 9-.3,-.2) --  ( 5+.3,-.2);
        \draw [thick,decorate,decoration = {calligraphic brace}] (12-.3,-.2) --  ( 9+.3,-.2);
        \draw [thick,decorate,decoration = {calligraphic brace}] (17-.3,-.2) --  (12+.3,-.2);
        \draw [thick,decorate,decoration = {calligraphic brace}] (20-.3,-.2) --  (17+.3,-.2);
        \node at (2.5,-.7) {{\footnotesize Topologically trivial}};
        \node at (7,-.7) {{\footnotesize 1RSB}};
        \node at (10.5,-.7) {{\footnotesize FRSB}};
        \node at (14.5,-.7) {{\footnotesize 1RSB}};
        \node at (18.5,-.7) {{\footnotesize RS}};
    \end{tikzpicture}
    \caption{Decomposition of an overlap distribution into fundamental components. Our proof of Theorem~\ref{thm:main} combines lower bounds on the free or ground state energy of each piece.}
    \label{fig:decomposition}
\end{figure}

The aforementioned lower bounds are stated in Subsection~\ref{subsec:fundamental-lower-bounds}.
For topologically trivial and strictly FRSB models, they are already known, using the Kac--Rice formula and an explicit optimization algorithm respectively.
Namely \cite{fyodorov2014topology, belius2022triviality} showed that topologically trivial models have w.h.p. two critical points, the global maximum and minimum, and characterized their energies (the intuition is that the strong external field overpowers the remainder of the disorder).
Meanwhile \cite{subag2018following} showed how to construct an approximate ground state of any strictly FRSB model (which can easily be made $k$-replicated).
It is worth pointing out that locating the ground state of the topologically trivial model in the first stage is analogous to the recentering step from the conditional second moment method approach used in \cite{bolthausen2018morita,ding2019capacity,brennecke2022replica} for related problems.
These works studied replica-symmetric models with external field, which for the purposes of this paper amount to a combination of topologically trivial and strictly RS models.
% For replica-symmetric models with external field $\gamma_1\neq 0$, our proof amounts to an implementation of the strategy outlined there, with a sharper analysis of the conditionally replica-symmetric model.

Given this, our primary remaining tasks are the lower bounds for strictly RS and 1RSB models.
We prove these bounds using a new truncated second moment argument, explained below.

We note that the ``decomposition'' strategy we follow was introduced by \cite{subag2018free}, and has been subsequently implemented to recover the Parisi formula in several restricted settings.
\cite{arous2020geometry} showed that under the subset of the 1RSB phase called ``Condition $M$'', the free energy can be understood at low enough temperature by decomposing the model into 1RSB and RS parts.
A similar ``full RSB + RS'' decomposition was observed in \cite{subag2018following}.
In the shattered phase (a subset of the replica-symmetric phase), \cite{arous2021shattering} used a similar decomposition to understand geometric properties of the landscape.
These results were highly suggestive and motivational for our work.
However they did not apply in fully general models because the RS and 1RSB phases could not always be handled without depending on the Parisi formula.
Because we solve these cases independently without this dependency, we can combine this with the above ingredients to arrive at a new proof of the Parisi formula.

\paragraph{Truncated Second Moment}

It is natural to study the free energy of strictly RS models using the second moment method.
A direct calculation shows that if
\begin{equation}
    \label{eq:2mt-rs-phase}
    \xi(q) + \fr12 \log (1-q^2) \le 0 \qquad \forall q\in [0,1),
\end{equation}
then the second moment method succeeds and $\plim_{N\to\infty} F_N = \fr12 \xi(1)$.
However, this does not encompass the entire RS phase.
Indeed, by Lemma~\ref{lem:cs-extremality} (with $x \equiv 1$), the model is strictly RS if (and only if)
\begin{equation}
    \label{eq:true-rs-phase}
    \xi(q) + q + \log(1-q) \le 0 \qquad \forall q\in [0,1)
\end{equation}
with equality only at $q=0$.
These conditions do not agree, so even in the strictly RS phase it is possible for the dominant contribution to the second moment to come from pairs with nonzero overlap.
Similarly, \cite[Condition M]{arous2020geometry} gives a condition under which the second moment method, applied to a suitable critical point count, identifies the ground state energy.
However, this condition does not encompass the entire zero-temperature 1RSB phase, given by \eqref{eq:1rsb-test} below.

We overcome these difficulties by truncating the moment calculation to \emph{typical} points, see Definitions~\ref{def:fe-typical} and \ref{def:gs-typical}.
Roughly, $\bsig \in \cS_N$ will be said to be \textbf{free energy typical} if for any $q$, the free energy on the band
\[
    \Band_q(\bsig) = \{\brho \in \cS_N : R(\brho,\bsig) = q\}
\]
is at most $\cP(\xi)$, and \textbf{ground state typical} if the ground state energy on such bands is at most $\cQ(\xi)$.
We will show that typical points dominate the respective first moments for both of the model types described above, by applying Guerra's interpolation bound to appropriate conditional models.
This implies that truncation has only a slight effect on the first moment.
On the other hand, truncation immediately ensures that the second moment is dominated by pairs of orthogonal points, causing the second moment method to succeed throughout the RS and 1RSB regimes.

% Our chief innovation is a second moment argument, truncated to points $\bsig \in \cS_N$ where the free or ground state energy on bands orthogonal to $\bsig$ is maximized by the band with zero overlap.
% We will see that truncation to such points does not affect the first moment on an exponential scale, but ensures that the second moment is dominated by pairs of orthogonal points throughout the RS and 1RSB regimes.

To prove these highly non-obvious typicality properties, we rely on the following upper bounds which follow from the interpolation method of \cite{guerra2003broken}.
In fact we require only replica-symmetric bounds at positive temperature, and 1RSB bounds at zero temperature.

\begin{proposition}
\label{prop:RS-interpolation}
    For any $x$ as in \eqref{eq:cs-functional} of the form $x(q)=\ind\{q\ge q_*\}$,
    \[
    \plimsup_{N\to\infty} F_N\leq \cP(x;\xi).
    \]
\end{proposition}

\begin{proposition}
\label{prop:1RSB-interpolation}
    For any $(L,\alpha)$ as in \eqref{eq:cs-functional-0temp} of the form $\alpha(q)=u \ind\{q\ge q_*\}$,
    \[
    \plimsup_{N\to\infty} \frac{1}{N}\max_{\bsig\in\cS_N} H_N(\bsig)
    \leq \cQ(L,\alpha;\xi).
    \]
\end{proposition}

\begin{remark}
\label{rem:interpolation-and-positivity}
Proposition~\ref{prop:RS-interpolation} follows directly from \cite{talagrand2006spherical}, while Proposition~\ref{prop:1RSB-interpolation} follows from the zero temperature limits taken in \cite{chen2017parisi} or \cite{jagannath2017low}.
We note that proving them for general $\xi$ requires Talagrand's positivity principle \cite{panchenko2007note,talagrand2011mean2}, which follows from the Ghirlanda--Guerra identities.
If one wishes to avoid these to keep things elementary, one may assume throughout this paper that $\xi$ is convex on $[-1,1]$.
\end{remark}

\subsection{Lower Bounds for Free and Ground State Energy}
\label{subsec:fundamental-lower-bounds}

The following propositions lower bound the free or ground state energies in the four fundamental model types.
They are special cases of Theorem~\ref{thm:main} where $\xi$ is of these types.
\begin{proposition}
    \label{prop:fe-rs}
    Suppose $\xi$ is strictly RS.
    For all $\delta, \eps > 0$, there exists $c = c(\xi,\delta,\eps)$ such that if $k\le e^{cN}$, with probability $1-e^{-cN}$,
    % If $\xi$ is strictly RS, for all $k\in \bbN$ and $\delta,\eps > 0$, w.h.p.
    \[
        \fr{1}{kN} \log
        \int_{\Band_{k,1,\delta}(\bzero)}
        \exp \lt(\sum_{i=1}^k H_N(\bsig^i)\rt) ~\de \vbsig
        \ge \cP(\xi) - \eps.
    \]
\end{proposition}
\begin{proposition}
    \label{prop:opt-1rsb}
    Suppose $\xi$ is strictly 1RSB.
    For all $\delta, \eps > 0$, there exists $c = c(\xi,\delta,\eps)$ such that if $k\le e^{cN}$, with probability $1-e^{-cN}$ there exists $\vbsig \in \Band_{k,1,\delta}(\bzero)$ such that
    \[
        \fr1N H_N(\bsig^i) \ge \cQ(\xi) - \eps \qquad \forall i\in [k].
    \]
    % If $\xi$ is strictly 1RSB, for all $k\in \bbN$ and $\delta,\eps>0$, w.h.p. there exists $\vbsig \in \Band_{k,1,\delta}(\bzero)$ such that
    % there are $k$ points $\bsig^1,\ldots,\bsig^k \in \cS_N$ such that $|R(\bsig^i,\bsig^j)| \le \delta$ for all $i\neq j$ and
    % for all $i \in [k]$.
\end{proposition}
\begin{proposition}
    \label{prop:opt-topologically-trivial}
    % The model $(\xi,h)$ is topologically trivial if and only if $\xi'(1) + h^2 \ge \xi''(1)$.
    % If this occurs,
    Suppose $\xi$ is topologically trivial.
    For all $\eps>0$, there exists $c = c(\xi,\eps)$ such that with probability $1-e^{-cN}$, there exists $\bsig \in S_N$ such that $\fr1N H_N(\bsig) \ge \cQ(\xi) - \eps$.
    % If $(\xi,h)$ is topologically trivial, for all $\delta,\eps>0$ w.h.p. there exists $\bsig \in \cS_N$ such that
\end{proposition}
\begin{proposition}
    \label{prop:opt-frsb}
    If $\xi$ is strictly FRSB, the conclusion of Proposition~\ref{prop:opt-1rsb} also holds.
\end{proposition}

We will prove Propositions~\ref{prop:fe-rs} and \ref{prop:opt-1rsb} in Section~\ref{sec:moment} by the aforementioned truncated second moment argument.
Propositions~\ref{prop:opt-topologically-trivial} and \ref{prop:opt-frsb} are known from previous work, and we outline their proofs below.

\begin{lemma}
    \label{lem:type-characterization}
    The following holds.
    \begin{enumerate}[label=(\alph*)]
        \item \label{itm:char-topologically-trivial} If $\xi$ is topologically trivial, then $\xi'(1) \ge \xi''(1)$ and $\cQ(\xi) = \sqrt{\xi'(1)}$.
        \item \label{itm:char-strong-frsb} If $\xi$ is strictly FRSB, then
        % $\xi''(q)^{-1/2}$ is concave on $[0,1]$ and
        $\cQ(\xi) = \int_0^1 \xi''(q)^{1/2} ~\de q$.
    \end{enumerate}
\end{lemma}
\begin{proof}
    If $\xi$ is topologically trivial, then (recalling notation of Lemma~\ref{lem:cs-extremality-0temp}), $\alpha \equiv 0$, and so $\halpha \equiv L$.
    Thus
    \[
        G(q) = \xi'(q) - \fr{q}{L^2}.
    \]
    Since Lemma~\ref{lem:cs-extremality-0temp} gives $G(1)=0$, we have $L^{-2} = \xi'(1)$.
    Because $\min g \ge 0$, we have
    \[
        0 \le g''(1) = - G'(1) = -\xi''(1) + L^{-2},
    \]
    so $\xi'(1) \ge \xi''(1)$.
    Moreover, plugging this $(L,\alpha)$ into \eqref{eq:cs-functional-0temp} shows $\cQ(\xi) = \sqrt{\xi'(1)}$.
    This proves part \ref{itm:char-topologically-trivial}.
    If $\xi$ is strictly FRSB, then $g\equiv 0$, so
    \[
        G'(q) = \xi''(q) - \fr{1}{\halpha(q)^2} = 0.
    \]
    This implies $\halpha(q) = \xi''(q)^{-1/2}$.
    % However $\halpha(q) = L - \int_0^q \alpha(r)~\de r$ is concave, so $\xi''(q)^{-1/2}$ must be concave.
    Plugging this $\halpha$ into \eqref{eq:cs-functional-0temp}
    % gives the desired formula for $\cQ(\xi)$.
    % This
    proves part \ref{itm:char-strong-frsb}.
\end{proof}

\begin{proof}[Proof of Proposition~\ref{prop:opt-topologically-trivial}]
    By Lemma~\ref{lem:type-characterization}\ref{itm:char-topologically-trivial}, we have $\xi'(1) \ge \xi''(1)$.
    Define
    \begin{equation}
        \label{eq:def-GS}
        GS_N \equiv \fr1N \max_{\bsig \in \cS_N} H_N(\bsig).
    \end{equation}
    \cite[Theorem 1.1]{belius2022triviality} shows via the Kac--Rice formula that with high probability, $GS_N \ge \sqrt{\xi'(1)} - \eps/2$.
    By concentration of the ground state energy (see Proposition~\ref{prop:borell-tis} below), $GS_N \ge \sqrt{\xi'(1)} - \eps$ with probability $1-e^{-cN}$.
    Lemma~\ref{lem:type-characterization}\ref{itm:char-topologically-trivial} further implies $\sqrt{\xi'(1)} = \cQ(\xi)$, concluding the proof.
\end{proof}

\begin{proof}[Proof of Proposition~\ref{prop:opt-frsb}]
    We use a randomized version of Subag's Hessian ascent algorithm~\cite{subag2018following} (see also \cite[Section 3.7]{huang2021tight}).
    Starting from $\bx_0=0\in\bbR^N$, we choose small $\eta=\eta(\eps,\delta)\in 1/\bbN$ and for $0\leq j< 1/\eta$ construct $\bx_{j+1}$ from $\bx_j$ as follows.
    For $\bx\in\bbR^N$ let $S^{(\eta)}(\bx)$ be the span of the top $N\eta$ eigenvectors of $\nabla^2 H_N(\bx)|_{\bx^{\perp}}$, with any measurable-in-$\bx$ tie-breaking procedure. (Here we view $\nabla^2 H_N(\bx)$ as a quadratic form and restrict it to the subspace $\bx^{\perp}=\{\by\in\bbR^N~:~\la \bx,\by\ra=0\}$.)
    Let $S_j=S^{(\eta)}(\bx_j)$, and choose $\bv_j\in S_j$ uniformly at random from the corresponding unit sphere, independently of all previous choices.
    Then let $\bx_{j+1}=\bx_j+\bv_j\sqrt{\eta N}$ if $\la \bv_j,\nabla H_N(\bx_j)\ra\geq 0$, else $\bx_{j+1}=\bx_j-\bv_j\sqrt{\eta N}$.
    The output of this algorithm is $\bx_*=\bx_{1/\eta}\in\cS_N$. (Note that $\la\bx_j,\bv_j\ra=0$ by construction, thus $\|\bx_j\|_2^2 = j\eta N$ almost surely.)

    Let $\bx_*,\bx_*'$ be independent outputs of this algorithm for the same $H_N$.
    We claim that with probability $1-e^{-cN}$ over $H_N$ and the randomness inside both runs of the algorithm:
    \begin{align}
    \label{eq:subag-overlap}
    |\la \bx_*,\bx_*'\ra|
    &\leq
    \delta N,
    \\
    \label{eq:subag-energy}
    H_N(\bx_*)/N&\geq  \int_0^1 \xi''(q)^{1/2} ~\de q-\eps
    \stackrel{Lem\,\ref{lem:type-characterization}\ref{itm:char-strong-frsb}}{=}\cQ(\xi)-\eps,
    .
    \end{align}
    This claim implies the desired conclusion by taking $e^{cN/3}$ independent runs of the algorithm.
    The first part \eqref{eq:subag-overlap} holds conditionally on $\bx_*'$: we have $\bbP[|\la \bv_j,\bx_*'\ra|\leq \eta N~|~\bx_*',\bx_j]\geq 1-e^{-cN}$ since $\bv_j$ is uniform on a $\Omega(N)$ dimensional sphere.

    The second part \eqref{eq:subag-energy} is similar to \cite{subag2018following} (see also \cite[Section 3.7]{huang2021tight}) so we give an outline.
    First one notes that $\nabla^2 H_N(\bx)|_{\bx^{\perp}}$ is a GOE matrix scaled by $\sqrt{\big(1-\frac{1}{N}\big)\xi''(\|\bx\|_2^2)}$ for $\bx$ independent of $H_N$ (see e.g. \cite[Eq. (3.10)]{subag2018following}).
    Because of the $N^2$ speed in the large deviation rate function for the bulk spectrum of GOE, combining a $\eta\sqrt{N}$-net of the ball with Proposition~\ref{prop:gradients-bounded} below implies that $\lambda_{N\eta}(\nabla^2 H_N(\bx)|_{\bx^{\perp}})\geq 2\sqrt{\xi''(\|\bx\|^2/N)}-\eps^2$ uniformly in $\|\bx\|\leq \sqrt{N}$ with probability $1-e^{-cN}$.
    Using Proposition~\ref{prop:gradients-bounded} again to control the Taylor approximation error, one finds with the same probability:
    \[
    H_N(\bx_{j+1})-H_N(\bx_j)
    \geq
    \fr12 \cdot \eta N \cdot \lt(2\sqrt{\xi''(\|\bx_j\|^2/N)}-\eps^2\rt),
    \quad\forall 0\leq j<1/\eta.
    \]
    Telescoping gives \eqref{eq:subag-energy}, completing the proof.
    % \mscomment{Our first BOGP paper's randomized version of Subag's algorithm just suffices here. See ``our-first-paper.png'' screenshot. The main proof should just be this (can be short, maybe cite \cite{subag2018following} for the necessary uniform concentration, prove near-orthogonality is exponential high probability for pairs, and just say the energy value works for the same reason as in \cite{subag2018following}).}
    %
    %
    % Follows from \cite[Proposition 3.6]{huang2023optimization}, specialized to the case with $1$ species, $\vec h = (0)$, and $m=1$.
    % The parameters $q_1, \Phi$ defined therein are $q_1=0$, $\Phi(q)=q$.
    % So, $\bbA(\Phi) = \int_0^1 \xi''(q)^{1/2} ~\de q$, which equals $\cQ(\xi)$ by Lemma~\ref{lem:type-characterization}\ref{itm:char-strong-frsb}.
    % Thus the proposition constructs $k$ orthogonal points of energy at least $\bbA(\Phi) - \eps = \cQ(\xi) - \eps$.
\end{proof}

% MS: I think the two remarks below are not particularly necessary, it might just look like we are dying to self-cite
% \begin{remark}
%     Proposition~\ref{prop:opt-topologically-trivial} can also be shown by an explicit gradient-based algorithm as shown in \cite[Section 6]{huang2023strong}. The idea is to first move to an appropriate multiple $\bx_1$ of $\nabla H_N(\bzero)$, and then aim to find the ground state of $H_N$ restricted to $\Band_{1,R(\bx_1,\bx_1),0}(\bx_1)$ (the orthogonal band centered at $\bx_1$). This restriction of $H_N$ turns out to have the conditional law of another topologically trivial spin glass, enabling a recursion that leads to a completely elementary proof.
%     See \cite[Section 5]{sellke2021optimizing}, \cite[Section 2.2]{huang2023optimization} for essentially equivalent approximate message passing formulations that are more convenient for calculation.
% \end{remark}

% \begin{remark}
%     Approximate message passing also suffice to prove Proposition~\ref{prop:opt-frsb}.
%     In particular one may specialize \cite[Proposition 3.6]{huang2023optimization} to the single-species case with parameters $\vec h = (0)$, $m=1$, $q_1=0$, and $\Phi(q)=q$.
%     Then the algorithm therein constructs $e^{cN}$ nearly orthogonal outputs, each with energy approximately $\bbA(\Phi) = \int_0^1 \xi''(q)^{1/2} ~\de q$. This equals $\cQ(\xi)$ by Lemma~\ref{lem:type-characterization}\ref{itm:char-strong-frsb}.
% \end{remark}

\subsection{Preliminary Concentration Estimates}

% \bhcomment{define op norm}
% \mscomment{does this seem OK? I added the $N^{1-\frac{k}{2}}$ rescaling in the Proposition. It seems like we never actually write an operator norm otherwise, so no propagation needed.}

% \mscomment{Defined operator norm (with rescaling, but it seems to never appear directly; we just cite Proposition~\ref{prop:gradients-bounded}). Some other definitions might still need to be added, e.g. $\partial_{\rd},\nabla_{\sph},\nabla^2_{\sph}$}

For a tensor $\bA \in (\bbR^N)^{\otimes k}$, define the operator norm
\[
    \tnorm{\bA}_\op =
    \max_{\|\bsig^1\|_2,\ldots,\|\bsig^k\|_2\leq 1}
    |\la \bA, \bsig^1 \otimes \cdots \otimes \bsig^k \ra|\,.
\]

\begin{proposition}
    \label{prop:gradients-bounded}
    For any model $\xi$, there exists a constant $c=c(\xi)>0$ and sequence of constants $(C_k)_{k\ge 0}$ independent of $N$ such that the following holds.
    Defining the convex set
    \[
    K_N=\lt\{H_N\in\sH_N~:~\norm{\nabla^k H_N(\bsig)}_{\op} \le C_k N^{1-\frac{k}{2}}~~~~\forall k\geq 0,\norm{\bsig}_2 \le \sqrt{N}\rt\}\subseteq \sH_N,
    \]
    \begin{enumerate}[label=(\roman*)]
         \item
         \label{it:gradients-bounded}
         For all $N$, we have $\bbP[H_N \in K_N/2] \ge 1-e^{-cN}$.
         \item
         \label{it:conditional-gradients-bounded}
         More generally, let $\psi:\sH_N\to \bbR^M$ be an almost surely finite linear map. Then $\bbP[H_N \in K_N~:~\psi(H_N)] \ge 1-e^{-cN}$ whenever $\bbE[H_N~:~\psi(H_N)]\in K_N/2$.
     \end{enumerate}
\end{proposition}

\begin{proof}
    Part~\ref{it:gradients-bounded} follows from e.g. \cite[Proposition 2.3]{huang2021tight} (modulo dilation of $K_N$ by a factor of two).
    For part~\ref{it:conditional-gradients-bounded}, note that the conditional law of $H_N-\bbE[H_N~|~\psi(H_N)]$ does not depend on $\psi(H_N)$.
    Moreover, letting $H_N^{(\psi)}\in\sH_N$ be a Hamiltonian with this law, there exists an independent centered Gaussian $H_N^{(\neg\psi)}\in\sH_N$ such that the independent sum $H_N^{(\psi)}+H_N^{(\neg\psi)}$ has the law of $H_N$.
    Then whenever $\bbE[H_N~|~\psi(H_N)]\in K_N/2$, we have
    \begin{align*}
        \bbP[H_N \in K_N~|~\psi(H_N)]
        &\geq
        \bbP[H_N^{(\psi)}\in K_N/2]
        \\
        &\stackrel{(\dagger)}{\geq}
        2\,\bbP[H_N^{(\psi)}+H_N^{(\neg\psi)}\in K_N/2]-1
        \\
        &=
        2\,\bbP[H_N \in K_N/2] -1
        \\
        &\ge 1-2e^{-cN}.
    \end{align*}
    Here $(\dagger)$ follows from the simple observation that by symmetry in law of $H_N^{(\neg\psi)}$ and independence,
    \[
    \bbP\big[H_N^{(\psi)}+H_N^{(\neg\psi)}\in K_N/2~|~H_N^{(\psi)}\notin K_N/2\big]
    \leq
    1/2.
    \qedhere
    \]
\end{proof}

\begin{proposition}
    \label{prop:borell-tis}
    For any model $\xi$, there exists a constant $c=c(\xi)>0$ such that $F_N$ \eqref{eq:def-F} and $GS_N$ \eqref{eq:def-GS} satisfy the concentration inequality
    \[
        \bbP \lt(|F_N - \bbE F_N| \ge t \rt),
        \bbP \lt(\lt|GS_N - \bbE GS_N\rt| \ge t \rt)
        \le 2\exp(-ct^2N).
    \]
\end{proposition}
\begin{proof}
    The concentration inequality for $F_N$ is \cite[Theorem 1.2]{panchenko2013sherrington}, and that for $GS_N$ is by the Borell-TIS inequality \cite{borell1976gaussian, cirelson1976norms}.
\end{proof}

\subsection{Preliminaries on the Kac--Rice Formula}

For each $\bsig \in S_N$, let $\{e_1(\bsig),\ldots,e_N(\bsig)\}$ be an orthonormal basis of $\bbR^N$ with $e_1(\bsig) = \bsig / \sqrt{N}$.
Let $\cT = \{2,\ldots,N\}$.
Let $\nabla_\cT H_N(\bsig) \in \bbR^\cT$ denote the restriction of $\nabla H_N(\bsig) \in \bbR^N$ to the space spanned by $\{e_2(\bsig),\ldots,e_N(\bsig)\}$, and $\nabla^2_{\cT \times \cT} H_N(\bsig) \in \bbR^{\cT \times \cT}$ analogously.
Define the radial and tangential derivatives
\[
    \partial_\rd H_N(\bsig) = \lt\la e_1(\bsig), \nabla H_N(\bsig)\rt\ra, \qquad
    \nabla_\sph H_N(\bsig) = \nabla_{\cT} H_N(\bsig).
\]
Further, define the Riemannian Hessian
\begin{equation}
\label{eq:riemannian-hessian}
    \nabla^2_{\sph} H_N(\bsig) = \nabla^2_{\cT \times \cT} H_N(\bsig) - \fr{1}{\sqrt{N}} \partial_\rd H_N(\bsig) I_{\cT \times \cT}.
\end{equation}
The following lemma describes the joint law of these derivatives for any fixed $\bsig \in S_N$.
\begin{lemma}[{\cite[Lemma 3.2]{belius2022triviality}}]
    \label{lem:derivative-laws}
    For any $\bsig \in S_N$, the random variables $\big(H_N(\bsig), \partial_\rd H_N(\bsig)\big)$ and $\nabla_\sph H_N(\bsig)$ and $\nabla^2_{\cT \times \cT} H_N(\bsig)$ are independent, with the following laws.
    \begin{itemize}
        \item $(H_N(\bsig), \partial_\rd H_N(\bsig))$ is a centered Gaussian with covariance
        \[
            \bbE (H_N(\bsig), \partial_\rd H_N(\bsig))^{\otimes 2}
            = \begin{bmatrix}
                N\xi(1) & \sqrt{N} \xi'(1) \\
                \sqrt{N} \xi'(1) & \xi'(1) + \xi''(1)
            \end{bmatrix}.
        \]
        \item $\nabla_\sph H_N(\bsig)$ is a centered Gaussian with covariance $\xi'(1) I_{N-1}$.
        \item $\nabla^2_{\cT \times \cT} H_N(\bsig)$ is a scaled GOE matrix, with
        \[
            \bbE (\nabla^2_{\cT \times \cT} H_N(\bsig))_{i,j}^2 = \fr{(1+\delta_{i,j}) \xi''(1)}{N},
        \]
        symmetry across the diagonal, and independent entries on and above the diagonal.
    \end{itemize}
\end{lemma}
Say $\bsig \in S_N$ is a \textbf{critical point} of $H_N$ if $\nabla_\sph H_N(\bsig) = \bzero$, and let $\Crt$ denote the set of such points.
Further, for $(\bsig,H_N)$-measurable event $\cE$, let
% Further, for measurable $A \subseteq \bbR^2$, let
\begin{equation}
\label{eq:Crt-def}
    \Crt(\cE) = \lt\{
        \bsig \in \Crt : (\bsig, H_N) \in \cE.
    \rt\}.
\end{equation}
% Let $\Crt(\cE)$ denote the set of critical points such that $(H_N(\bsig), \partial_\rd H_N(\bsig), \nabla_\sph^2 H_N(\bsig)) \in \cE$, for any measurable $\cE$.
The Kac--Rice formula \cite{rice1944mathematical,kac1948average}, applied to $\nabla_{\sph} H_N$, states that
\begin{equation}
    \label{eq:kac-rice-main}
    \bbE |\Crt(\cE)| = \int_{S_N} \bbE \lt[
        |\det \nabla_\sph^2 H_N(\bsig)|
        \ind\lt\{ (\bsig, H_N) \in \cE \rt\}
        \Big|
        \nabla_\sph H_N(\bsig) = \bzero
    \rt]
    \varphi_{\nabla_\sph H_N(\bsig)}(\bzero)
    ~\de \bsig,
\end{equation}
where $\varphi_{X}$ denotes the probability density of the random variable $X$.
In Section~\ref{sec:moment}, we will use specific known consequences of this formula, which hold because conditional on $(H_N(\bsig), \partial_\rd H_N(\bsig))$, the matrix $\nabla_\sph^2 H_N(\bsig)$ appearing in the determinant is a shifted and scaled GOE matrix.

\section{Strictly RS and 1RSB Models via Truncated Second Moment}
\label{sec:moment}

In this section we prove Propositions~\ref{prop:fe-rs} and \ref{prop:opt-1rsb}.
Subsection~\ref{subsec:fe-rs} proves Proposition~\ref{prop:fe-rs}, and the rest of the section is devoted to the proof of Proposition~\ref{prop:opt-1rsb}.

\subsection{Strictly RS models}
\label{subsec:fe-rs}

Let $\xi$ be strictly RS.
Recall from Remark~\ref{rmk:need-no-field} that this implies $\xi'(0)=0$.
Using the notation of Lemma~\ref{lem:cs-extremality}, $x \equiv 1$ so
\[
    \cP(\xi) = \fr12 \xi(1).
\]
Moreover,
\[
    f(q) = \xi(q) + q + \log(1-q)
\]
has unique maximum $q=0$ on $[0,1]$.
Set $\eta > 0$ small depending on $\delta$ such that
\begin{equation}
    \label{eq:f-gap}
    f(q) \le -6\eta \qquad \forall q\in [\delta,1].
\end{equation}
For $\bsig \in \cS_N$ and $q\in [-1,1]$, let $\Band_q(\bsig) = \{\brho : R(\bsig,\brho) = q\}$.
\begin{definition}
    \label{def:fe-typical}
    A point $\bsig$ is \textbf{free energy typical} if for all $|q| \ge \delta$,
    \[
        \Phi(q;\bsig) \equiv \fr1N \log \int_{\Band_q(\bsig)} \exp H_N(\brho) ~\de \brho
        \le \fr12 \xi(1) - \eta.
    \]
    We denote by $\cT_N \subseteq \cS_N$ denote the (random) set of free energy typical points, and
    \[
    \wtZ_N \equiv \int_{\cT_N} \exp H_N(\bsig) \de \bsig.
    \]
\end{definition}

We will prove Proposition~\ref{prop:fe-rs} by computing two moments of $\wtZ_N$.
\begin{proposition}
    \label{prop:rs-restricted-1mt}
    We have $\bbE \wtZ_N \ge (1-e^{-cN}) \exp(N\xi(1)/2)$.
\end{proposition}
\begin{proposition}
    \label{prop:rs-restricted-2mt}
    We have $\bbE \wtZ_N^2 \le \exp(N(\xi(1) + O(\delta)))$.
\end{proposition}
For fixed $\bsig \in \cS_N$, define the \textbf{planted model}
\[
    H_N(\brho) = N \xi(R(\brho,\bsig)) + \hH_N(\brho),
\]
where $\hH_N(\brho)$ is a spin glass with mixture $\xi$.
Let $\bbP^\bsig_\Pl$ denote the law of this $H_N$.

The following crucial lemma uses the interpolation bound Proposition~\ref{prop:RS-interpolation} on a band.
A similar estimate was observed by Alaoui, Montanari, and the second author in \cite[Proposition 3.9]{alaoui2023shattering} to study shattering.
\begin{lemma}
    \label{lem:rs-band-interpolation}
    For fixed $\bsig \in \cS_N$, $|q| \ge \delta$, the following holds.
    With probability $1-e^{-cN}$ over $\bbP^\bsig_\Pl$,
    \[
        \Phi(q;\bsig) \le \fr12 \xi(1) - 2\eta.
    \]
\end{lemma}
\begin{proof}
    % We first compute the conditional law of $H_N$ given $H_N(\bsig) = EN$.
    % We have
    % \[
    %     \bbE [H_N(\brho) | H_N(\bsig)]
    %     = \fr{\bbE H_N(\bsig)H_N(\brho)}{\bbE H_N(\bsig)^2} H_N(\bsig)
    %     = \fr{\xi(R(\bsig,\brho))}{\xi(1)} H_N(\bsig),
    % \]
    % and so the process $\hH_N(\brho) = H_N(\brho) - \bbE [H_N(\brho) | H_N(\bsig)]$ has covariance
    % \[
    %     \bbE \hH_N(\brho^1)\hH_N(\brho^2)
    %     = N \lt(\xi(R(\brho^1,\brho^2)) - \fr{\xi(R(\bsig,\brho^1))\xi(R(\bsig,\brho^2))}{\xi(1)}\rt).
    % \]
    % We conclude that for $\brho \in \Band_q(\bsig)$, conditional on $H_N(\bsig) = EN$,
    % \[
    %     H_N(\brho) \stackrel{d}{=} \fr{\xi(q)EN}{\xi(1)} + \hH_N(\brho).
    % \]
    For $\brho \in \Band_q(\bsig)$, under $\bbP^\bsig_\Pl$ we have
    \[
        H_N(\brho) = \xi(q)N + \hH_N(\brho).
    \]
    Moreover, on $\Band_q(\bsig)$, writing $\brho = q \bsig + \sqrt{1-q^2} \btau$ for $\btau \perp \bsig$, the process $\oH_N(\btau) = \hH_N(\brho)$ has mixture
    \begin{equation}
    \label{eq:txi-def}
        \txi_q(t) = \xi(q^2 + (1-q^2)t).
        % - \fr{\xi(q)^2}{\xi(1)}.
    \end{equation}
    Define the order parameter $\tx(t) = \ind\{t \ge r\}$ where $r = \fr{q}{1+q}$. By Proposition~\ref{prop:RS-interpolation},
    \[
        \Phi(q;\bsig) \le \xi(q)
        % \fr{\xi(q) E}{\xi(1)}
        + \cP(\wt x;\txi_q,0) + \fr12 \log (1-q^2) + o_P(1),
    \]
    where $\fr12 \log (1-q^2)$ accounts for the volume of $\Band_q(\bsig)$ and $o_P(1)$ is a term tending to $0$ in probability (under $\bbP^\bsig_\Pl$).
    Note that $\cP(\wt x;\txi_q,0)$ depends on $\txi_q$ through $\txi'_q$ and $\txi''_q$ only, and these depend on $q$ only through $|q|$.
    Moreover $\xi(q) \le \xi(|q|)$.
    So we may assume without loss of generality that $q>0$.

    We upper bound $\cP(\wt x;\txi_q,0)$ via:
    \begin{align*}
        % 2\cP(\txi_q,0)
        % \le
        2\cP(\tx;\txi_q,0)
        &\le
        \txi'_q(r) (1-r) + \int_r^1 \txi''_q(t)(1-t)~\de t + \fr{r}{1-r} + \int_r^1 \fr{\de t}{1-t} \\
        &= \txi_q(1) - \txi_q(r) + \fr{r}{1-r} + \log(1-r) \\
        &= \xi(1) - \xi(q) + q - \log(1+q).
    \end{align*}
    Thus
    \begin{align*}
        \Phi(q;\bsig)
        &\le \fr12 \lt(\xi(1) + \xi(q) + q + \log (1-q) \rt) + o_P(1) \stackrel{\eqref{eq:f-gap}}{\le}
        \fr12 \xi(1)
        % + \fr{\xi(q)}{\xi(1)} (E - \xi(1))
        - 3\eta + o_P(1).
    \end{align*}
    The conclusion follows from Proposition~\ref{prop:borell-tis}, applied to the free energy $\Phi(q;\bsig)$.
\end{proof}
\begin{lemma}
    \label{lem:typical-whp}
    For any fixed $\bsig \in \cS_N$, $\bbP^\bsig_\Pl[\bsig \in \cT_N] \ge 1-e^{-cN}$.
\end{lemma}
\begin{proof}
    Suppose the event in Lemma~\ref{lem:rs-band-interpolation} holds for all $q \in \{\pm \delta, \pm (\delta + 1/N), \ldots, \pm (\delta + M/N)\}$ for the largest $M$ such that $\delta + M/N \le 1$, and furthermore that $H_N\in K_N$ holds.
    % \bhcomment{grargh I guess we need $K_N$ to hold for the CONDITIONAL random part $\hH$. argue by gaussian domination}
    % \mscomment{should be OK now}
    By Proposition~\ref{prop:gradients-bounded}\ref{it:conditional-gradients-bounded} with $\psi(H_N)=H_N(\bsig)$, this occurs with probability $1-e^{-cN}$ after adjusting $c$.
    Then, for all $q = \pm (\delta + m/N)$,
    \[
        \Phi(q;\bsig) \le \fr12 \xi(1) - 2\eta.
    \]
    % since $E - \xi(1) \le \eta$.
    For all $q \in [\delta + m/N, \delta + (m+1)/N]$, on event $K_N$,
    \[
        \Phi(q;\bsig)
        \le \Phi(\delta + m/N;\bsig) + O(N^{-1})
        \le \fr12 \xi(1) - \eta,
    \]
    and similarly for $q\in [-\delta - (m+1)/N, -\delta - m/N]$.
    This implies $\bsig \in \cT_N$.
\end{proof}

\begin{proof}[Proof of Proposition~\ref{prop:rs-restricted-1mt}]
    For any $\bsig \in \cS_N$,
    \[
        \fr{\bbE \wtZ_N}{\bbE Z_N} = \fr{
            \bbE [\exp(H_N(\bsig)) \ind\{\bsig \in \cT_N\} ]
        }{
            \bbE [\exp(H_N(\bsig))]
        }.
        % = \bbE \lt[
        %     \exp(H_N(\bsig)) \bbP[\bsig \in \cT_N | H_N(\bsig)]
        % \rt].
    \]
    This is the probability of $\bsig \in \cT_N$ under the reweighted measure $\wt \bbP$ with Radon--Nikodym derivative $\wt \bbP / \bbP \propto \exp(H_N(\bsig))$, and by properties of Gaussian conditioning we have precisely $\wt \bbP = \bbP^\bsig_\Pl$.
    Combining with Lemma~\ref{lem:typical-whp} yields
    \[
        \fr{\bbE \wtZ_N}{\bbE Z_N}
        = \bbP^\bsig_\Pl(\bsig \in \cT_N)
        = 1 - e^{-cN}.
    \]
    The result follows because $\bbE Z_N = \exp(N\xi(1)/2)$.
    % By Lemma~\ref{lem:typical-whp}
    % \[
    %     \bbP[\bsig \in \cT_N | H_N(\bsig)] \ge (1-e^{-cN}) \ind\{H_N(\bsig) \le N(\xi(1) + \eta)\}.
    % \]
    % Thus, with $g\sim \cN(N\xi(1), N\xi(1))$, by a Gaussian change of measure,
    % \begin{align*}
    %     \bbE \wtZ_N
    %     &\ge (1-e^{-cN}) \bbE \lt[
    %         \exp(H_N(\bsig)) \ind\{H_N(\bsig) \le N(\xi(1) + \eta)\}
    %     \rt] \\
    %     &= (1-e^{-cN}) \exp(N\xi(1)/2) \bbP[g \le N(\xi(1) + \eta)] \\
    %     &= (1-e^{-cN}) \exp(N\xi(1)/2).
    % \qedhere
    % \end{align*}
\end{proof}
\begin{proof}[Proof of Proposition~\ref{prop:rs-restricted-2mt}]
    Writing $\wtZ_N^2$ as a double integral and recalling $\cT_N$ from Definition~\ref{def:fe-typical}, we have almost surely
    \begin{align*}
        \wtZ_N^2 &= \iint_{\cT_N\times \cT_N} \exp(H_N(\bsig^1) + H_N(\bsig^2)) \de \bsig^1 \de \bsig^2 \\
        &\le \iint_{\cS_N \times \cS_N} \exp(H_N(\bsig^1) + H_N(\bsig^2)) \ind\{R(\bsig^1,\bsig^2) \le \delta\} \de \bsig^1 \de \bsig^2
        + e^{N(\xi(1)/2 - \eta)} \int_{\cS_N} \exp(H_N(\bsig)) ~\de \bsig.
    \end{align*}
    Taking expectations,
    \begin{align*}
        \bbE \wtZ_N^2 &\le
        \int_{-\delta}^{\delta}
        e^{N(\xi(1) + \xi(q))} (1-q^2)^{-(N-1)/2}
        ~\de q
        + e^{N(\xi(1) - \eta)}.
    \end{align*}
    The exponential growth rate of the integral is
    \[
        \xi(1) +
        \max_{q\in [-\delta, \delta]}
        \lt\{
            \xi(q) + \fr12 \log (1-q^2)
        \rt\}
        = \xi(1) + O(\delta)
    \]
    by Lipschitz continuity of the quantity inside the maximum around $0$.
\end{proof}

\begin{proof}[Proof of Proposition~\ref{prop:fe-rs}]
    The statement is clearly monotone in $\delta$, so it suffices to prove it for $\delta$ suitably small in $\eps$.
    By the last two propositions and Paley-Zygmund,
    \begin{equation}
        \label{eq:fe-rs-nontrivial-prob-lb}
        \bbP\lt(\wtZ_N \ge \fr12 \exp(N\xi(1)/2))\rt)
        \ge e^{-O(\delta) N}.
    \end{equation}
    Suppose this event holds.
    % Then we also have $Z_N \ge \fr12 \exp(N(\xi(1)/2))$.
    Let $G$ denote the Gibbs measure of $\wtZ_N$ (i.e. the Gibbs measure of $Z_N$ conditioned on $\bsig \in \cT_N$).
    Then,
    \begin{align*}
        \int_{\Band_{k,1,\delta}(\bzero)}
        \exp \lt(\sum_{i=1}^k H_N(\bsig^i)\rt) ~\de \vbsig
        &\ge
        \int_{\Band_{k,1,\delta}(\bzero) \cap \cT_N^k}
        \exp \lt(\sum_{i=1}^k H_N(\bsig^i)\rt) ~\de \vbsig \\
        &= \wtZ_N^k G^{\otimes k}\lt(|R(\bsig^i,\bsig^j)| \le \delta ~\text{for all}~1\le i<j\le k\rt) \\
        &\ge \wtZ_N^k \lt(1 - \binom{k}{2} G^{\otimes 2}(|R(\bsig^1,\bsig^2)| \ge \delta) \rt).
    \end{align*}
    For any $\bsig^1 \in \cT_N$, $G$-almost surely
    % since $\wtZ_N\leq Z_N$, we have almost surely
    \begin{align*}
        G(|R(\bsig^1,\bsig^2)| \ge \delta | \bsig^1)
        &= \fr{1}{\wtZ_N} \int_{|q| \ge \delta} \int_{\Band_q(\bsig^1)} \exp(H_N(\bsig^2)) ~\de \bsig^2 ~\de q \\
        &\le \fr{\exp(N(\xi(1)/2 - \eta))}{\fr12 \exp(N\xi(1)/2)}
        = 2\exp(-\eta N).
    \end{align*}
    For $k \le e^{\eta N / 3}$,
    \[
        1 - \binom{k}{2} G^{\otimes 2}(|R(\bsig^1,\bsig^2)| \ge \delta)  \ge \fr12.
    \]
    Combining the above and letting
    \[
        Z^{(k)} = \fr{1}{kN} \log
        \int_{\Band_{k,1,\delta}(\bzero)}
        \exp \lt(\sum_{i=1}^k H_N(\bsig^i)\rt) ~\de \vbsig,
    \]
    we conclude that
    \[
        \bbP \lt(
            Z^{(k)} \ge
            \fr{\xi(1)}{2} - o(1)
        \rt)
        \ge e^{-O(\delta) N}.
    \]
    Similarly to \cite[Theorem 1.2]{panchenko2013sherrington} we can show the concentration inequality
    \[
        \bbP(|Z^{(k)} - \bbE Z^{(k)}| \ge t) \le \exp(-ct^2N).
    \]
    (In fact a much stronger inequality is true, see Lemma~\ref{lem:uc} below.)
    The result follows from the last two inequalities for $\delta$ suitably small in $\eps$.
\end{proof}

\subsection{Strictly 1RSB models}

We turn to the proof of Proposition~\ref{prop:opt-1rsb}.
Let $\xi$ be strictly 1RSB, and so $\xi'(0)=0$ by Remark~\ref{rmk:need-no-field}.
For most of the proof we will assume that $\xi$ is \textbf{not} pure; we then appeal to continuity at the end by slightly perturbing any pure $\xi$ (which will preserve the strict 1RSB property).
This simplifies the presentation below as certain formulas degenerate in the pure case, see Remark~\ref{rmk:not-pure}.
We note that some intermediate computations and lemmas resemble those from previous work including \cite{auffinger2013random,auffinger2013complexity,arous2020geometry}, which made different assumptions on $\xi$.

Recalling the notation of Lemma~\ref{lem:cs-extremality-0temp}, there exists $u > 0$ such that $\alpha \equiv u$, and the order parameter $(L,\alpha)$ is described by the pair $(L,u)$.
It can easily be verified that the function $\upsilon : [0,+\infty) \to \bbR$ given by
\[
    \upsilon(z) = \fr{(1+z) \log (1+z)}{z^2} - \fr{1}{z}
\]
is strictly decreasing with $\lim_{z\to 0^+} \upsilon(z) = \fr12$ and $\lim_{z\to\infty} \upsilon(z) = 0$.
\begin{lemma}
\label{lem:1RSB-parameters-setup}
    Assume $\xi$ is strictly 1RSB. Let $z$ be the unique solution to $\upsilon(z) = \fr{\xi(1)}{\xi'(1)}$ and $y = \sqrt{(1+z) \xi'(1)}$. Then
    \begin{align*}
        L &= \fr{1+z}{y}, &
        u &= \fr{z}{y}, &
        \cQ(\xi) &= \fr{\xi'(1) + z\xi(1)}{y},
    \end{align*}
    and for all $q\in [0,1]$,
    \begin{equation}
        \label{eq:1rsb-test}
        \xi(1) - \xi(q) \ge
        \xi'(1) \lt(
            \fr{1+z}{z^2} \log \lt(1 + (1-q)z\rt)
            - \fr{1-q}{z}
        \rt),
    \end{equation}
    with equality at exactly $q=0,1$.
\end{lemma}
\begin{proof}
    As $\halpha(q) = L - uq$, we calculate that
    \begin{align*}
        G(q) &= \xi'(q) - \fr{q}{L(L-uq)}, &
        g(q) &= \xi(1) - \xi(q) - \fr1u \lt(\fr1u \log \fr{L-uq}{L-u} - \fr{1-q}{L}\rt).
    \end{align*}
    Since $G(1) = g(0) = 0$, we get the system of equations
    \begin{align*}
        \xi'(1) &= \fr{1}{L(L-u)}, &
        \xi(1) &= \fr{1}{u^2} \log \fr{L}{L-u} - \fr{1}{Lu}.
    \end{align*}
    Let $z' = \fr{u}{L-u}$, so
    \[
        \fr{\xi(1)}{\xi'(1)}
        = \fr{L(L-u)}{u^2} \log \fr{L}{L-u} - \fr{L-u}{u}
        % = \fr{(1+z') \log (1+z')}{(z')^2} - \fr{1}{z'}
        = \upsilon(z').
    \]
    Thus $z'=z$ by monotonicity of $\upsilon$.
    Then
    \begin{align*}
        L &= \sqrt{\fr{1+z'}{\xi'(1)}} = \fr{1+z}{y}, &
        u &= L\cdot \fr{z'}{1+z'} = \fr{z}{y}.
    \end{align*}
    These formulas and the condition $g(q) \ge 0$ for all $q$ (with equality at precisely $q=0,1$) imply \eqref{eq:1rsb-test}.
    Finally,
    \begin{align*}
        2\cQ(\xi)
        = 2\cQ(L,\alpha;\xi)
        &= \int_0^1 \xi''(q)(L-uq) \de q + \int_0^1 \fr{\de q}{L-uq} \\
        &= (L-u) \xi'(1) + u\xi(1) + \fr{1}{u} \log (1+z).
    \end{align*}
    Note that $(L-u) \xi'(1) = \fr{\xi'(1)}{y}$, $u\xi(1) = \fr{z\xi'(1)}{y}$, and
    \[
        \fr{1}{u} \log(1+z)
        = \fr{z\xi'(1)}{y} \lt( \fr{(1+z)}{z^2} \log (1+z) \rt)
        = \fr{z\xi'(1)}{y} \lt( \fr{1}{z} + \fr{\xi(1)}{\xi'(1)} \rt)
        = \fr{\xi'(1) + z\xi(1)}{y}.
    \]
    This gives the formula for $\cQ(\xi)$.
\end{proof}
We also record a simple inequality in the parameters we will use later.
\begin{lemma}
    \label{lem:y-ineq}
    We have $y^2 \ge \xi''(1)$.
\end{lemma}
\begin{proof}
    Because $\min g(q) = 0$ is attained at $q=1$,
    \[
        0 \le g''(1) = -G'(1) = - \xi''(1) + (L-u)^{-2} = -\xi''(1) + y^2.
        \qedhere
    \]
\end{proof}

Let $\eta_1,\eta_2,\eta_3 > 0$ be small constants to be determined later, where each $\eta_i$ will be set small in terms of $\delta$, $\xi$, and $\{\eta_j : j<i\}$ (i.e. informally, $0<\eta_3\ll\eta_2\ll\eta_1\ll \delta\ll 1$).
By \eqref{eq:1rsb-test}, we may set $\eta_3$ such that
\begin{equation}
    \label{eq:1rsb-gap}
    \xi(1) - \xi(q) \ge
    \xi'(1) \lt(
        \fr{1+z}{z^2} \log \lt(1 + (1-q)z\rt)
        - \fr{1-q}{z}
    \rt) + \fr{6y}{z}\eta_3
\end{equation}
for all $q \in [\delta,1-\eta_2]$.
Set
% $E_0 = \cQ(\xi) = \fr{\xi'(1)+z\xi(1)}{y}$ and $R_0 = y + \fr{\xi''(1)}{y}$, and
\begin{align}
\notag
    E_0 &= \cQ(\xi) = \fr{\xi'(1)+z\xi(1)}{y},
    \qquad \qquad \qquad \qquad R_0 = y + \fr{\xi''(1)}{y},
    \\
    \label{eq:def-B}
    B &= [E_0 - \eta_1 - \eta_3, E_0 - \eta_1 + \eta_3] \times
    [R_0 + \eta_1^{3/4} - \eta_3, R_0 + \eta_1^{3/4} + \eta_3].
\end{align}
By slight abuse of notation, for $A \subseteq \bbR^2$, let
\begin{equation}
    \label{eq:crt-pts-A}
    \Crt(A) = \lt\{
        \bsig \in \Crt : \lt(\fr1N H_N(\bsig), \fr{1}{\sqrt{N}} \partial_{\rd} H_N(\bsig)\rt) \in A
    \rt\}.
\end{equation}
\begin{definition}
    \label{def:gs-typical}
    A point $\bsig \in \Crt(B)$ is \textbf{ground state typical} if the following conditions hold.
    \begin{enumerate}[label=(\roman*)]
        \item \label{itm:1rsb-bands-good} For all $q\in [\delta,1-\eta_2]$,
        \[
            \Psi(q;\bsig) \equiv \fr1N \sup_{\brho \in \Band_q(\bsig)} H_N(\brho) \le E_0 - \eta_1 - \eta_3.
        \]
        \item \label{itm:1rsb-well} $H_N$ does not have any critical points $\brho$ with
        % $|R(\bsig,\brho)| \ge 1 - \eta_2$.
        $R(\bsig,\brho) \ge 1 - \eta_2$.
        % \mscomment{removed absolute value, this is fine right?}
        % \bhcomment{yeah}
    \end{enumerate}
    Denote the set of such points by $\wtCrt(B)$.
\end{definition}
We will prove Proposition~\ref{prop:opt-1rsb} via the next two propositions, whose proofs comprise the rest of the section.

\begin{proposition}
    \label{prop:1rsb-restricted-1mt}
    We have $\bbE |\wtCrt(B)| \ge e^{\Omega(\eta_1) N}$.
\end{proposition}
\begin{proposition}
    \label{prop:1rsb-restricted-2mt}
    We have $\bbE |\wtCrt(B)|^2 \le e^{O(\delta) N}$.
\end{proposition}

\begin{remark}
    The choice \eqref{eq:def-B} of $B$ looks strange at first, because when $\xi$ is a pure model $H_N(\bsig)$ and $\partial_{\rd} H_N(\bsig)$ are almost surely proportional, so there are a.s. no critical points with energy and radial derivative described by $B$.
    This is not a problem for the proof because the parameters $\eta_1,\eta_2,\eta_3$ can be taken small in $\xi$, and then the statement of Proposition~\ref{prop:opt-1rsb} is continuous in $\xi$; see the end of the proof of Proposition~\ref{prop:opt-1rsb}.
    % \bhcomment{These parameters are chosen so that (1) the Kac-Rice first moment is slightly positive and (2) all band OPTs are less than $E_0 - \eta_1 - 2\eta_3$.}
    % \bhcomment{It's a bit bizarre because these parameters make no sense for pure models. The reason is, the OPT critical points of pure models are already wells so there's no need to get welliness from perturbing $R$.}
\end{remark}

\subsection{Ground State Typicality is With High Probability}

The main result of this subsection is the following proposition.
In it, we fix $\bsig\in\cS_N$ and condition on the event $\{\bsig \in \Crt(B)\}$. Here when conditioning, we refer to the standard regular conditional probability given $(H_N(\bsig),\partial_{\rd} H_N(\bsig))$, which is a linear function of $H_N$.
% \mscomment{Don't we have to bundle the truncation event into the Kac-Rice integral, technically? Similarly to the update for the multi-species paper (can even make something citable?).}
\begin{proposition}
    \label{prop:1rsb-typical-whp}
    If $(E,R)\in B$, then
    \[
    \bbP\big[\bsig \in \wtCrt(B)~|~\big(H_N(\bsig),\partial_{\rd}H_N(\bsig),\nabla_{\sph}H_N(\bsig)\big)=(EN,R\sqrt{N},\bzero)\big]\geq 1-e^{-cN}.
    \]
    % Conditional on $\bsig \in \Crt(B)$, $\bsig \in \wtCrt(B)$ with probability $1-e^{-cN}$.
\end{proposition}
% In this subsection we prove the following proposition, which will imply that critical points counted by $\Crt(B_\iota)$ are ground state typical with high probability.

We will prove Proposition~\ref{prop:1rsb-typical-whp} by studying the ground state energy on bands defined by their overlap with $\bsig$, analogously to the replica-symmetric case.

\begin{lemma}
    \label{lem:1rsb-conditional-law}
    Conditional on $(H_N(\bsig),\partial_{\rd}H_N(\bsig),\nabla_{\sph}H_N(\bsig))=(EN,R\sqrt{N},\bzero)$, the restriction of $H_N$ to $\Band_q(\bsig)$ has law
    \begin{equation}
        \label{eq:1rsb-band-expectation}
        H_N(\brho) \stackrel{d}{=}
        \fr{N(q\xi'(q) + z\xi(q))}{y}
        + N \lt\la
            v^q,
            \begin{bmatrix}
                E - E_0 \\
                R - R_0
            \end{bmatrix}
        \rt\ra
        + \hH_N(\brho),
    \end{equation}
    where
    \begin{equation}
        \label{eq:def-vq}
        v^q = \begin{bmatrix}
            v^q_E \\ v^q_R
        \end{bmatrix}
        = \begin{bmatrix}
            \xi(1) & \xi'(1) \\
            \xi'(1) & \xi'(1) + \xi''(1)
        \end{bmatrix}^{-1}
        \begin{bmatrix}
            \xi(q) \\
            q\xi'(q)
        \end{bmatrix}
    \end{equation}
    and $\hH_N$ is a $(N-1)$-dimensional spin glass with the following covariance.
    Write $\brho = q \bsig + \sqrt{1-q^2} \btau$ and let $\oH_N(\btau) = \hH_N(\brho)$.
    Then $\oH_N$ has mixture
    \begin{equation}
        \label{eq:1rsb-band-covariance}
        \txi_q(t) = \xi(q^2 + (1-q^2)t) - C - \fr{(1-q^2)\xi'(q)^2}{\xi'(1)}t.
    \end{equation}
    for some constant $C$ (which will be irrelevant for our purposes).
\end{lemma}
\begin{remark}
    \label{rmk:not-pure}
    Note that the matrix in \eqref{eq:def-vq} has determinant
    \[
        \xi(1) (\xi'(1)+\xi''(1)) - \xi'(1)^2
        = \lt(\sum_{p\ge 2} \gamma_p^2\rt)\lt(\sum_{p\ge 2} p^2 \gamma_p^2\rt)
        - \lt(\sum_{p\ge 2} p\gamma_p^2\rt)^2.
    \]
    This is nonnegative by Cauchy-Schwarz and strictly positive when $\xi$ is not pure, so the matrix inverse is well-defined.
    This is one reason we assume $\xi$ is not pure.
\end{remark}
\begin{proof}
    By Lemma~\ref{lem:derivative-laws}, $H_N(\bsig)$, $\partial_{\rd} H_N(\bsig)$, and $\nabla_\sph H_N(\bsig)$ are jointly Gaussian with covariance matrix
    \[
        \begin{bmatrix}
            N\xi(1) & \sqrt{N}\xi'(1) & \bzero^\top \\
            \sqrt{N}\xi'(1) & \xi'(1) + \xi''(1) & \bzero^\top \\
            \bzero & \bzero & \xi'(1) I_{N-1}
        \end{bmatrix}.
    \]
    For $\brho \in \Band_q(\bsig)$, we further have
    \begin{align*}
        \bbE H_N(\brho)H_N(\bsig) &= N\xi(q), \\
        \bbE H_N(\brho)\partial_{\rd} H_N(\bsig) &= \sqrt{N} q\xi'(q), \\
        \bbE H_N(\brho)\nabla_\sph H_N(\bsig) &= \xi'(q) P_\bsig^\perp \brho.
    \end{align*}
    Thus
    \begin{align}
        \notag
        \bbE[H_N(\brho) | H_N(\bsig), \partial_{\rd} H_N(\bsig), \nabla_\sph H_N(\bsig)]
        &= \lt\la
            \begin{bmatrix}
                \xi(1) & \xi'(1) \\
                \xi'(1) & \xi'(1) + \xi''(1)
            \end{bmatrix}^{-1}
            \begin{bmatrix}
                \xi(q) \\
                q\xi'(q)
            \end{bmatrix},
            \begin{bmatrix}
                H_N(\bsig) \\
                \sqrt{N} \partial_{\rd} H_N(\bsig)
            \end{bmatrix}
        \rt\ra \\
        \label{eq:1rsb-band-expectation-intermediate}
        &\qquad + \fr{\xi'(q)}{\xi'(1)} \la P_\bsig^\perp \brho, \nabla_\sph H_N(\bsig) \ra.
    \end{align}
    Then $\hH_N(\brho) = H_N(\brho) - \bbE[H_N(\brho) | H_N(\brho), \partial_{\rd} H_N(\brho), \nabla_\sph H_N(\brho)]$ has covariance
    \begin{align*}
        \fr1N \bbE \hH_N(\brho^1) \hH_N(\brho^2)
        &=
        \xi(R(\brho^1,\brho^2))
        - \lt\la
            \begin{bmatrix}
                \xi(1) & \xi'(1) \\
                \xi'(1) & \xi'(1) + \xi''(1)
            \end{bmatrix}^{-1}
            \begin{bmatrix}
                \xi(q) \\
                q\xi'(q)
            \end{bmatrix},
            \begin{bmatrix}
                \xi(q) \\
                q\xi'(q)
            \end{bmatrix}
        \rt\ra \\
        &\qquad - \fr{\xi'(q)^2}{\xi'(1)} R(P_\bsig^\perp \brho^1, P_\bsig^\perp \brho^2).
    \end{align*}
    This proves \eqref{eq:1rsb-band-covariance}.
    The conclusion \eqref{eq:1rsb-band-expectation} follows from \eqref{eq:1rsb-band-expectation-intermediate}, by noting that
    \begin{align}
        \notag
        &\begin{bmatrix}
            \xi(1) & \xi'(1) \\
            \xi'(1) & \xi'(1) + \xi''(1)
        \end{bmatrix}^{-1}
        \begin{bmatrix}
            E_0 \\
            R_0
        \end{bmatrix} \\
        \label{eq:E0R0-identity}
        &= \fr{1}{y(\xi(1)(\xi'(1) + \xi''(1)) - \xi'(1)^2)}
        \begin{bmatrix}
            \xi'(1) + \xi''(1) & -\xi'(1) \\
            -\xi'(1) & \xi(1)
        \end{bmatrix}
        \begin{bmatrix}
            z\xi(1) + \xi'(1) \\
            (1+z)\xi'(1) + \xi''(1)
        \end{bmatrix}
        = \fr{1}{y}
        \begin{bmatrix}
            z \\ 1
        \end{bmatrix},
    \end{align}
    and thus
    % \begin{equation}
        % \label{eq:E0R0-identity}
    \[
        \lt\la
            \begin{bmatrix}
                \xi(1) & \xi'(1) \\
                \xi'(1) & \xi'(1) + \xi''(1)
            \end{bmatrix}^{-1}
            \begin{bmatrix}
                \xi(q) \\
                q\xi'(q)
            \end{bmatrix},
            \begin{bmatrix}
                E_0 \\
                R_0
            \end{bmatrix}
        \rt\ra
        = \fr{q\xi'(q) + z\xi(q)}{y}.
    \qedhere
    \]
    % \end{equation}
\end{proof}

The next estimate will prepare us to apply Proposition~\ref{prop:1RSB-interpolation}. We will use the order parameters:
\begin{equation}
\label{eq:1RSB-band-interpolation-parameters}
\begin{aligned}
    \wt L &= \fr{1+(1-q)z}{(1-q^2)y}, &
    \wt \alpha(t) &= u \ind\{t\ge r\}, &
    r = \fr{q}{1+q}.
\end{aligned}
\end{equation}

\begin{proposition}
    \label{prop:1rsb-band-interpolation-bash}
    For fixed $\bsig \in \cS_N$, $E,R \in \bbR$, $q \in [\delta,1-\eta_2]$,
    \[
        \fr{q\xi'(q) + z\xi(q)}{y} + \cQ(\wt L,\wt\alpha;\txi_q,0)
        \le
        E_0 - 3\eta_3.
    \]
\end{proposition}
\begin{proof}
    By direct computation,
    \begin{equation}
        \label{eq:1rsb-cs-intermediate}
        2\cQ(\wt L, \wt \alpha; \txi_q,0)
        = \txi'_q(r) \wt L + \int_r^1 \txi''_q(t) \lt(\wt L - (t-r)u\rt) ~\de t
        + \fr{r}{\wt L} + \int_r^1 \fr{\de t}{\wt L - (t-r)u}.
    \end{equation}
    The first two terms on the right-hand side simplify as
    \begin{align*}
        &\txi'_q(r) \wt L + \int_r^1 \txi''_q(t) \lt(\wt L - (t-r)u\rt) ~\de t \\
        &= \txi'_q(1) \wt L + (\txi_q(1) - \txi_q(r) - (1-r)\txi'_q(1)) u \\
        &= (1-q^2) \lt(\xi'(1) - \fr{\xi'(q)^2}{\xi'(1)} \rt) \wt L
        + \lt(\xi(1) - \xi(q) - (1-q) \xi'(1)\rt) u \\
        &= \fr{1}{y\xi'(1)} \lt\{
            \lt(\xi'(1)^2 - \xi'(q)^2) (1 + (1-q)z\rt)
            + \lt(\xi(1) - \xi(q) - (1-q) \xi'(1)\rt) z\xi'(1)
        \rt\} \\
        &= \fr{1}{y\xi'(1)} \lt\{
            \xi'(1)^2 + z\xi'(1) (\xi(1) - \xi(q)) - \xi'(q)^2 (1 + (1-q)z)
        \rt\}.
    \end{align*}
    The last two terms of \eqref{eq:1rsb-cs-intermediate} simplify as
    \begin{align*}
        \fr{r}{\wt L}
        &= \fr{q(1-q) y}{1+(1-q)z}
        = \fr{q(1-q)(1+z)\xi'(1)}{y(1+(1-q)z)}, \\
        \int_r^1 \fr{\de t}{\wt L - (t-r)u}
        &= \fr{1}{u} \log \fr{\wt L}{\wt L - (1-r) u}
        = \fr{1}{u} \log (1 + (1-q)z) \\
        &\stackrel{\eqref{eq:1rsb-gap}}{\le}
        \fr{1}{y} \lt(z(\xi(1) - \xi(q)) + (1-q)\xi'(1) \rt) - 6\eta_3.
    \end{align*}
    It thus suffices to show
    \begin{align*}
        \fr{q\xi'(q) + {{\color{red} z\xi(q)}}}{y}
        &+ \fr{1}{2y\xi'(1)} \lt\{
            {{\color{red} \xi'(1)^2 + z\xi'(1) (\xi(1) - \xi(q))}} - \xi'(q)^2 (1 + (1-q)z)
        \rt\} \\
        &+ \fr{q(1-q)(1+z)\xi'(1)}{2y(1+(1-q)z)}
        + \fr{1}{2y} \lt({{\color{red} z(\xi(1) - \xi(q))}} + ({{\color{red} 1}}-q)\xi'(1) \rt)
        \le
        {{\color{red}\fr{ \xi'(1) + z\xi(1)}{y}}}.
    \end{align*}
    The terms in red cancel.
    Also clearing a factor of $y$, it remains to show
    \[
        q\xi'(q)
        - \fr{1 + (1-q)z}{2} \cdot \fr{\xi'(q)^2}{\xi'(1)}
        - \fr12 \lt(q - \fr{q(1-q)(1+z)}{1+(1-q)z}\rt) \xi'(1)
        \le
        0.
    \]
    Since $q - \fr{q(1-q)(1+z)}{1+(1-q)z} = \fr{q^2}{1+(1-q)z}$, the desired inequality reduces to the trivial
    \[
        -\fr{1}{2\xi'(1)(1+(1-q)z)} \lt(
            q\xi'(1) - (1+(1-q)z)\xi'(q)
        \rt)^2 \le 0
    .
    \qedhere
    \]
    % which is trivial.
\end{proof}

\begin{proposition}
    \label{prop:1rsb-band-interpolation}
    For fixed $\bsig \in \cS_N$, $E,R \in \bbR$, $q \in [\delta,1-\eta_2]$, the following holds.
    With probability $1-e^{-cN}$ conditionally on $(H_N(\bsig),\partial_{\rd}H_N(\bsig),\nabla_{\sph}H_N(\bsig))=(EN,R\sqrt{N},\bzero)$:
    \[
        \Psi(q;\bsig) \le E_0
        + \lt\la
            v^q,
            \begin{bmatrix}
                E-E_0 \\
                R-R_0
            \end{bmatrix}
        \rt\ra
        - 2\eta_3.
    \]
\end{proposition}
\begin{proof} %[Proof of Proposition~\ref{prop:1rsb-band-interpolation}]
    Lemma~\ref{lem:1rsb-conditional-law}, Proposition~\ref{prop:1rsb-band-interpolation-bash}, and Proposition~\ref{prop:1RSB-interpolation} imply
    \[
        \Psi(q;\bsig) \le E_0 + \lt\la
            v^q,
            \begin{bmatrix}
                E - E_0 \\
                R - R_0
            \end{bmatrix}
        \rt\ra - 3\eta_3 + o_P(1).
    \]
    The result follows by Proposition~\ref{prop:borell-tis}, applied to the ground state energy $\Psi(q;\bsig)$.
\end{proof}

\begin{lemma}
    \label{lem:vq-bds}
    There exist constants $c_1,c_2>0$ depending only on $\delta$ such that for all $q\in [\delta,1]$,
    \[
        v^q_E \ge 1 - c_1 (1-q)^2, \qquad
        v^q_R \le -c_2 (1-q).
    \]
\end{lemma}
\begin{proof}
    We have that
    \begin{align*}
        1 - v^q_E &= \fr{(\xi(1) - \xi(q))(\xi'(1)+\xi''(1)) - (\xi'(1) - q\xi'(q))\xi'(1)}{\xi(1)(\xi'(1)+\xi''(1)) - \xi'(1)^2}, \\
        -v^q_R &= \fr{q\xi'(q)\xi(1) - \xi(q)\xi'(1)}{\xi(1)(\xi'(1)+\xi''(1)) - \xi'(1)^2}.
    \end{align*}
    Note that
    \begin{align*}
        q\xi'(q)\xi(1) - \xi(q)\xi'(1)
        &= \lt(\sum_{p\ge 2} p\gamma_p^2 q^p\rt)\lt(\sum_{p\ge 2} \gamma_p^2\rt)
        - \lt(\sum_{p\ge 2} \gamma_p^2 q^p\rt)\lt(\sum_{p\ge 2} p\gamma_p^2\rt) \\
        &= \sum_{p > p'\ge 2}
        \gamma_p^2 \gamma_{p'}^2
        (p-p')(q^{p'} - q^p) \\
        &= (1-q) \sum_{p > p'\ge 2}
        \gamma_p^2 \gamma_{p'}^2
        q^{p'}(p-p')(1 + q + \cdots + q^{p-p'-1}).
    \end{align*}
    The sum is positive and uniformly bounded away from $0$ for $q\in [\delta,1]$.
     % which implies the reuslt for $v^q_2$.
    Similarly
    \begin{align*}
        &(\xi(1) - \xi(q))(\xi'(1)+\xi''(1)) - (\xi'(1) - q\xi'(q))\xi'(1) \\
        &= \lt(\sum_{p\ge 2} \gamma_p^2 (1-q^p)\rt)\lt(\sum_{p\ge 2} p^2\gamma_p^2\rt)
        - \lt(\sum_{p\ge 2} p\gamma_p^2 (1-q^p)\rt)\lt(\sum_{p\ge 2} p\gamma_p^2\rt) \\
        &= \sum_{p>p'\ge 2}
        \gamma_p^2 \gamma_{p'}^2
        (p-p') \lt((1-q^{p'})p - (1-q^p)p'\rt) \\
        &= (1-q)^2
        \sum_{p>p'\ge 2}
        \gamma_p^2 \gamma_{p'}^2
        (p-p') \lt(
            \sum_{r=0}^{p'-1}
            (r+1)(p-p') q^r
            + \sum_{r=p'}^{p-2}
            p' (p-1-r) q^r
        \rt)
    \end{align*}
    and the sum is uniformly bounded above for $q\in [\delta,1]$.
    % \bhcomment{Todo: for $q$ near $1$, Taylor expansion. Otherwise dumby.}
\end{proof}

\begin{proof}[Proof of Proposition~\ref{prop:1rsb-typical-whp}]
    Consider $\bsig \in \Crt(B)$ with $H_N(\bsig) = EN$, $\partial_{\rd} H_N(\bsig) = R\sqrt{N}$, so $(E,R) \in B$.
    We will show both conditions \ref{itm:1rsb-bands-good} and \ref{itm:1rsb-well} hold with conditional probability $1-e^{-cN}$.

    We begin with condition \ref{itm:1rsb-bands-good}, considering a fixed $q\in [\delta,1-\eta_2]$.
    Let $E = E_0 - \eta_1 + \iota_1$, $R = R_0 + \eta_1^{3/4} + \iota_2$, where $|\iota_1|,|\iota_2| \le \eta_3$.
    We will show that with probability $1-e^{-cN}$,
    \begin{equation}
        \label{eq:1rsb-typical-wtp-one-q}
        \Psi(q;\bsig) \le E_0 - \eta_1 - 2\eta_3.
    \end{equation}
    By Proposition~\ref{prop:1rsb-band-interpolation}, it suffices to show
    \begin{equation}
        \label{eq:1rsb-one-band-ineq}
        \lt\la
            v^q,
            \begin{bmatrix}
                -\eta_1 \\
                \eta_1^{3/4}
            \end{bmatrix}
        \rt\ra
        +
        \lt\la
            v^q,
            \begin{bmatrix}
                \iota_1 \\
                \iota_2
            \end{bmatrix}
        \rt\ra
        \le -\eta_1.
    \end{equation}
    By Lemma~\ref{lem:vq-bds},
    \[
        \lt\la
            v^q,
            \begin{bmatrix}
                -\eta_1 \\
                \eta_1^{3/4}
            \end{bmatrix}
        \rt\ra
        \le
        -\eta_1 + c_1 (1-q)^2 \eta_1 - c_2 (1-q) \eta_1^{3/4}.
    \]
    Setting $\eta_1$ small enough, we can ensure that $c_1 (1-q)^2 \eta_1 \le \fr12 c_2 (1-q) \eta_1^{3/4}$.
    Since $\eta_3$ can be taken small in $\eta_1$, this proves \eqref{eq:1rsb-one-band-ineq}, and \eqref{eq:1rsb-typical-wtp-one-q} follows.

    Suppose the event \eqref{eq:1rsb-typical-wtp-one-q} holds for $q \in \{\delta, \delta + 1/N, \ldots, \delta + M/N\}$ for the largest $M$ such that $\delta + M/N \le 1$, and that $H_N\in K_N$.
    This occurs with probability $1-e^{-cN}$ by Proposition~\ref{prop:gradients-bounded}\ref{it:conditional-gradients-bounded}.
    On $K_N$, for all $q\in [\delta + m/N, \delta + (m+1)/N]$,
    \[
        \Psi(q;\bsig) \le \Psi(\delta + m/N;\bsig) + O(N^{-1}) \le E_0 - \eta_1 - 2\eta_3.
    \]
    Thus part \ref{itm:1rsb-bands-good} holds.

    % \bhcomment{Condition \ref{itm:1rsb-well}: $\eta_1^{2/3}$-well will not have another crit within $\eta_2$, provided $\eta_2 \ll \eta_1^{2/3}$}
    % \bhcomment{Mark can you fill in the Riemannian geometry formalism?}
    % \mscomment{I think the below suffices, hope the notation is understandable}

    To verify condition~\ref{itm:1rsb-well}, we will argue that with high (conditional) probability, $\bsig$ is a ``well'' for $H_N$.
    Let $\theta:[0,\infty)\to\cS_N$ be an arbitrary unit-speed geodesic on $\cS_N$ with $\theta(0)=\bsig$, and consider the function $f(t)=H_N(\theta(t\sqrt{N}))/N$.
    We have $H_N\in K_N$ with conditional probability $1-e^{-cN}$, and on this event the $C^3$ norm of $f$ is bounded independently of $N$.
    Moreover $f'(0)=0$ since $\partial_{\rd}H_N(\bsig)=0$, while
    \[
    f''(0)
    =
    \la \theta'(0),\nabla_{\sph}^2 H_N(\bsig)\theta'(0)\ra
    \leq
    \lambda_{\max}(\nabla_{\sph}^2 H_N(\bsig))
    .
    \]
    Given $(H_N(\bsig),\partial_{\rd} H_N(\bsig))$, the conditional law of $\nabla_{\sph}^2 H_N(\sig)$ is (see e.g. \cite[Lemma 2.1]{huang2023strong}):
    \[
    \sqrt{\xi''(1)\cdot \lt(1-\frac{1}{N}\rt)}\cdot GOE(N-1)-\partial_{\rd}H_N(\bsig)\cdot I_{N-1}.
    \]
    Hence $\lambda_{\max}(\nabla_{\sph}^2 H_N(\sig))\leq (2\sqrt{\xi''(1)}+\eta_3)-\partial_{\rd}H_N(\bsig)$ has conditional probability $1-e^{-cN}$.
    We will prove condition~\ref{itm:1rsb-well} whenever this inequality and $H_N\in K_N$ both hold.
    By definition of $B$, we then have:
    \[
    f''(0)
    \leq
    (2\sqrt{\xi''(1)}+\eta_3)-\partial_{\rd}H_N(\bsig)
    \leq
    2\sqrt{\xi''(1)}-R_0-\eta_1^{3/4}+2\eta_3
    \leq
    -\frac{\eta_1^{3/4}}{2}.
    \]
    The final bound follows because $R_0=y+\frac{\xi''(1)}{y}\geq 2\sqrt{\xi''(1)}$ by AM-GM, while $\eta_3$ is sufficiently small depending on $\eta_1$.
    Recalling that $f'(0)=0$ and $f$ has bounded $C^3$ norm, it follows that $f'(t)\neq 0$ for all $t\leq o(\eta_1^{3/4})$.
    Since $\theta'(0)$ was arbitrary, we conclude that $H_N$ has no other critical point within distance $o(\eta_1^{3/4}\sqrt{N})$ of $\bsig$.
    In particular for small enough $\eta_2$ depending on $\eta_1$, $H_N$ has no critical point $\bsig'$ with $R(\bsig,\bsig')\geq 1-\eta_2$.
    This completes the proof.
\end{proof}

We will actually use Proposition~\ref{prop:1rsb-typical-whp} via the following natural corollary.

\begin{corollary}
\label{cor:1rsb-typical-whp}
    For any $\bsig\in\cS_N$ and $(E,R)\in B$:
    \begin{align*}
    &\bbE\Big[|\det \nabla^2_{\sph}H_N(\bsig)|\cdot \ind\{\bsig\notin \wtCrt(B)\}~\big|~\big(H_N(\bsig),\partial_{\rd}H_N(\bsig),\nabla_{\sph}H_N(\bsig)\big)=(EN,R\sqrt{N},\bzero)\Big]
    \\
    &\leq
    e^{-cN/3}
    \bbE\Big[|\det \nabla^2_{\sph}H_N(\bsig)|~\big|~\big(H_N(\bsig),\partial_{\rd}H_N(\bsig),\nabla_{\sph}H_N(\bsig)\big)=(EN,R\sqrt{N},\bzero)\Big].
    \end{align*}
\end{corollary}

\begin{proof}
    Note that the conditional law of $\nabla^2_{\sph}H_N(\bsig)$ that of a $GOE(N-1)$ matrix scaled by $\sqrt{\frac{N-1}{N}}$ and shifted by $R\cdot I_{N-1}$.
    Recalling the notation \eqref{eq:def-kappa}, \cite[Theorem A.2]{ben2023exponential} implies:
    \begin{align*}
    \frac{1}{N}\log\bbE\Big[|\det \nabla^2_{\sph}H_N(\bsig)|~\big|~\big(H_N(\bsig),\partial_{\rd}H_N(\bsig),\nabla_{\sph}H_N(\bsig)\big)=(EN,R\sqrt{N},\bzero)\Big]
    &=
    \kappa(R)\pm o_N(1),
    \\
    \frac{1}{N}\log\bbE\Big[|\det \nabla^2_{\sph}H_N(\bsig)|^2~\big|~\big(H_N(\bsig),\partial_{\rd}H_N(\bsig),\nabla_{\sph}H_N(\bsig)\big)=(EN,R\sqrt{N},\bzero)\Big]
    &=
    2\kappa(R)\pm o_N(1).
    \end{align*}
    (See the end of \cite[Proof of Proposition 3.1]{huang2023strong} for further details.)
    By conditional Cauchy--Schwarz,
    \begin{align*}
    &\bbE\Big[|\det \nabla^2_{\sph}H_N(\bsig)|\cdot \ind\{\bsig\notin \wtCrt(B)\}~\big|~\big(H_N(\bsig),\partial_{\rd}H_N(\bsig),\nabla_{\sph}H_N(\bsig)\big)=(EN,R\sqrt{N},\bzero)\Big]
    \\
    &\leq \bbE\Big[|\det \nabla^2_{\sph}H_N(\bsig)|^2~\big|~\big(H_N(\bsig),\partial_{\rd}H_N(\bsig),\nabla_{\sph}H_N(\bsig)\big)=(EN,R\sqrt{N},\bzero)\Big]^{1/2}
    \\
    &
    \quad\quad\times
    \bbP\Big[ \ind\{\bsig\notin \wtCrt(B)\}~\big|~\big(H_N(\bsig),\partial_{\rd}H_N(\bsig),\nabla_{\sph}H_N(\bsig)\big)=(EN,R\sqrt{N},\bzero)\Big]^{1/2}.
    \end{align*}
    Applying Proposition~\ref{prop:1rsb-typical-whp} to the last term gives the claimed estimate.
\end{proof}

\subsection{Truncated Moments of Critical Point Count via Kac-Rice}

We will need the following critical point count formulas from \cite{arous2020geometry}.
Let
\begin{align*}
    \Sigma &=
    \begin{bmatrix}
        \xi(1) & \xi'(1) \\
        \xi'(1) & \xi'(1) + \xi''(1)
    \end{bmatrix}, &
    \Sigma_q &=
    \begin{bmatrix}
        \xi(1) & \xi(q) & \xi'(1) & q\xi'(q) \\
        \xi(q) & \xi(1) & q\xi'(q) & \xi'(1) \\
        \xi'(1) & q\xi'(q) & \xi'(1) + \xi''(1) & q\xi'(q) + q^2\xi''(q) \\
        q\xi'(q) & \xi'(1) & q\xi'(q) + q^2\xi''(q) & \xi'(1) + \xi''(1)
    \end{bmatrix}
\end{align*}
be the covariances of $(\fr{1}{\sqrt{N}} H_N(\bsig), \partial_{\rd} H_N(\bsig))$ and $(\fr{1}{\sqrt{N}} H_N(\bsig), \fr{1}{\sqrt{N}} H_N(\brho), \partial_{\rd} H_N(\bsig), \partial_{\rd} H_N(\brho))$, where $R(\bsig,\brho) = q$.
Let
\[
    \rho(\de \lambda) = \fr{1}{2\pi} \sqrt{4-\lambda^2} \ind\{|\lambda|\le 2\} ~\de \lambda
\]
be the semicircle measure, and define
\begin{align}
\label{eq:def-kappa}
    \kappa(x) &= \int_{\bbR} \log |\lambda - x| \rho(\de \lambda) \\
\notag
    &= \fr14 x^2 - \fr12 - \ind\{|x|>2\} \lt(
        \fr14 |x| \sqrt{x^2-4} - \log \lt(\fr{\sqrt{x^2-4} + |x|}{2}\rt)
    \rt) \\
\label{eq:def-Theta}
    \Theta(E,R) &= \fr12 + \fr12 \log \fr{\xi''(1)}{\xi'(1)} - \fr12 \lt\la (E,R), \Sigma^{-1} (E,R) \rt\ra + \kappa\lt(R/\sqrt{\xi''(1)}\rt) \\
\notag
    \Xi(q,E_1,E_2,R_1,R_2) &= 1 + \fr12 \log \fr{(1-q^2) \xi''(1)^2}{\xi'(1)^2 - \xi'(q)^2}
    - \fr12 \lt\la (E_1,E_2,R_1,R_2), \Sigma_q^{-1} (E_1,E_2,R_1,R_2) \rt\ra \\
\notag
    &\qquad + \kappa\lt(R_1/\sqrt{\xi''(1)}\rt) + \kappa\lt(R_2/\sqrt{\xi''(1)}\rt).
\end{align}
Similarly to \eqref{eq:crt-pts-A}, for $A \subseteq [-1,1] \times \bbR^4$ let
\[
    \Crt_2(A) = \lt\{
        (\bsig,\brho) \in \Crt^2 : \lt(R(\bsig,\brho), \fr1N H_N(\bsig), \fr1N H_N(\brho), \fr{1}{\sqrt{N}} \partial_{\rd} H_N(\bsig), \fr{1}{\sqrt{N}} \partial_{\rd} H_N(\brho)\rt) \in A
    \rt\}.
\]
The next lemma, shown by the Kac--Rice formula and Laplace's method, gives the first and second moments for the relevant critical point counts.
We note that although only an upper bound is stated below for the second moment, it actually holds with equality as shown in \cite[Appendix A]{ben2023exponential}.
\begin{lemma}[{\cite[Theorems 3.1 and 3.2]{arous2020geometry}}]
    \label{lem:moment-rates}
    For any product of intervals $A\subseteq \bbR^2$,
    \[
        \lim_{N\to\infty} \fr1N \log \bbE |\Crt(A)| = \sup_{(E,R) \in A} \Theta(E,R).
    \]
    Furthermore, for any product of intervals $A\subseteq [-1,1] \times \bbR^4$,
    \[
        \limsup_{N\to\infty} \fr1N \log \bbE |\Crt_2(A)| \le \sup_{(q,E_1,E_2,R_1,R_2) \in A} \Xi(q,E_1,E_2,R_1,R_2).
    \]
\end{lemma}

\begin{lemma}
    \label{lem:complexity-zero}
    We have $\Theta(E_0,R_0) = \Xi(0,E_0,E_0,R_0,R_0) = 0$.
\end{lemma}
\begin{proof}
    Let $x_0 = R_0 / \xi''(1)^{1/2}$.
    Then,
    \begin{equation}
        \label{eq:R0-identity-1}
        x_0 = \fr{y}{\xi''(1)^{1/2}} + \fr{\xi''(1)^{1/2}}{y}
    \end{equation}
    so by Lemma~\ref{lem:y-ineq}
    \begin{equation}
        \label{eq:R0-identity-2}
        \sqrt{x_0^2 - 4} = \fr{y}{\xi''(1)^{1/2}} - \fr{\xi''(1)^{1/2}}{y}.
    \end{equation}
    Also clearly $x_0 \ge 2$.
    It follows that
    \begin{align*}
        \kappa(x_0) &=
        \fr14 \lt(\fr{y}{\xi''(1)^{1/2}} + \fr{\xi''(1)^{1/2}}{y}\rt)^2
        - \fr12 - \lt\{
            \fr14 \lt(\fr{y^2}{\xi''(1)} - \fr{\xi''(1)}{y^2}\rt) - \log \fr{y}{\xi''(1)^{1/2}}
        \rt\} \\
        &= \fr{\xi''(1)}{2y^2} + \log \fr{y}{\xi''(1)^{1/2}}
        = \fr{\xi''(1)}{2y^2} + \fr12 \log (1+z) - \fr12 \log \fr{\xi''(1)}{\xi'(1)}.
    \end{align*}
    By \eqref{eq:E0R0-identity},
    \[
        -\fr12 \lt\la (E_0,R_0), \Sigma^{-1} (E_0,R_0) \rt\ra
        = -\fr{zE_0+R_0}{2y}
        = -\fr{z\xi'(1) + z^2 \xi(1)}{2y^2} - \fr12 - \fr{\xi''(1)}{2y^2}.
    \]
    Thus,
    \begin{align*}
        \Theta(E_0,R_0)
        &= \fr12 \log (1+z) - \fr{z\xi'(1) + z^2 \xi(1)}{2y^2} \\
        &= \fr{z^2}{2(1+z)} \lt(\fr{(1+z)\log(1+z)}{z^2} - \fr{1}{z} - \fr{\xi(1)}{\xi'(1)}\rt)
        = 0.
    \end{align*}
    Clearly $\Xi(0,E_0,E_0,R_0,R_0) = 2 \Theta(E_0,R_0)$, which concludes the proof.
\end{proof}

\begin{remark}
    Since we restrict attention to $\Crt(B)$ in this section, we do not need to verify that $\Theta(E_0,\cdot)$ is actually maximized at $R_0$.
    However this is true at least on $R\geq 2\sqrt{\xi''(1)}$ (the range corresponding to local maxima of $H_N$) and follows from Lemma~\ref{lem:stationary-check} later and concavity of $\Theta$.
    Hence the ``annealed complexity'' of local maxima at energy $E_0$ is indeed zero.
    For the special case of pure models, the appearance of the ground state energy as a threshold for annealed complexity was verified in \cite{auffinger2013random}.
\end{remark}

\begin{proof}[Proof of Proposition~\ref{prop:1rsb-restricted-1mt}]
    We will show that
    \begin{equation}
        \label{eq:1rsb-restricted-1mt-goal}
        \Theta(E_0-\eta_1,R_0+\eta_1^{3/4}) = \fr{z\eta_1}{y} + O(\eta_1^{9/8}).
    \end{equation}
    Note that $\kappa$ is $C^{3/2}$ on $[2,+\infty)$, with derivative
    % \mscomment{NOT ANALYTIC. $\kappa'$ is only $C^{1/2}$ at $2$, so Taylor error $O(\eps^{3/2})$ instead of $O(\eps^2)$. (See commented out link.)
    % https://math.stackexchange.com/questions/3050275/a-taylor-theorem-for-h%C3%B6lder-continuous-function
    % I guess taking $\eps=\eta_1^{3/4}$ should work?
    % }
    \begin{equation}
    \label{eq:kappa-deriv-formula}
        \kappa'(x)
        = \fr12x - \fr14 \sqrt{x^2-4} - \fr{x^2}{4 \sqrt{x^2-4}} + \fr{1}{\sqrt{x^2-4}}
        = \fr12 \lt(x - \sqrt{x^2-4}\rt).
    \end{equation}
    This implies in particular that $\kappa'\big(a+\frac{1}{a}\big)=1/a$ for $a\geq 1$.
    Recalling \eqref{eq:R0-identity-1}, \eqref{eq:R0-identity-2}, and using Taylor's theorem for H{\"o}lder continuous functions,
    \begin{align*}
        \kappa\lt(\fr{R_0 + \eta_1^{3/4}}{\sqrt{\xi''(1)}}\rt)
        &= \kappa\lt(\fr{R_0}{\sqrt{\xi''(1)}}\rt)
        + \kappa'\lt(\fr{R_0}{\sqrt{\xi''(1)}}\rt) \fr{\eta_1^{3/4}}{\sqrt{\xi''(1)}}
        + O\lt((\eta_1^{3/4})^{3/2}\rt) \\
        &= \kappa\lt(\fr{R_0}{\sqrt{\xi''(1)}}\rt)
        + \fr{\eta_1^{3/4}}{y}
        + O(\eta_1^{9/8}).
    \end{align*}
    The function $f(E,R) = -\fr12 \la (E,R), \Sigma^{-1} (E,R) \ra$ is clearly analytic, so
    \begin{align*}
        f(E_0 - \eta_1, R_0 + \eta_1^{3/4})
        &= f(E_0,R_0) - \la (-\eta_1,\eta_1^{3/4}), \Sigma^{-1} (E_0,R_0) \ra + O(\eta_1^{3/2}) \\
        &\hspace{-6pt} \stackrel{\eqref{eq:E0R0-identity}}{=} f(E_0,R_0) + \fr{z\eta_1}{y} - \fr{\eta_1^{3/4}}{y} + O(\eta_1^{3/2}).
    \end{align*}
    It follows that
    \[
        \Theta(E_0-\eta_1,R_0+\eta_1^{3/4})
        = \Theta(E_0,R_0) + \fr{z\eta_1}{y} + O(\eta_1^{9/8})
        \stackrel{Lem.~\ref{lem:complexity-zero}}{=}
        \fr{z\eta_1}{y} + O(\eta_1^{9/8}),
    \]
    and thus, by Lemma~\ref{lem:moment-rates},
    \[
        \bbE |\Crt(B)| \ge e^{\Omega(\eta_1) N}.
    \]
    Let $\cR(H_N)$ denote the set of ground state typical $\bsig \in S_N$.
    We apply the Kac--Rice formula \eqref{eq:kac-rice-main} with event
    \[
        \cE = \lt\{\bsig \in (\Crt(B) \setminus \wtCrt(B))\rt\}
        = \lt\{
            \lt(\fr{1}{N} H_N(\bsig), \fr{1}{\sqrt{N}} \partial_\rd H_N(\bsig)\rt)
            \in B
        \rt\} \cap \lt\{
            \bsig \not \in \cR(H_N)
        \rt\}.
    \]
    By the law of iterated expectation, the Kac--Rice formula gives
    \begin{align*}
        \bbE |\Crt(B) \setminus \wtCrt(B)|
        &=
        \int_{S_N} \bbE \Bigg[
            \bbE \lt[
                |\det \nabla^2_\sph H_N(\bsig)| \cdot \ind\{\bsig \not\in \cR(H_N)\}
                \big| H_N(\bsig), \partial_\rd H_N(\bsig), \nabla_\sph H_N(\bsig)
            \rt] \\
            &\qquad \times
            \ind\lt\{
                \lt(\fr{1}{N} H_N(\bsig), \fr{1}{\sqrt{N}} \partial_\rd H_N(\bsig) \rt) \in B
            \rt\} \bigg| \nabla_\sph H_N(\bsig) = \bzero
        \Bigg]
        \varphi_{\nabla_\sph H_N(\bsig)}(\bzero) ~\de \bsig \\
        &\hspace{-14pt}\stackrel{Cor.~\ref{cor:1rsb-typical-whp}}{\le}
        e^{-cN/3}
        \int_{S_N} \bbE \Bigg[
            \bbE \lt[
                |\det \nabla^2_\sph H_N(\bsig)|
                \big| H_N(\bsig), \partial_\rd H_N(\bsig), \nabla_\sph H_N(\bsig)
            \rt] \\
            &\qquad \times
            \ind\lt\{
                \lt(\fr{1}{N} H_N(\bsig), \fr{1}{\sqrt{N}} \partial_\rd H_N(\bsig) \rt) \in B
            \rt\} \bigg| \nabla_\sph H_N(\bsig) = \bzero
        \Bigg]
        \varphi_{\nabla_\sph H_N(\bsig)}(\bzero) ~\de \bsig \\
        &= e^{-cN/3} \bbE |\Crt(B)|.
    \end{align*}
    Thus,
    \[
        \bbE |\wtCrt(B)|
        \ge (1-e^{-cN/3}) \bbE |\Crt(B)|
        \ge e^{\Omega(\eta_1)N}.
    \qedhere
    \]
    % , $\bbE |\Crt(B)| \ge e^{\Omega(\eta_1)N}$.
    % By Corollary~\ref{cor:1rsb-typical-whp} and the Kac--Rice formula,
    % \bhcomment{should this be written out more? I can do if so}
    % \mscomment{yeah that would be good, especially since the end of Subsection 5.1 re-uses this argument (by just saying it follows similarly). Actually we don't currently ever state the Kac--Rice formula which we really should do. Maybe the end of Section 2 can have a quick primer on it while defining various gradients? It can just say that in Section 3 we will use specific known consequences of this formula, which roughly follow because the matrix in the determinant is conditionally a shifted GOE.}
    % \[
    % \bbE |\wtCrt(B)| \ge (1-e^{-cN}) \bbE |\Crt(B)|.
    % \qedhere
    % \]
\end{proof}
We defer the proofs of the following two lemmas to Subsection~\ref{subsec:1rsb-deferred}.
\begin{lemma}
    \label{lem:extremal-combinatorics}
    Let $V \subseteq \cS_N$ be a finite set of points with $|V|=M$.
    \begin{enumerate}[label=(\alph*)]
        \item \label{itm:ramsey} There exists $V_1 \subseteq V$ such that $|V_1| = M^{\Omega(\delta)}$ and $R(\bsig,\brho) \ge -\delta$ for all $\bsig,\brho \in V_1$.
        \item \label{itm:turan} There exists $V_2 \subseteq V \times V$ such that $|V_2| = \Omega(M^2 \delta)$ and $R(\bsig,\brho) \ge -\delta$ for all $(\bsig,\brho) \in V_2$.
    \end{enumerate}
\end{lemma}
\begin{lemma}
    \label{lem:convex-implies-1rsb}
    If the function $q \mapsto \xi''(q)^{-1/2}$ is convex on $[0,1]$, then $\xi$ is strictly 1RSB.
\end{lemma}

\begin{proof}[Proof of Proposition~\ref{prop:1rsb-restricted-2mt}]
    By definition of ground state typical, a.s. any distinct $\bsig,\brho \in \wtCrt(B)$ satisfy $R(\bsig,\brho) \le \delta$.
    Let $V_2 \subseteq \wtCrt(B)^2$ be the set of $(\bsig,\brho)$ which furthermore satisfy $R(\bsig,\brho) \ge -\delta$.
    By Lemma~\ref{lem:extremal-combinatorics}\ref{itm:turan}, $|\wtCrt(B)|^2 \le C\delta^{-1} |V_2|$ for an absolute constant $C$.
    Also, a.s. $V_2 \subseteq \Crt_2(B_2)$ where
    \[
        B_2 = [-\delta,\delta] \times [E_0 - \eta_1 - \eta_3, E_0 - \eta_1 + \eta_3]^2 \times [R_0 + \eta_1^{2/3} - \eta_3, R_0 + \eta_1^{2/3} + \eta_3]^2.
    \]
    Thus, by Lemma~\ref{lem:moment-rates},
    \[
        \bbE |\wtCrt(B)|^2
        \le C\delta^{-1} \bbE |\Crt_2(B_2)|
        \le C\delta^{-1} \exp\lt\{N \sup \Xi(B_2) + o(N)\rt\}.
    \]
    It is clear that $\Xi$ is locally Lipschitz near $(0,E_0,E_0,R_0,R_0)$.
    By Lemma~\ref{lem:complexity-zero}, $\Xi(0,E_0,E_0,R_0,R_0) = 0$, so $\sup \Xi(B_2) \le O(\delta)$.
    The result follows.
\end{proof}

% which is without loss of generality because any perturbation of such $\xi$ remains strictly 1RSB.
% Indeed, if $\xi$ is pure, then $q\mapsto \xi''(q)^{-1/2}$ is convex on $[0,1]$, and this remains true for any perturbation, so \cite[Proposition 2.2]{talagrand2006spherical} implies that $g$ has only two zeros on $[0,1]$.

\begin{proof}[Proof of Proposition~\ref{prop:opt-1rsb}]
    We first prove the proposition for non-pure $\xi$, as we have been assuming throughout the section.
    The statement is monotone in $\delta$, so we may assume $\delta$ is small in $\eps$.
    By Propositions~\ref{prop:1rsb-restricted-1mt} and \ref{prop:1rsb-restricted-2mt} and Paley-Zygmund,
    \begin{equation}
        \label{eq:opt-1rsb-nontrivial-prob-lb}
        \bbP\lt(|\wtCrt(B)| \ge \fr12 \bbE |\wtCrt(B)|\rt) \ge e^{-O(\delta) N}.
    \end{equation}
    Suppose this event holds.
    By Lemma~\ref{lem:extremal-combinatorics}\ref{itm:ramsey}, there exists $V_1 \subseteq \wtCrt(B)$ with
    \[
        |V_1|
        \ge \lt(\fr12 \bbE |\wtCrt(B)|\rt)^{O(\delta)}
        = e^{O(\delta \eta_1 N)}
    \]
    such that $|R(\bsig,\brho)| \le \delta$ for all $\bsig,\brho \in V_1$.
    By the choice of $B$, all these points have energy at least
    \[
        E_0 - \eta_1 - \eta_3
        \ge \cQ(\xi) - \eps/2,
    \]
    where we recall $E_0 = \cQ(\xi)$ and take $\eta_1,\eta_3$ small in $\eps$.
    Let
    \[
        X = \fr1N \sup_{\vbsig \in \Band_{k,1,\delta}(\bzero)}
        \inf_{i\in [k]}
        H_N(\bsig^i).
    \]
    Combining the above shows that for any $k\le e^{O(\delta \eta_1 N)}$,
    \[
        \bbP \lt(
            X \ge \cQ(\xi) - \eps/2
        \rt)
        \ge e^{-O(\delta) N}.
    \]
    A direct calculation shows that for fixed $\bsig$, $H_N(\bsig)$ is $O(N^{1/2})$-Lipschitz in the disorder Gaussians.
    Since suprema and infima preserve Lipschitz constants, $NX$ is also $O(N^{1/2})$-Lipschitz in the disorder Gaussians.
    We thus have the concentration inequality
    \[
        \bbP(|X-\bbE X| \ge t) \le \exp(-ct^2N).
    \]
    Combining the last two inequalities implies that (for $\delta$ small in $\eps$) $\bbP(X \ge \cQ(\xi) - \eps) \ge 1-e^{-cN}$.
    This completes the proof for non-pure $\xi$.

    Finally, we turn to the case where $\xi(q) = \beta^2 q^p$ is pure.
    Then, $H_N(\bsig) = \beta H_N^{(p)}(\bsig)$, where $H_N^{(p)}(\bsig) = \la \bG^{(p)}, \bsig^{\otimes p} \ra$.
    Consider a perturbation $\hH_N(\bsig) = \beta H_N^{(p)}(\bsig) + \iota \beta H_N^{(p+1)}(\bsig)$, for a fixed $\iota>0$ chosen small in $\delta,\eps$.
    This has mixture $\hxi(q) = \beta^2 (q^p + \iota^2 q^{p+1})$.
    Note that
    \[
        \hxi''(q)^{-1/2} = \fr{q^{-(p-2)/2}}{\beta \sqrt{p(p-1)}}\lt(1 + \fr{\iota^2 (p+1)}{p-1}q\rt)^{-1/2}
    \]
    is convex on $[0,1]$, so Lemma~\ref{lem:convex-implies-1rsb} implies $\hxi$ is strictly 1RSB.
    By the result for non-pure $\xi$, there exists $c = c(\hxi,\delta,\eps/2)$ such that for all $k\le e^{cN}$, with probability $1-e^{-cN}$ there exists $\vbsig \in \Band_{k,1,\delta}(\bzero)$ such that for all $i\in [k]$,
    \[
        \fr1N \hH_N(\bsig^i) \ge \cQ(\hxi,0) - \fr{\eps}{2} \ge \cQ(\xi) - \fr{\eps}{2}.
    \]
    By Proposition~\ref{prop:gradients-bounded}, there exists a constant $C$ such that $\fr1N \sup_{\bsig \in \cS_N} H_N^{(p+1)}(\bsig) \le C$ with probability $1-e^{-cN}$.
    On the intersection of these events, for each $i\in [k]$,
    \[
        \fr1N H_N(\bsig^i)
        \ge \fr1N \hH_N(\bsig^i) - \fr{\iota}{N} H_N^{(p+1)}(\bsig^i)
        \ge \cQ(\xi) - \fr{\eps}{2} - C\iota
        \ge \cQ(\xi) - \eps
    \]
    for $\iota$ small in $\eps$.
\end{proof}

\subsection{Deferred Proofs}
\label{subsec:1rsb-deferred}

\begin{proof}[Proof of Lemma~\ref{lem:extremal-combinatorics}]
    Consider the graph $G$ with vertex set $V$ where $(\bsig,\brho)$ is an edge if $R(\bsig,\brho) < -\delta$.
    Note that $G$ does not contain a $r = \lceil 1/\delta \rceil$-clique, because if such a clique $U \subseteq V$ existed, then the Gram matrix $[R(\bsig,\brho)]_{\bsig,\brho \in U}$ would not be positive semi-definite.

    Let $\fR(s,t)$ denote the $(s,t)$ Ramsey number.
    Recall the classic Ramsey upper bound
    \[
        \fR(s,t) \le \binom{s+t-2}{s-1},
    \]
    which can be proved by applying the inequality $\fR(s,t) \le \fR(s-1,t) + \fR(s,t-1)$ recursively.
    Thus $\fR(r,M^{\delta/2}) \lesssim M^{(r-1)\delta/2} \ll M$, so $G$ contains a $M^{\delta/2}$-independent set.
    This proves part \ref{itm:ramsey}.
    Since $G$ avoids an $r$-clique, by Tur\'an's theorem $G$ avoids at least $\fr{1}{r-1} = O(\delta)$ fraction of edges, proving \ref{itm:turan}.
\end{proof}

\begin{lemma}
    \label{lem:0-in-T}
    In any model with $\xi'(0) = 0$, we have $0\in T$.
\end{lemma}
\begin{proof}
    Assume otherwise and let $q \in (0,1]$ be the minimal point in $T$.
    Then $g(q)=0$.
    If $q<1$, then $q$ is an interior local minimizer of $g$, so $0 = g'(q) = -G(q)$; if $q=1$, then the characterization from Lemma~\ref{lem:cs-extremality-0temp} implies $G(q)=0$.
    So in either case $G(q)=0$.
    Also from the definition \eqref{eq:def-G} of $G$ we have $G(0)=0$.

    Recall that the measure $\nu$ given by $\nu([0,s]) = \alpha(s)$ is supported on $T$.
    Thus $\alpha \equiv 0$ on $[0,q)$, and so $\halpha$ is constant on $[0,q]$.
    Therefore $G$ is convex on $[0,q]$.
    Since $G(0)=G(q)=0$, this implies $G \le 0$ on $[0,q]$.

    Thus $g(0) = g(q) + \int_0^q G(s)~\de s \le g(q)=0$.
    So $0\in T$, contradicting minimality of $q$.
\end{proof}
\begin{proof}[Proof of Lemma~\ref{lem:convex-implies-1rsb}]
    Lemma~\ref{lem:0-in-T} implies $0\in T$, and Lemma~\ref{lem:cs-extremality-0temp} ensures $1\in T$.
    The following argument, adapted from \cite[Proposition 2.2]{talagrand2006spherical}, shows that $|T| \le 2$, which then implies $T = \{0,1\}$.

    Consider any $q_1,q_2 \in T$ such that $q_1<q_2$.
    If $q_i \in (0,1)$, then $G(q_i)=0$ because $q_i$ is an interior local minimizer of $g$; if $q_i=1$ then $G(q_i)=0$ by Lemma~\ref{lem:cs-extremality-0temp}; and if $q_i=0$ then $G(q_i)=0$ by definition \eqref{eq:def-G} of $G$.
    So $G(q_1)=G(q_2)=0$.
    Moreover $g(q_1)=g(q_2)=0$, so $\int_{q_1}^{q_2} G(s)~\de s = 0$.
    It follows that there are two points in $q_3,q_4 \in (q_1,q_2)$ such that $G'(q_3)=G'(q_4)=0$.

    We have thus shown that between any two elements of $T$ lies two zeros of $G'$.
    However,
    $
        G'(q) = \xi''(q) - \fr{1}{\halpha(q)^2},
    $
    so at any zero of $G'$ we have $\xi''(q)^{-1/2} = \halpha(q)$.
    However $\xi''(q)^{-1/2}$ is convex by assumption, while $\halpha(q)$ is concave by definition, so these functions intersect at most twice.
    It follows that $|T| \le 2$.
\end{proof}

\section{Building a Model from Fundamental Types: Proof of Theorem~\ref{thm:main}}
\label{sec:building-model}

In this section we complete the proof of Theorem~\ref{thm:main}.
The proof proceeds in two steps:
\begin{enumerate}[label=(\arabic*)]
    \item Using the lower bounds in Propositions~\ref{prop:fe-rs} through \ref{prop:opt-frsb} and a uniform concentration lemma due to Subag (Lemma~\ref{lem:uc} below) we prove Theorem~\ref{thm:main-avg-case} below, which constructs an ultrametric tree with somewhat more lenient constraints than Theorem~\ref{thm:main}.
    \item By pruning this ultrametric tree we arrive at the ultrametric tree in Theorem~\ref{thm:main}.
\end{enumerate}

Before giving the full proof, we note the free energy lower bound Corollary~\ref{cor:parisi-LB} follows by combining what we have done with \cite{subag2018free}, where the decomposition approach we follow was introduced.
Namely \cite[Theorem 5]{subag2018free} used the existence of many orthogonal replicas to lower bound the free energy by a sum of the ground state on a subsphere $\sqrt{q}\cS_N$ plus the free energy of a ``band'' model centered at a typical point on this subsphere.
Using this idea sequentially with $q=q_0,q_1,\dots$ yields Corollary~\ref{cor:parisi-LB} since each intermediate model takes one of the four fundamental types analyzed previously in this paper.

Our analysis below is conceptually similar to \cite{subag2018free}, but constructs $e^{\Omega(N)}$ near-orthogonal approximate ground states in each intermediate model, and hence gives a larger tree of pure states than was known to exist previously by any method in this generality.
By contrast \cite{subag2018free} relies on Chatterjee's superconcentration, which only gives a slowly diverging number of near-orthogonal approximate ground states.
This improvement is also what allows us to prove the lower tail large deviations for the ground state have speed at least $N^2$ (see Subsection~\ref{subsec:LDP-speed}).

\subsection{A Tree with Local Constraints}

% Our proof of Theorem~\ref{thm:main-avg-case}
% will decompose a model $(\xi,h)$ into several sub-models of these types, which can be analyzed with these results.

\begin{lemma}
    \label{lem:intervals}
    For any model $\xi$, the set $S$ (recall \eqref{eq:def-S}) is a disjoint union of finitely many closed intervals, possibly including atoms.
    Moreover $q_D < 1$.
    % Moreover $\sup(S) < 1$, and $r=\inf(S)$ satisfies $\xi'(r) \le r\xi''(r)$.
\end{lemma}
\begin{proof}
    The first statement follows from \cite[Corollary 1.3]{jagannath2018bounds} and the ensuing observation that $\mathfrak{d}=\lt(\frac{1}{\sqrt{\xi''}}\rt)''$ changes sign finitely many times on $[0,1]$.
    Indeed, $S$ is precisely the \emph{coincidence set} denoted $\{\eta=\xi\}$, as can be seen from the display between (1.1.1) and (1.1.2) therein.

    For the second statement, note that $\hx(q) = 1-q$ for $q$ in some interval $[\hq,1]$, so $F(q) \le C - \fr{1}{1-q}$ for some $C$ independent of $q$.
    Hence $f(q) \le Cq + \log(1-q)$, so $\lim_{q\to 1^-} f(q) = -\infty$, which implies $q_D < 1$.
\end{proof}

\begin{definition}
    A sequence $q_0,\ldots,q_D$ with $0 \le q_0 < \cdots < q_D \le 1$ is an \textbf{$S$-refinement} if
    \[
        \partial S \subseteq \{q_0, \ldots, q_D \} \subseteq S,
    \]
    where $\partial S = S \cap \overline{S^c}$ is the boundary of $S$ in $\bbR$. (In particular $q_0 = \inf(S)$ and $q_D = \sup(S)$.)
\end{definition}

In the following variant of Definition~\ref{def:ultrametric-tree}, the orthogonality constraints are enforced only locally.
For $u,v\in \bbT$, write $u\sim v$ if $u=v$, or $u,v$ are siblings, or one of $u,v$ is the parent of the other.
\begin{definition}
    \label{def:local-ultrametric-tree}
    Let $k,D \in \bbN$, $0 \le q_0 < \cdots < q_D \le 1$, $\vq = (q_0,\ldots,q_D)$, and $\delta > 0$.
    A $(k,D,\vq,\delta)$-locally ultrametric tree is a collection of points $(\bsig^u)_{u \in \bbT}$ such that \eqref{eq:ultrametric} holds for all $u \sim v$.
\end{definition}

We will first prove the following variant of Theorem~\ref{thm:main}, where the properties required of the ultrametric tree are relaxed in two ways: ultrametricity will be enforced only locally, and we will lower bound the average energy increment from each node to its children rather than the energy of each node.
We will deduce Theorem~\ref{thm:main} from Theorem~\ref{thm:main-avg-case} in Subection~\ref{subsec:pruning} by pruning this tree.
\begin{theorem}
    \label{thm:main-avg-case}
    For any $\delta,\eps>0$, $D \in \bbN$, and $S$-refinement $q_0,\ldots,q_D$, there exists $c>0$ such that the following holds for any $k\le e^{cN}$.
    With probability $1-e^{-cN}$, there is a $(k,D,\vq,\delta)$-locally ultrametric tree $(\bsig^u)_{u\in \bbT}$ with the following properties.
    \begin{enumerate}[label=(\roman*)]
        \item \label{itm:main-avg-case-root} Energy of root: $\fr1N H_N(\bsig^\emptyset) \ge E(q_0) - \eps$.
        \item \label{itm:main-avg-case-energy} Parent-to-child energy increments: for each $u\in \bbT \setminus \bbL$,
        \[
            \fr{1}{kN} \sum_{i=1}^k \lt(H_N(\bsig^{ui}) - H_N(\bsig^u) \rt)
            \ge E(q_{|u|+1}) - E(q_{|u|}) - \eps.
        \]
        \item \label{itm:main-avg-case-pure} Free energy of pure states: for each $u\in \bbL$,
        \[
            \fr{1}{kN} \log \int_{\Band_{k,1,\delta}(\bsig^u)}
            \exp \lt(\sum_{i=1}^k \lt(H_N(\brho^i) - H_N(\bsig^u)\rt)\rt)
            ~\de \vbrho
            \ge \cP(\xi) - E(q_D) - \eps.
        \]
    \end{enumerate}
\end{theorem}

% \bhcomment{wip: some of these defns currently scattered through this doc can be moved up in the end}

\subsection{Model Decomposition into Fundamental Types}

Fix parameters $\delta, \eps, D, (q_0,\ldots,q_D)$ as in Theorem~\ref{thm:main-avg-case}.
We take as convention $q_{-1}=0$, $q_{D+1}=1$.
Define $\xi_{-1}(x) = \xi(q_0 x)$ and, for $0 \le d \le D$,
\begin{equation}
    \label{eq:def-xi-d}
    \xi_d(x) = \xi(q_d + (q_{d+1}-q_d)x) - \xi(q_d) - \xi'(q_d)(q_{d+1}-q_d)x.
\end{equation}
The following proposition gives the link between $\xi$ and each $\xi_d$ from the point of view of the Parisi formula. It follows by matching order parameters, see \cite[Proposition 11]{subag2018free}.

\begin{proposition}
    \label{prop:component-types}
    The following hold.
    \begin{enumerate}[label=(\alph*)]
        \item \label{itm:comp-type-root} The model $\xi_{-1}$ is topologically trivial and satisfies $\cQ(\xi_{-1}) = E(q_0)$. (We treat this part as vacuous if $q_0=0$, in which case $\xi_{-1} \equiv 0$.)
        \item \label{itm:comp-type-middle} For each $0\le d\le D-1$, $\xi_d$ is either strictly 1RSB or strictly FRSB, and satisfies $\cQ(\xi_d) = E(q_{d+1}) - E(q_d)$.
        \item \label{itm:comp-type-pure} The model $\xi_D$ is strictly RS and satisfies $\cP(\xi_D) = \cP(\xi) - E(q_D)- \fr12 \log(1-q_D)$.
    \end{enumerate}
\end{proposition}

% For $k\in \bbN$, $d\le D$, and $\norm{\bsig}_2 = \sqrt{q_d N}$, define
% \[
%     Band_{k,\delta}(\bsig) = \lt\{
%         (\brho^1,\ldots,\brho^k) :
%         |R(\brho^i - \bsig, \bsig)| \le \delta,
%         |R(\brho^i - \bsig, \brho^j - \bsig)| \le \delta
%     \rt\}.
% \]
For $0\le d\le D-1$ and $\norm{\bsig}_2 = \sqrt{q_d N}$, define
\[
    F_{d,k}(\bsig) = \fr{1}{kN} \max_{\vbrho \in \Band_{k,q_{d+1},\delta}(\bsig)}
    \sum_{i=1}^k \lt(H_N(\brho^i) - H_N(\bsig))\rt),
\]
and for $\norm{\bsig}_2 = \sqrt{q_D N}$,
\[
    F_{D,k}(\bsig) = \fr{1}{kN} \log \int_{\Band_{k,1,\delta}(\bsig)}
    \exp \lt(\sum_{i=1}^k \lt(H_N(\brho^i) - H_N(\bsig)\rt)\rt)
    ~\de \vbrho.
\]
The following \textbf{uniform concentration} lemma was proved at finite temperature in \cite[Proposition 1]{subag2018free}, as consequence of the concentration of Lipschitz functions of Gaussians. Its proof at zero temperature is identical.

\begin{lemma}
    \label{lem:uc}
    For all $\eps>0$, there exists $\delta_0 = \delta_0(\xi,\eps)$ and $k_0 = k_0(\xi,\eps)$ such that for all $\delta \le \delta_0$, $k\ge k_0$ the following holds with probability $1-e^{-cN}$.
    For all $0\le d\le D$ and all $\norm{\bsig}_2 = \sqrt{q_d N}$, $|F_{d,k}(\bsig) - \bbE F_{d,k}(\bsig)| \le \eps$.
\end{lemma}

\begin{proposition}
    \label{prop:energy-increments}
    There exists $c>0$ (depending on $\delta,k,D,q_0,\ldots,q_D)$ such that the following holds.
    \begin{enumerate}[label=(\alph*)]
        \item \label{itm:energy-incr-root} With probability $1-e^{-cN}$, there exists $\norm{\bsig}_2 = \sqrt{q_0 N}$ such that $\fr1N H_N(\bsig) \ge E(q_0) - \eps$.
        \item \label{itm:energy-incr-middle} For all $k\le e^{cN}$, $0\le d \le D-1$ and any fixed $\bsig$ with $\norm{\bsig}_2 = \sqrt{q_d N}$:
        \[
        \bbE F_{d,k}(\bsig) \ge E(q_{d+1}) - E(q_d) - \eps.
        \]
        \item \label{itm:energy-incr-pure} For all $k\le e^{cN}$ and any fixed $\bsig$ with $\norm{\bsig}_2 = \sqrt{q_D N}$:
        \[
        \bbE F_{D,k}(\bsig) \ge \cP(\xi) - E(q_D) - \eps.
        \]
    \end{enumerate}
\end{proposition}
\begin{proof}
    Since $H_N^{(-1)}(\brho) \equiv H_N(\sqrt{q_0} \brho)$ has mixture $\xi_{-1}$, part \ref{itm:energy-incr-root} follows by Propositions~\ref{prop:component-types}\ref{itm:comp-type-root} and \ref{prop:opt-topologically-trivial}.

    Next we prove part~\ref{itm:energy-incr-middle}.
    For $0\le d\le D-1$, and fixed $\norm{\bsig}_2 = \sqrt{q_d N}$, the model
    \[
        H_N^{(d,\bsig)}(\brho) = H_N(\sqrt{q_{d+1}-q_d} \brho + \bsig) - H_N(\bsig) - \sqrt{q_{d+1}-q_d} \la \nabla H_N(\bsig), \brho \ra,
    \]
    restricted to the band $\bsig^\perp = \{\brho \in \cS_N : R(\brho, \bsig) = 0\}$ is a $(N-1)$-dimensional model with mixture $\xi_d$.
    % \mscomment{It doesn't look like the gradient correction term is being handled anywhere. Maybe it is easier to just work in the $N-2$ dimensional subspace that's also orthogonal to the gradient?}
    % \bhcomment{It's handled in step ($\ast$) below; I rewrote it more clearly, is this ok?}
    By Proposition~\ref{prop:component-types}\ref{itm:comp-type-middle}, this model is either strictly 1RSB or strictly FRSB.
    By Propositions~\ref{prop:opt-1rsb} and \ref{prop:opt-frsb},
    % in either case w.h.p. there exists $\vbrho \in \Band_{k,1,\delta}(\bzero)\cap(\bsig^{\perp})^k$ such that $\fr1N H_N^{(d,\bsig)}(\brho^i) \ge E(q_{d+1}) - E(q_d) - \eps/2$ for all $i\in [k]$.
    % In particular,
    % \[
    %     \fr{1}{kN} \max_{\vbrho \in \Band_{k,1,\delta}(\bzero)\cap(\bsig^{\perp})^k}
    %     \sum_{i=1}^k H_N^{(d,\bsig)}(\brho^i)
    %     \ge E(q_{d+1}) - E(q_d) - \eps/2.
    % \]
    % By concentration of the left-hand side via Borell-TIS, this implies:
    combined with concentration via Borell-TIS, we thus find:
    \[
        \fr{1}{kN} \bbE \max_{\vbrho \in \Band_{k,1,\delta}(\bzero)\cap(\bsig^{\perp})^k}
        \sum_{i=1}^k H_N^{(d,\bsig)}(\brho^i)
        \ge E(q_{d+1}) - E(q_d) - \eps.
    \]
    Still with $\bsig$ fixed, let $\vbrho_\ast$ attain the maximum in the previous display.
    Then we have the inequality chain
    \begin{align*}
        \bbE F_{d,k}(\bsig)
        &\ge \fr{1}{kN} \bbE \max_{\vbrho \in \Band_{k,1,\delta}(\bzero)\cap(\bsig^{\perp})^k}
        \sum_{i=1}^k \lt(H_N(\sqrt{q_{d+1}-q_d} \brho^i + \bsig) - H_N(\bsig)\rt) \\
        &\ge \fr{1}{kN} \bbE
        \sum_{i=1}^k \lt(H_N(\sqrt{q_{d+1}-q_d} \brho^i_\ast + \bsig) - H_N(\bsig)\rt) \\
        &\stackrel{(\ast)}{=} \fr{1}{kN} \bbE
        \sum_{i=1}^k H_N^{(d,\bsig)}(\brho^i_\ast)
        \ge E(q_{d+1}) - E(q_d) - \eps,
    \end{align*}
    where the step $(\ast)$ uses that $H_N^{(d,\bsig)}$ is independent of $\nabla H_N(\bsig)$ as a process.
    This proves part \ref{itm:energy-incr-middle}.

    The proof of \ref{itm:energy-incr-pure} is similar.
    The model $H_N^{(D,\bsig)}$ restricted to $\bsig^\perp$ is a $(N-1)$-dimensional model with mixture $\xi_D$.
    By Proposition~\ref{prop:component-types}\ref{itm:comp-type-pure}, this model is strictly RS with respect to the normalized $N-2$ dimensional Hausdorff measure $\wt\cH_{N-2}$ on $\bsig^{\perp}$.
    By Propositions~\ref{prop:fe-rs} and \ref{prop:component-types}\ref{itm:comp-type-pure}, with probability $1-e^{-cN}$:
    \[
        \fr{1}{kN}
        \log
        \int_{\Band_{k,1,\delta^2}(\bzero)\cap(\bsig^{\perp})^k}
        \exp \lt(
        \sum_{i=1}^k H_N^{(D,\bsig)}(\brho^i)\rt)
        ~\de \wt\cH_{N-2}^k(\vbrho)
        \ge \cP(\xi) - E(q_D)- \fr12 \log(1-q_D) - \eps/3.
    \]
    Moreover since $H_N\in K_N$ with probability $1-e^{-cN}$, we easily find that with high probability, the restricted free energy with respect to the original uniform measure on $\cS_N$ obeys a similar bound:
    \[
    \fr{1}{kN}
        \log
        \int_{\Band_{k,1,\delta}(\bsig)}
        \exp \lt(\sum_{i=1}^k H_N(\brho^i)-H_N(\bsig)\rt)
        ~\de \vbrho
        \ge \cP(\xi) - E(q_D) - \eps/2.
    \]
    Here the term $\fr12 \log(1-q_D)$ disappeared from rescaling.
    By Lipschitz concentration of the left-hand side,
    \[
        \fr{1}{kN} \bbE \int_{\Band_{k,1,\delta}(\bsig)}
        \exp \lt(\sum_{i=1}^k H_N(\brho^i)-H_N(\bsig)\rt)
        ~\de \vbrho
        \ge \cP(\xi) - E(q_D) - 2\eps/3.
    \]
    This concludes the proof.
    % \bhcomment{write out this integrating over a small amount in the $N$th dimension nonsense}
\end{proof}

\begin{proof}[Proof of Theorem~\ref{thm:main-avg-case}]
    It suffices to prove the theorem for $k=e^{cN}$ and $\delta \le \delta_0(\xi,\eps)$, as the statement is clearly monotone in $k$ and $\delta$.
    As $k$ is growing in $N$, $k\ge k_0(\xi,\eps)$ and Lemma~\ref{lem:uc} holds.

    % Propositions~\ref{prop:1rsb-component} and \ref{prop:frsb-component} imply that for $d\le D-1$ and \textbf{fixed} $\norm{\bsig}_2 = \sqrt{q_d N}$,
    % \[
    %     \bbE F_{d,k}(\bsig) \ge E(q_{d+1}) - E(q_d) - \eps.
    % \]
    % Similarly, Proposition~\ref{prop:pure-state-fe} implies that for fixed $\norm{\bsig}_2 = \sqrt{q_D N}$,
    % \[
    %     \bbE F_{D,k}(\bsig) \ge \cP(\xi,h) - E(q_D) - \eps.
    % \]
    By Proposition~\ref{prop:energy-increments}, for $d\le D-1$ and any fixed $\norm{\bsig}_2 = \sqrt{q_d N}$, $\norm{\bsig}_2 = \sqrt{q_D N}$, respectively,
    \[
        \bbE F_{d,k}(\bsig) \ge E(q_{d+1}) - E(q_d) - \eps/2, \qquad
        \bbE F_{D,k}(\bsig) \ge \cP(\xi) - E(q_D) - \eps/2.
    \]
    By Lemma~\ref{lem:uc}, with probability $1-e^{-cN}$, for \textbf{all} $\norm{\bsig}_2 = \sqrt{q_d N}$, $\norm{\bsig}_2 = \sqrt{q_D N}$, respectively,
    \begin{equation}
        \label{eq:good-energy-increments}
        F_{d,k}(\bsig) \ge E(q_{d+1}) - E(q_d) - \eps, \qquad
        F_{D,k}(\bsig) \ge \cP(\xi) - E(q_D) - \eps.
    \end{equation}
    By Proposition~\ref{prop:energy-increments}\ref{itm:energy-incr-root}, with probability $1-e^{-cN}$ there exists $\norm{\bsig^\emptyset}_2 = \sqrt{q_0 N}$ such that $\fr1N H_N(\bsig^\emptyset) \ge E(q_0) - \eps$.
    Starting from this point, we can construct the remaining $\bsig^u$ using \eqref{eq:good-energy-increments}.
\end{proof}

\subsection{Pruning the Relaxed Tree}
\label{subsec:pruning}

We apply Theorem~\ref{thm:main-avg-case} with parameters $(\delta^2/2D^4, \eps/2(D+1))$ in place of $(\delta,\eps)$.
Let $c>0$ be given by this theorem and $k = e^{cN}$.
Then, with probability $1-e^{-cN}$, there is a $(k,D,\vq,\delta)$-locally ultrametric tree $(\bsig^u)_{u\in \bbT}$, where $\bbT = \bbT(k,D)$, with properties \ref{itm:main-avg-case-root}, \ref{itm:main-avg-case-energy}, \ref{itm:main-avg-case-pure} (where $\delta,\eps$ are replaced by $\delta^2/2D^4, \eps/2(D+1)$).
Throughout this subsection, assume this event holds and $K_N$ from Proposition~\ref{prop:gradients-bounded} holds.

Let $c_0 = c/2D$ and $k'' = e^{c_0N}$.
We will show that for a subtree $\bbT'' \cong \bbT(k'',D)$ of $\bbT$, $(\bsig^u)_{u\in \bbT''}$ has the properties described in Theorem~\ref{thm:main}.
We obtain $\bbT''$ from $\bbT$ by two steps of pruning: we first ensure all energies are suitably large (Proposition~\ref{prop:pruning-intermediate-1}), and then that global overlap constraints are satisfied (Proposition~\ref{prop:pruning-intermediate-2}).

\begin{proposition}
    \label{prop:pruning-intermediate-1}
    For an absolute constant $C = C(\xi)$ and $k' = \fr{\eps}{CD} e^{cN}$, there exists a subtree $\bbT' \cong \bbT(k',D)$ of $\bbT$ such that the following holds.
    For each $u\in \bbT' \setminus \bbL$ and all $i\in [k]$ such that $ui \in \bbT'$,
    \begin{equation}
        \label{eq:pruning-intermediate-energy}
        \fr{1}{N} \lt(H_N(\bsig^{ui}) - H_N(\bsig^u) \rt)
        \ge E(q_{|u|+1}) - E(q_{|u|}) - \eps/(D+1).
    \end{equation}
    % with the following properties.
    % \begin{enumerate}[label=(\roman*)]
    %     \item \label{itm:pruning-intermediate-root} Root energy: $\fr1N H_N(\bsig^\emptyset) \ge E(q_0) - \eps/D$.
    %     \item \label{itm:pruning-intermediate-energy} Energy increments: for each $u\in \bbT' \setminus \bbL'$ and all children $ui$ of $u$ in $\bbT'$,
    %     \[
    %         \fr{1}{N} \lt(H_N(\bsig^{ui}) - H_N(\bsig^u) \rt) \ge
    %         \ge E(q_{|u|+1}) - E(q_{|u|}) - \eps/D.
    %     \]
    %     \item \label{itm:pruning-intermediate-pure} Pure states: for each $u\in \bbL$,
    %     \[
    %         \fr{1}{k'N} \log \int_{\Band_{k',1,\delta}(\bsig^u)}
    %         \exp \lt(\sum_{i=1}^{k'} \lt(H_N(\brho^i) - H_N(\bsig^u)\rt)\rt)
    %         ~\de \vbrho
    %         \ge \cP(\xi,h) - E(q_D) - \eps/D.
    % \end{enumerate}
\end{proposition}
\begin{proof}
    We will construct $\bbT'$ by breadth-first exploration starting from the root: at every non-leaf $u$ we encounter, we will find $k'$ children of it such that \eqref{eq:pruning-intermediate-energy} holds.

    Consider one such $u$, and abbreviate $\Delta_i = \fr{1}{N} \lt(H_N(\bsig^{ui}) - H_N(\bsig^u) \rt)$, $\Delta E_{|u|} = E(q_{|u|+1}) - E(q_{|u|})$.
    By property \ref{itm:main-avg-case-energy} of Theorem~\ref{thm:main-avg-case},
    \[
        \fr{1}{k} \sum_{i=1}^k \Delta_i
        \ge \Delta E_{|u|} - \eps/2(D+1).
    \]
    On event $K_N$, $\sup_{\norm{\bx}_2 \le \sqrt{N}} |H_N(\bsig)| \le C_0 N$, so deterministically $|\Delta_i| \le 2C_0$ for all $i\in [k]$.
    By Markov's inequality on $\mathrm{unif}([k])$,
    \begin{align*}
        \fr{1}{k} \lt|i\in [k]: \Delta_i \le \Delta E_{|u|} - \eps/(D+1) \rt|
        &= \fr{1}{k} \lt|i\in [k]: 2C_0 - \Delta_i \ge 2C_0 - \Delta E_{|u|} - \eps/(D+1) \rt| \\
        &\le \fr{2C_0 - \Delta E_{|u|} - \eps/(D+1)}{2C_0 - \Delta_i \le 2C_0 - \Delta E_{|u|} - \eps/2(D+1)},
    \end{align*}
    and thus
    \[
        \fr{1}{k} \lt|i\in [k]: \Delta_i \ge \Delta E_{|u|} - \eps/(D+1) \rt|
        \ge \fr{\eps/2D}{2C_0 - \Delta_i \le 2C_0 - \Delta E_{|u|} - \eps/2(D+1)}
        \ge \fr{\eps}{4C_0(D+1)}.
    \]
    Setting $C = 8C_0$, we conclude that we can find $k'$ children of $u$ such that \eqref{eq:pruning-intermediate-energy} holds.
\end{proof}

\begin{proposition}
    \label{prop:pruning-intermediate-2}
    There is a subtree $\bbT'' \cong \bbT(k'',D)$ of $\bbT'$ such that for any distinct parent-child pairs $(u,ui)$, $(v,vj)$ in $\bbT''$ (where possibly $u=v$), $|R(\bsig^{ui}-\bsig^u,\bsig^{vj}-\bsig^v)| \le \delta/D^2$.
\end{proposition}
\begin{proof}
    We will construct $\bbT''$ by breadth-first exploration starting from the root.
    We will abbreviate $\bx^{ui} = \bsig^{ui} - \bsig^u$.
    We maintain a set
    \[
        \cC = \{
            \bx^{ui} :
            \text{$(u,ui)$ is a parent-child pair in $\bbT''$}
        \},
    \]
    and will maintain the invariant that $|R(\bx,\by)| \le \delta/D^2$ for any distinct $\bx,\by \in \cC$.
    At every non-leaf $u$ we encounter in the exploration, we will find $k''$ children of it such that, when the corresponding $k''$ parent-child pairs are added to $\bbT''$, this invariant continues to hold.
    Note that at all times,
    \[
        |\cC| \le k'' + (k'')^2 + \cdots + (k'')^D \le 2(k'')^D = 2e^{cN/2}.
    \]
    Consider the step in this procedure where we choose children for node $u$.
    Let $I_u = \{i\in [k] : ui \in \bbT'\}$, so $|I_u| = k'$.
    For $\by \in \cC$, let $I_u^+(\by),I_u^-(\by)$ be the sets of $i\in [k]$ such that $R(\by,\bx^{ui}) > \delta/D^2$ and $R(\by,\bx^{ui}) < -\delta/D^2$.
    We \textbf{claim} that at any such step, $|I_u^+(\by)| \le 2D^4/\delta^2$ and $|I_u^-(\by)| \le 2D^4 / \delta^2$.
    This claim suffices, since it implies
    \[
        \lt|I_u \setminus \bigcup_{\by \in \cC} (I_u^+(\by) \cup I_u^-(\by)) \rt|
        \ge k' - 2e^{cN/2} \cdot \fr{4D^4}{\delta^2}
        \gg k''.
    \]
    Indeed, we may choose any $k''$ elements $i$ in this set and add the corresponding $ui$ to $\bbT''$, which preserves the required invariant by definition.

    It remains to prove the above claim, which we do now.
    We bound only $|I_u^+(\by)|$ since the case of $|I_u^-(\by)|$ is analogous.
    For all $i\in I_u^+(\by)$, write
    \[
        \bx^{ui} = \fr{R(\by,\bx^i)}{R(\by,\by)}\by + \btau^{ui}
    \]
    for $\btau^{ui} \perp \by$.
    Because $\bbT$ is a $(k,D,\vq,\delta^2/2D^4)$-locally ultrametric tree, $|R(\bx^{ui},\bx^{uj})| \le \delta^2/2D^4$ for all $i\neq j$.
    Thus, for all distinct $i,j\in I_u^+(\by)$.
    \[
        \fr{\delta^2}{2D^4}
        \ge R(\bx^{ui},\bx^{uj})
        = \fr{R(\by,\bx^i)R(\by,\bx^j)}{R(\by,\by)} + R(\btau^{ui},\btau^{uj})
        \ge \fr{\delta^2}{D^4} + R(\btau^{ui},\btau^{uj}),
    \]
    where we use that $R(\by,\by) \le 1$.
    Thus $R(\btau^{ui},\btau^{uj}) \le -\delta^2/2D^4$.
    However, $R(\btau^{ui},\btau^{ui}) \le R(\bx^{ui},\bx^{ui}) \le 1$.
    Thus $|I_u^+(\by)| \le 2D^4/\delta^2$, as if not the Gram matrix of $(\btau^{ui})_{i\in I_u^+(\by)}$ would not be positive semi-definite.
    % A similar argument shows $|I_u^-(\by)| \le 2D^4 / \delta^2$.
\end{proof}

\begin{proof}[Proof of Theorem~\ref{thm:main}]
    We will show $\bbT''$ satisfies the desired properties.
    First, for any $u,v \in \bbT''$, let $|u| = d_1$, $|v| = d_2$, and let $(\emptyset=u_0,u_1,u_2,\ldots,u_{d_1}=u)$, $(\emptyset=v_0,v_1,\ldots,v_{d_2}=v)$ be the ancestor paths of $u,v$.
    Also let $\ell = u \wedge v$, so $u_\ell=v_\ell$ is the least common ancestor of $u,v$.
    Then
    \begin{equation}
        \label{eq:ancestor-incr-decomposition}
        R(\bsig^u,\bsig^v)
        = \sum_{i=0}^{d_1-1} \sum_{j=0}^{d_2-1}
        R(\bsig^{u_{i+1}}-\bsig^{u_i},\bsig^{v_{j+1}}-\bsig^{v_j}).
    \end{equation}
    The sub-sum corresponding to $0\le i=j<\ell$ equals
    \[
        \sum_{i=0}^{\ell-1}
        R(\bsig^{u_{i+1}}-\bsig^{u_i},\bsig^{u_{j+1}}-\bsig^{u_j})
        = \sum_{i=0}^{\ell-1}
        (q_{i+1}-q_i)
        = q_\ell,
    \]
    while the remaining terms of \eqref{eq:ancestor-incr-decomposition} are bounded by $\delta / D^2$ in absolute value by Proposition~\ref{prop:pruning-intermediate-2}.
    Thus
    \[
        |R(\bsig^u,\bsig^v) - q_{u\wedge v}| \le D^2 \cdot \delta/D^2 = \delta,
    \]
    so $(\bsig^u)_{u\in \bbT''}$ is a $(k'',D,\vq,\delta)$-ultrametric tree.
    By property \ref{itm:main-avg-case-root} of Theorem~\ref{thm:main-avg-case} and \eqref{eq:pruning-intermediate-energy}, for all $u\in \bbT''$ with $|u|=d$,
    \[
        \fr1N H_N(\bsig^u)
        = \fr1N H_N(\bsig^\emptyset)
        + \sum_{i=0}^{d-1} \fr1N \lt(H_N(\bsig^{u_{i+1}}) - H_N(\bsig^{u_i})\rt)
        \ge E(q_d) - (d+1)\eps/(D+1)
        \ge E(q_d) - \eps.
    \]
    Thus property \ref{itm:thm-main-energy} of Theorem~\ref{thm:main} holds.
    Property \ref{itm:thm-main-pure} of Theorem~\ref{thm:main} follows immediately from property \ref{itm:main-avg-case-pure} of Theorem~\ref{thm:main-avg-case}, as this property is monotone in $k'$.
\end{proof}

\begin{proof}[Proof of Corollary~\ref{cor:many-orthogonal}]
    The proof is essentially identical to Theorem~\ref{thm:main} but without the strictly RS part. Thus we just give an outline.
    Using \cite[Theorem 1.13]{jagannath2017low} and analyticity of $\xi$, it follows that Lemma~\ref{lem:intervals} also holds for $T$.
    With the obvious definition, let $q_0,\dots,q_D=1$ be a $T$-refinement.
    Proposition~\ref{prop:component-types} remains true with the same proof, except that $\xi_D=0$ is now trivial.
    The remainder of the proof is as before.
\end{proof}

\section{Large Deviations for the Ground State}
\label{sec:LDP}

Here we make a brief study of large deviations for the ground state energy $GS_N=\max_{\bsig\in\cS_N} H_N(\bsig)$.
\cite{fyodorov2023replica} recently investigated this problem using the replica method, obtaining very interesting but non-rigorous results.
It was predicted that for 1RSB models (without external field), the upper tail has rate function given by a natural Kac--Rice upper bound, referred to as ``replica-symmetric'' behavior therein.
We verify this prediction in Subsection~\ref{subsec:LDP-1RSB} by adapting the interpolation-enhanced truncation from Section~\ref{sec:moment}.
In Subsection~\ref{subsec:LDP-speed} we employ Corollary~\ref{cor:many-orthogonal} to show the lower tail speed transitions to $\Omega(N^2)$ below $\cQ(\xi-\gamma_1^2 t)$ for general mixtures.
As mentioned in Remark~\ref{rem:lower-tail-existing-results}, the super-linearity in the lower tail is closely connected to the Dotsenko--Franz--M{\'e}zard conjecture \cite{dotsenko1994partial,talagrand2007large,jagannath2017approximate}.

\subsection{Upper Tail for 1RSB Models}
\label{subsec:LDP-1RSB}

We assume in this subsection that $\gamma_1 = 0$ and $\xi$ is 1RSB, i.e. the minimizer $(L,\alpha)$ of \eqref{eq:cs-functional-0temp-inf} satisfies $\alpha \equiv u$.
Unlike Section~\ref{sec:moment}, we do not assume \textbf{strict} 1RSB, but parameters such as $y,z,E_0,R_0$ still retain the same definitions.
We will also assume throughout this subsection that $\xi$ is not pure.
Similarly to Proposition~\ref{prop:opt-1rsb}, pure $\xi$ can be handled by adding small perturbation terms to $\xi$, e.g. chosen small enough so the perturbation Hamiltonian has maximum absolute value at most $\delta N$ with probability $1-e^{N/\delta}$ (such perturbations have essentially no effect even in a large deviation sense).

The main computation is again encapsulated in controlling conditional band models as described in Lemma~\ref{lem:1rsb-conditional-law}. Note that the only term in \eqref{eq:1rsb-band-expectation} that depends on $(E,R)$ is
\[
N \lt\la
    v^q,
    \begin{bmatrix}
        E - E_0 \\
        R - R_0
    \end{bmatrix}
\rt\ra
\]
for $v^q=(v^q_E,v^q_R)$ defined in \eqref{eq:def-vq}.
\begin{proposition}
\label{prop:LDP-regression-coefs}
    For any $\xi,\eps$ with $\gamma_1=0$, there exists $\delta$ such that for all $q\in [\eps,1-\eps]$,
    \begin{align}
    \label{eq:v-q-E}
    v^q_E&\in [\delta, 1-\delta],
    \\
    \label{eq:v-q-R}
    v^q_R&\leq -\delta.
    \end{align}
\end{proposition}

\begin{proof}
    We easily compute
    \begin{align*}
    v^q_E
    =
    \frac{
    \xi(q)\xi'(1)+\xi(q)\xi''(1)-q\xi'(q)\xi'(1)
    }
    {\xi(1)\xi'(1)+\xi(1)\xi''(1)-\xi'(1)^2}
    ;
    \quad\quad\quad
    v^q_R
    =
    \frac{
    q\xi'(q)\xi(1)-\xi(q)\xi'(1)
    }
    {\xi(1)\xi'(1)+\xi(1)\xi''(1)-\xi'(1)^2}.
    \end{align*}
    (Recall from Remark~\ref{rmk:not-pure} that the denominators are strictly positive as long as $\xi$ is not pure.)
    Note that
    \begin{equation}
        \label{eq:xi-quotient-increasing}
        \fr{\de}{\de q} \fr{q\xi'(q)}{\xi(q)} = \fr{\xi(q)(\xi'(q)+q\xi''(q)) - q\xi'(q)^2}{\xi(q)^2} > 0
    \end{equation}
    by Cauchy--Schwarz (as in Remark~\ref{rmk:not-pure}), and so $q \mapsto \frac{q\xi'(q)}{\xi(q)}$ is strictly increasing.
    Thus $\fr{q\xi'(q)\xi(1)}{\xi(q)} - \xi'(1)$ is negative and bounded away from $0$ on $q\in [\eps,1-\eps]$.
    Since $\xi(q)$ is also bounded away from $0$ on this interval, this implies \eqref{eq:v-q-R}.
    Moreover, since $q\mapsto \frac{q\xi'(q)}{\xi(q)}$ is increasing,
    \[
    \xi(q)\xi'(1)+\xi(q)\xi''(1)-q\xi'(q)\xi'(1)
    \geq
    \xi(q)
    \Big(
    \xi'(1)+\xi''(1)
    -
    \frac{\xi'(1)^2}{\xi(1)}
    \Big)
    ,
    \]
    i.e. $v^q_E\geq \xi(q)/\xi(1)\geq \delta$.
    It now suffices  to show $v^q_E$ is strictly increasing in $q$.
    Differentiating and rearranging, it suffices to show
    \[
        \frac{\xi''(1)}{\xi'(1)}\stackrel{?}{>}\frac{q\xi''(q)}{\xi'(q)}.
    \]
    This holds because, by a calculation analogous to \eqref{eq:xi-quotient-increasing}, $q \mapsto \fr{q\xi''(q)}{\xi'(q)}$ is increasing.
    % We argued above that $\frac{q\xi'(q)}{\xi(q)}$ is strictly increasing; applying the same to $\xi'$ in place of $\xi$ implies the above display and thus completes the proof.
\end{proof}

Given $\xi$, we let $\cR=\bbR\times [2\sqrt{\xi''(1)},\infty)$. Critical points with $(E,R)\in\cR$ will correspond to possible local maxima in Kac--Rice.
Recalling \eqref{eq:def-Theta}, it is easy to see that $\Theta$ is strictly concave and continuously differentiable on $\cR$. (Indeed the integral definition of $\kappa$ immediately implies strict concavity outside the support of $\rho$.)
For all $E\in\bbR$, define
\[
    R_*(E)
    =
    \argmax_{R\geq 2\sqrt{\xi''(1)}}
    \Theta(E,R),\quad\quad
    \Theta_*(E)=\Theta(E,R_*(E)).
\]
It is easy to see that both are finite, since the matrix $\Sigma$ in \eqref{eq:def-Theta} is positive definite.

\begin{lemma}
\label{lem:stationary-check}
    For 1RSB $\xi$, we have $\frac{\partial}{\partial R}\Theta(E_0,R_0)=0$.
\end{lemma}

\begin{proof}
    Recall \eqref{eq:def-Theta} and \eqref{eq:kappa-deriv-formula}.
    Since $R_0=y+\frac{\xi''(1)}{y}$ and $y\geq \sqrt{\xi''(1)}$ by Lemma~\ref{lem:y-ineq}, we find:
    \begin{align*}
    \nabla\Theta(E_0,R_0)
    &=
    -
    \Sigma^{-1}
    \begin{bmatrix}
        E_0 \\
        R_0
    \end{bmatrix}
    +
    \begin{bmatrix}
        0 \\
        \xi''(1)^{-1/2}
        \kappa'(R_0/\sqrt{\xi''(1)})
    \end{bmatrix}
    \\
    &=
    -
    \Sigma^{-1}
    \begin{bmatrix}
        (\xi'(1)+z\xi(1))/y \\
        y+(\xi''(1)/y)
    \end{bmatrix}
    +
    \begin{bmatrix}
        0 \\
        1/y
    \end{bmatrix}
    .
    \end{align*}
    We would like to show the second entry in this vector vanishes, and (using Cramer's rule) it is given by
    \begin{align*}
    \frac{
    \xi'(1)^2+z\xi(1)\xi'(1)-\xi(1)y^2-\xi(1)\xi''(1)
    }{y\det(\Sigma)} + \frac{1}{y}\,.
    \end{align*}
    Recalling from Lemma~\ref{lem:1RSB-parameters-setup} that $y=\sqrt{(1+z)\xi'(1)}$, the conclusion follows.
\end{proof}

\begin{lemma}
\label{lem:R*-increasing}
    If $\xi$ is 1RSB, then $R_*(E)$ is continuous and strictly increasing on $[E_0,\infty)$ with $R_*(E_0)=R_0$.
    Moreover $\Theta_*(E)$ is continuous and strictly decreasing with $\Theta_*(E_0)=0$ and $\lim\limits_{E\to\infty}\Theta_*(E)=-\infty$.
\end{lemma}

\begin{proof}
    Let
    \[
        M(R) = \max_{E\in \bbR} \Theta(E,R).
    \]
    This is easily seen to be $C^1$ on $[2\sqrt{\xi''},\infty)$ and smooth on the interior, and inherits concavity from $\Theta$.
    Since $\Theta(E,R)$ is a strictly concave quadratic function of $(E,R)$ plus a strictly concave function of $R$, it can be written as
    \[
        \Theta(E,R)
        =
        M(R) - K_1(E-K_2R)^2,
    \]
    for $K_1,K_2$ depending only on $\xi$.
    Further, one easily finds that $-M(R)\asymp R^2$ for large $R$ since $\kappa$ grows sublinearly.
    Hence $M'(R)$ is a strictly decreasing, continuous function with $\lim_{R\to\infty} M'(R)=-\infty$.

    $R_*(E)$ is the unique solution in $[2\sqrt{\xi''},\infty)$ to
    \begin{equation}
    \label{eq:R*-equation}
    M'(R_*)
    =
    2K_1K_2 (K_2 R_*-E)
    ,
    \end{equation}
    assuming such a solution exists (if not, one would have the boundary solution $R_*=2\sqrt{\xi''(1)}$.)
    Lemma~\ref{lem:stationary-check} implies $R_0 = R_*(E_0)$.
    By monotonicity arguments (or inspecting a diagram), it follows that for all $E\geq E_0$, \eqref{eq:R*-equation} admits a solution $R_*(E)$ which is continuous and strictly decreasing in $E$.
    It similarly follows that $\Theta_*(E)$ is continuous and strictly decreasing.
    Finally to show $\lim\limits_{E\to\infty}\Theta_*(E)=-\infty$, note that since $\Sigma$ is positive definite one has $\Theta(E,R)\leq -\eps(|E|+|R|)^2 + \log(|R|)$ for some $\eps=\eps(\xi)>0$.
\end{proof}
% Thus for any $R_1>R_2$, we have
% \[
% \frac{\partial }{\partial E}(\Theta(E,R_1)-\Theta(E,R_2))
% =
% 2K_1K_2(R_1-R_2)>0.
% \]
% It immediately follows that $R_*$ is an increasing function of $E$.

Similarly to Section~\ref{sec:moment}, for $0<\eta_4\ll\eta_3\ll\eta_2\ll \eta_1\ll 1$ small depending on some fixed $E>E_0$ let
\[
\wt B(E)
=
[E-\eta_3,E+\eta_3]\times [R_*(E)-\eta_3,R_*(E)+\eta_3].
\]
Recalling \eqref{eq:crt-pts-A}, we say $\bsig\in \Crt(\wt B(E))$ is \textbf{large deviation typical} if $H_N$ has no critical points $\brho$ with $R(\bsig,\brho)\geq 1-\eta_1$ and
\begin{equation}
\label{eq:LDP-band-interpolation}
    \Psi(q;\bsig)
    =
    \sup_{\brho\in\Band_q(\bsig)}
    H_N(\brho)
    \leq
    E-\eta_3,\quad \forall  q\in [-\eta_4,1-\eta_1].
\end{equation}

\begin{lemma}
\label{lem:LDP-typical-points}
    For any $E>E_0$ and with small $0<\eta_3\ll\eta_2\ll \eta_1\ll 1$, given that $\bsig\in \Crt(\wt B(E))$, $\bsig$ is large deviation typical with conditional probability at least $1-e^{-cN}$.
\end{lemma}

\begin{proof}
    Let $E_{\bsig}=H_N(\bsig)/N$ and $R_{\bsig}=\partial_{\rd} H_N(\bsig) / \sqrt{N}$.
    For $\eta_3$ small enough we must have $E_{\bsig}-E_0\geq (E-E_0)/2>0$ and $R_{\bsig}-R_0\geq (R-R_0)/2>0$ since $\bsig\in\Crt(\wt B(E))$.

    The former condition that $H_N$ has no critical points $\brho$ with $|R(\bsig,\brho)|\geq 1-\eta_1$ follows by Proposition~\ref{prop:gradients-bounded} applied to the conditionally random part of the Hamiltonian $\hH_N$, exactly as in the proof of Proposition~\ref{prop:1rsb-typical-whp}\ref{itm:1rsb-well}. (Since $R_{\bsig}>R_0=y+\frac{\xi''}{y}\geq 2\sqrt{\xi''}$, which was also the case in that proof.)

    Now we show \eqref{eq:LDP-band-interpolation}. Combining Proposition~\ref{prop:LDP-regression-coefs} and Lemma~\ref{lem:R*-increasing}, and using $(\eta_1,\eta_2)$ for $(\eps,\delta)$ in the former,
    \begin{equation}
    \label{eq:linear-regression-bound}
    \lt\la
    v^q,
    \begin{bmatrix}
        E_{\bsig} - E_0 \\
        R_{\bsig} - R_0
    \end{bmatrix}
    \rt\ra
    \leq
    v^q_E (E_{\bsig}-E_0)
    \leq
    (1-\eta_2)(E_{\bsig}-E_0),
    \quad
    \forall q\in [0,1-\eta_1].
    \end{equation}
    % For $q=0$, the estimate above is uniformly non-sharp over $\wt B(E)$, since $v^q_E=v^q_R=0$ while the right-hand side is strictly positive.
    % Hence by continuity it actually holds for all $q\in [-\eta_1,1-\eta_1]$ (even with the factor $(1-\eta_2)$ replaced by $1/2$ for $q\in [-\eta_1,0]$).

    Next, note that if one replaces the $-2\eta_3$ term with $+\eta_3$, the proof of Proposition~\ref{prop:1rsb-band-interpolation-bash} goes through for (possibly non-strictly) 1RSB models and all $q\in [0,1-\eta_1]$.
    Moreover as usual it can be made simultaneous for all $q$ by using Proposition~\ref{prop:gradients-bounded} to union bound over a finite set of $q$.
    Thus conditional on $\bsig\in \Crt(\wt B(E))$ and the values $(E_{\bsig},R_{\bsig})$, with probability $1-e^{-cN}$ we have for all $q\in [0,1-\eta_1]$ simultaneously:
    \[
    \Psi(q;\bsig)
    \leq
    E_0
    +
    \lt\la
    v^q,
    \begin{bmatrix}
        E_{\bsig} - E_0 \\
        R_{\bsig} - R_0
    \end{bmatrix}
    \rt\ra
    +\eta_3
    \stackrel{\eqref{eq:linear-regression-bound}}{\leq}
    E_0
    +
    (1-\eta_2)(E_{\bsig}-E_0)
    +
    \eta_3
    \leq
    E_{\bsig}-2\eta_3
    .
    \]
    Finally Proposition~\ref{prop:gradients-bounded} implies $\Psi(q;\bsig)$ is $C(\xi,E_{\bsig},R_{\bsig})$-Lipschitz on $[-\eta_4,0]$ with probability $1-e^{-cN}$.
    Thus we find, as desired, that with conditional probability $1-e^{-cN}$,
    \[
    \Psi(q;\bsig)
    \leq
    E_{\bsig}-\eta_3,
    \quad\forall q\in [-\eta_4,1-\eta_1].
    \qedhere
    \]
\end{proof}

Recalling \eqref{eq:Crt-def}, let $\LMAX\subseteq \Crt$ denote the set of local maxima of $H_N$.

\begin{proposition}
\label{prop:local-maxima-kac-rice}
    For any $\xi$ and $\eps>0$, there exists $c(\eps)>0$ such that for $N$ large enough,
    \[
    \bbE\lt[\lt|
    \LMAX
    \cap
    \{
    \bsig~:~\partial_{\rd}H_N(\bsig)\leq 2\sqrt{\xi''(1)}-\eps
    \}
    \rt|\rt]
    \leq
    e^{-c(\eps)N^2}.
    \]
\end{proposition}

\begin{proof}
    Similarly to Corollary~\ref{cor:1rsb-typical-whp}, by combining Cauchy--Schwarz and the Kac--Rice formula it suffices to show that for all $R\leq 2\sqrt{\xi''(1)}-\eps$,
    \[
    \bbP\Big[ \bsig\in \LMAX~\big|~\big(\partial_{\rd}H_N(\bsig),\nabla_{\sph}H_N(\bsig)\big)=(R\sqrt{N},\bzero)\Big]
    \leq
    e^{-c'(\eps)N^2}.
    \]
    Recalling \eqref{eq:riemannian-hessian}, this follows easily by the large deviation principle for the bulk spectrum of a GOE matrix, which has speed $N^2$ \cite{arous1997large}.
\end{proof}

We are ready to determine the rate function for the upper tail of $GS_N$ in 1RSB models.

\begin{theorem}
\label{thm:1RSB-LDP}
    Assume $\xi$ is 1RSB.
    Then $\max(GS_N,E_0)$ obeys a large deviation principle on $[E_0,\infty)$ with speed $N$ and good rate function $-\Theta_*(E)$.
\end{theorem}

\begin{proof}
    Since $\Theta_*(E)$ decreases continuously from $\Theta_*(E_0)=0$ to $-\infty$, and exponential tightness is clear by e.g. Borell--TIS, it suffices\footnote{Given exponential tightness, \cite[Theorem D.4 and Corollary D.6]{Guionnet} show \eqref{eq:LDP-suffices-to-show} implies a large deviation principle on $[E_0+\eps,\infty)$ for any $\eps>0$. The large deviation principle easily extends to $[E_0,\infty)$ due to the aforementioned properties of $\Theta_*$.} to show
    \begin{equation}
    \label{eq:LDP-suffices-to-show}
    \lim_{\eta\downarrow 0}
    \lim_{N\to\infty}
    \fr1N
    \log \bbP[GS_N\in [E-\eta,E+\eta]]
    \stackrel{?}{=}
    \Theta_*(E),\quad\forall E>E_0.
    \end{equation}
    Thus, fix $E>E_0$ and let
    \[
    \Crt_{typ}(\wt B(E))\sqcup\Crt_{atyp}(\wt B(E))=\Crt(\wt B(E))
    \]
    respectively denote the large deviation typical and atypical critical points of $H_N$.
    For the large deviation upper bound, recall from Lemma~\ref{lem:R*-increasing} that on $[E,\infty)\times [2\sqrt{\xi''(1)},\infty)$ the function $\Theta$ is maximized at $(E,R_*(E))$ with value $\Theta_*(E)$.
    We claim that the expected number of local maxima $\bsig\in\LMAX$ satisfying $H_N(\bsig)/N\geq [E,\bar E]$ is at most $\exp(N\Theta_*(E)+o(N))$ for any $\bar E<\infty$ independent of $N$.
    Indeed, Proposition~\ref{prop:local-maxima-kac-rice} shows that points $R\leq 2\sqrt{\xi''(1)}-\eps$ contribute a negligible amount.
    Sending $\eps\to 0$ slowly with $N$ and applying Lemma~\ref{lem:moment-rates} (and continuity of $\Theta$) yields the claim.
    Since the global maximum of $H_N$ is of course a local maximum, this together with exponential tightness of the ground state yields the upper bound.

    % \bhcomment{I want to add a sentence here about critical points with $R < 2\sqrt{\xi''(1)}$ being maxima with probability $e^{-CN^2}$ - have to run but will do in a bit}

    For the lower bound,
    exactly as in Proposition~\ref{cor:1rsb-typical-whp} and its use in proving Proposition~\ref{prop:1rsb-restricted-1mt},
    we may deduce from Lemma~\ref{lem:LDP-typical-points} that
    \[
    \bbE|\Crt_{atyp}(\wt B(E))|
    \leq
    e^{-cN/3}
    \bbE|\Crt(\wt B(E))|
    .
    \]
    In particular
    \[
    \bbE|\Crt_{typ}(\wt B(E))|
    \geq
    \bbE|\Crt(\wt B(E))|/2
    \geq
    \exp(-N\Theta_*(E)\pm o(N))/2
    .
    \]
    By definition, any two distinct large deviation typical points have overlap at most $-\eta_4$. Hence there are almost surely at most $2\eta_4^{-2}$ large deviation typical points in total, for any $H_N$ (because their Gram matrix of overlaps must be positive semi-definite).
    Therefore
    \begin{align*}
    \bbP[|\Crt_{typ}(\wt B(E))|\geq 1]
    &\geq
    \eta_4^2 \bbE|\Crt_{typ}(\wt B(E))|/2
    \geq
    \eta_4^2\exp(-N\Theta_*(E)\pm o(N))/4
    \\
    &\geq
    \exp(-N\Theta_*(E)\pm o(N)).
    \end{align*}
    By definition, if $|\Crt_{typ}(\wt B(E))|\geq 1$ then $GS_N\geq E-\eta_3$.
    Since $\eta_3$ was arbitrarily small, we obtain
    \[
    \liminf_{\eta\downarrow 0}
    \lim_{N\to\infty}
    \fr1N
    \log \bbP[GS_N\geq E-\eta]
    \geq
    \Theta_*(E),\quad\forall E>E_0.
    \]
    Since we already established the large deviation upper bound with strictly increasing rate function $-\Theta_*$ as well as exponential tightness, the previous display implies \eqref{eq:LDP-suffices-to-show} as desired.
\end{proof}

\subsection{Transition to Quadratic Speed}
\label{subsec:LDP-speed}

In this subsection, $\xi$ is a general model, not necessarily 1RSB.
We show in Theorem~\ref{thm:LDP-speed} that the large deviations of $GS_N$ are of speed $O(N)$ above $\cQ(\xi-\gamma_1^2 t)$ and $\Omega(N^2)$ below.
The following exact orthogonalization lemma is crucial to show the latter result.
\begin{lemma}
\label{lem:exact-orthogonalization}
    Fix small constants $c,\delta>0$. Suppose $\bsig^1,\dots,\bsig^k\in\cS_N$ for $k=e^{cN}$ satisfy $|R(\bsig^i,\bsig^j)|\leq \delta$ for all $1\leq i<j\leq k$.
    Then for $c'>0$ depending only on $c,\delta$ and for $N$ large enough, there exists a subset $A\subseteq [k]$ of size $|A|\geq c'N$ and points $\{\wt\bsig^a\}_{a\in A}$ such that:
    \begin{equation}
    \label{eq:many-exactly-orthogonal}
    \begin{aligned}
    R(\wt\bsig^a,\wt\bsig^{a'})&=0,\quad \forall a\neq a'\in A,
    \\
    \|\wt\bsig^a-\bsig^a\|_2 &\leq \delta^{0.01}\sqrt{N}.
    \end{aligned}
    \end{equation}
\end{lemma}

\begin{proof}
    Let $A\subseteq [k]$ be any maximal subset such that there exist $\{\wt\bsig^a\}_{a\in A}$ obeying \eqref{eq:many-exactly-orthogonal}, and assume for sake of contradiction that $|A|<c'N$.
    For $i\in [k]$, let $\wh\bsig_i$ be the projection of $\bsig_i$ onto $\Span\big(\{\wt\bsig^a\}_{a\in A}\big)$.
    By maximality of $A$, we have $\|\wh\bsig_i\|_2 \geq \delta^{0.1}\sqrt{N}$ for all $i$.
    With $\cB(A)$ the radius $\sqrt{N}$ ball in $\Span\big(\{\wt\bsig^a\}_{a\in A}\big)$, let $\hat\cB=\cB(A)\backslash \delta^{0.1}\cB(A)$.
    Then, for a universal $C>0$, $\hat\cB(A)$ admits a covering by $\exp(C|A|\log(1/\delta))$ radius $\delta\sqrt{N}$ balls, whose centers are disjoint from $\delta^{0.1}\cB(A)/2$.
    Since we assumed $|A|<c'N$, we have (for small enough $c'$)
    \[
    \exp(C|A|\log(1/\delta))
    \leq
    \exp(cN/3).
    \]
    Hence by the pigeonhole principle, there exists $J\subseteq [k]\backslash A$ with $|J|\geq\exp(cN/3)$ such that the points $\{\wh\bsig^j\}_{j\in J}$ in $\wh\cB$ are all contained in a radius $\delta\sqrt{N}$ ball centered at some $\wh\bsig^J\in \Span\big(\{\wt\bsig^a\}_{a\in A}\big)$ with $\|\wh\bsig^J\|\geq \delta^{0.1}\sqrt{N}/2$.
    Therefore
    \[
    \frac{1}{|J|}
    \sum_{j\in J}
    \la \bsig^j,\wh\bsig^J\ra
    =
    \frac{1}{|J|}
    \sum_{j\in J}
    \la \wh\bsig^j,\wh\bsig^J\ra
    \geq
    \frac{1}{|J|}
    \sum_{j\in J}
    \big(\la \wh\bsig^J,\wh\bsig^J\ra
    -
    \|\wh\bsig^j-\wh\bsig^J\|_2\cdot \|\wh\bsig^J\|_2
    \big)
    \geq
    \delta^{0.3} N
    .
    \]
    On the other hand, since we assumed $|R(\bsig^i,\bsig^j)|\leq \delta$, we have
    \begin{align*}
    \lt\|\frac{1}{|J|}
    \sum_{j\in J}
    \bsig^j\rt\|_2^2
    =
    \frac{1}{|J|^2}
    \sum_{j_1,j_2\in J}
    \la \bsig^{j_1},\bsig^{j_2}\ra
    \leq
    \frac{1}{|J|^2}\lt(|J|+|J|^2\delta N\rt)
    \leq
    2\delta N.
    \end{align*}
    Since $\|\wh\bsig^J\|_2\leq \sqrt{N}$, Cauchy--Schwarz gives the desired contradiction.
\end{proof}
Recall $\gamma_1$ is the weight of the degree-$1$ interactions in \eqref{eq:def-HN}, and $\xi'(0) = \gamma_1^2$.
For $h\ge 0$, let
\[
    \xi^{\gamma_1 \leftarrow h}(t)
    = \xi(t) - \gamma_1^2 t + h^2 t
\]
denote $\xi$ with this interaction weight replaced by $h$.
In particular $\xi^{\gamma_1 \leftarrow 0}(t) =\xi(t)-\gamma_1^2 t$.

\begin{theorem}
    \label{thm:LDP-speed}
    For any $E>\cQ(\xi^{\gamma_1\leftarrow 0})$,
    \begin{equation}
    \label{eq:linear-speed}
    \limsup_{\eps\to 0}
    \limsup_{N\to\infty}
    -\fr{1}{N}
    \log\bbP[GS_N \in [E-\eps,E+\eps]]
    % \log\bbP[GS_N \geq E]
    \leq
    C_1(\xi,E).
    \end{equation}
    On the other hand, for any $\eps>0$,
    \begin{equation}
    \label{eq:quadratic-speed}
    \liminf_{N\to\infty}
    -\fr{1}{N^2} \log\bbP[GS_N \leq \cQ(\xi^{\gamma_1 \leftarrow 0})-\eps]
    \geq
    C_2(\xi,\eps).
    \end{equation}
\end{theorem}

\begin{proof}
    We first prove \eqref{eq:linear-speed}, assuming further that $\xi'(0) > 0$.
    By Theorem~\ref{thm:chen-sen-spherical} and Proposition~\ref{prop:gradients-bounded}, $\cQ(\xi^{\gamma_1 \leftarrow h})$ is uniformly Lipschitz in $h$.
    Moreover clearly $\lim_{h\to\infty}\cQ(\xi^{\gamma_1 \leftarrow h})=\infty$.
    Thus $E=\cQ(\xi^{\gamma_1 \leftarrow h})$ for some $h>0$.
    Let $\bg = (g_1,\ldots,g_N)$ be the vector of degree $1$ disorder coefficients in \eqref{eq:def-HN}.
    For small $\eta > 0$,
    \[
        \fr{\gamma_1 \norm{\bg}_2}{\sqrt{N}} \in [h-\eta,h+\eta]
    \]
    occurs with probability at least $e^{-C_1(\xi,E) N}$.
    Conditional on this event, $\bbP[GS_N \in [E-\eps,E+\eps]] \ge 1/2$.
    This proves \eqref{eq:linear-speed} if $\xi'(0) > 0$.

    Next, suppose $\xi'(0) = 0$.
    Assuming $\gamma_p>0$, we consider the conditional behavior of $H_N$ on the large deviation event $E_{N,x}$ that the order $p$ coefficient $g_{1,1,\dots, 1}$ (as in \eqref{eq:def-HN}) satisfies
    \[
    g_{1,1,\dots, 1}=x\sqrt{N}.
    \]
    Using Proposition~\ref{prop:gradients-bounded} to discretize $S_N$ into bands with fixed first coordinate, and applying the zero temperature Parisi formula to each band, it easily follows that, with $\txi_q$ as in \eqref{eq:txi-def},
    \[
    GS(x)\equiv
    \lim_{N\to\infty}
    \bbE[GS_N~|~E_{N,x}]
    =
    \sup_{0\leq q\leq 1}
    \lt\{
    \cQ(\txi_q)+ q^p x
    \rt\}.
    \]
    Moreover this limit holds locally uniformly in $x$.
    In particular $GS(x)$ is continuous and strictly increasing, and (since $\gamma_1=0$) $GS(0)=\cQ(\xi)$. Hence for some $x_*(E),\delta>0$ we have $GS(x)\in [E-\eps/2,E+\eps/2]$ for all $x\in [x_*(E)-\delta,x_*(E)+\delta]$.
    Then by Borell-TIS, for $N$ sufficiently large, we conclude \eqref{eq:linear-speed} from:
    \[
    \bbP[GS_N\in E-\eps,E+\eps]
    \geq
    \bbP[N^{-1/2} g_{1,1,\dots, 1}\in [x_*(E)-\delta,x_*(E)+\delta]]/2
    \geq
    e^{-C(\xi,E)N}
    .
    \]

    % Let $\be_1 = (1,0,\ldots,0)$ and
    % \[
    %     GS'_N = \fr{1}{N} \max\lt\{
    %         H_N(\bsig) : \bsig \in \Band_{1/2}(\sqrt{N}\be_1)
    %     \rt\}.
    % \]
    % The restriction of $H_N$ to $\Band_{1/2}(\sqrt{N}\be_1)$ is a spin glass with mixture $\txi(t) = \xi(\fr14 + \fr34t)$.
    % It is clear that $GS_N \ge GS'_N$, and since $\txi'(0) > 0$, we have $\bbP[GS'_N \ge E] \ge e^{-C_1(\xi,E) N}$.
    % \bhcomment{if $\xi'(0) = 0$, show LDP lower bound on the $1/2$-overlap band? slightly annoying because GS of the whole model might be \textbf{bigger} than GS of the band. If we don't want to deal with this we can bound $\bbP(GS_N \ge E)$ instead...}
    % \bhcomment{alternatively, planting $H(\bx) = EN$ is like planting a $\xi(R(\bx, \bsig))$ spike, and this should be equivalent to an external field by some kind of fenchel duality?}

    We turn to the proof of \eqref{eq:quadratic-speed}.
    We start from Corollary~\ref{cor:many-orthogonal}, applied to
    \[
        H_{N,\ge 2}(\bsig) = H_N(\bsig) - \gamma_1 \la \bg, \bsig \ra,
    \]
    which is a Hamiltonian with mixture $\xi^{\gamma_1 \leftarrow 0}$.
    This implies the high-probability existence of $\bsig^1,\dots,\bsig^k\in\cS_N$ for $k=e^{cN}$ such that $|R(\bsig^i,\bsig^j)|\leq \delta\ll\eps$ for all $1\leq i<j\leq k$ and $H_{N,\ge2}(\bsig^i)\geq \cQ(\xi^{\gamma_1 \leftarrow 0})-\frac{\eps}{4}$.
    On this event, Lemma~\ref{lem:exact-orthogonalization} ensures the existence of $A\subseteq [k]$ with $|A|\geq c'N$ and points $\{\wt\bsig^a\}_{a\in A}$ obeying \eqref{eq:many-exactly-orthogonal}.
    On the event of Proposition~\ref{prop:gradients-bounded} we then have
    \[
    \frac{1}{c'N^2}
    \max_{\vbrho\in\Band_{c'N,1,0}(\bzero)}
    \sum_{i=1}^{c'N}
    H_{N,\ge2}(\brho^i)
    \geq
    \min_{\wt\bsig^a}
    H_{N,\ge2}(\wt\bsig^a)/N
    \ge
    \cQ(\xi^{\gamma_1 \leftarrow 0})-\frac{\eps}{2}.
    \]
    Similarly to \cite[Proposition 1]{subag2018free}, the left-hand side above is sub-Gaussian with standard deviation proxy $O(N^{-2})$.
    With $c''=c''(\xi,c,c',\delta)>0$ a small constant, we find that  with probability $1-e^{-c'' N^2}$,
    \[
    \frac{1}{c'N^2}
    \max_{\vbrho\in\Band_{c'N,1,0}(\bzero)}
    \sum_{i=1}^{c'N}
    H_{N,\ge2}(\brho^i)
    \geq
    \cQ(\xi^{\gamma_1 \leftarrow 0})-\frac{3\eps}{4}.
    \]
    Let $(\wh\brho^1,\dots,\wh\brho^{c'N})$ attain the maximum on the left-hand side (and depend measurably on $H_{N,\ge 2}$).

    % \bhcomment{if there is an external field the lower tail should not have speed $N^2$?}

    Finally we add back in the external field $\gamma_1 \la \bg, \bsig \ra$.
    Note that $\{\wt\bsig^a\}_{a\in A}$ can be chosen (in some measurable way) depending only on $H_{N,\ge 2}$, so we may take $\bg$ independent of $\{\wh\brho^i\}_{1\leq i\leq c'N}$. Then $\sum_{i\leq c'N}\gamma_1 \la \bg,\wh\brho^i\ra$
    is conditionally a centered Gaussian with variance $c'N^2$, so it has absolute value smaller than $\eps c' N^2/4$ with probability $1-e^{-c''N^2}$.
    On this event and that of the preceding display, we find as desired:
    \begin{align*}
    \max_{\bsig\in\cS_N}
    H_N(\bsig)/N
    &\geq
    \frac{1}{c'N^2}
    \max_{\vbrho\in\Band_{c'N,1,0}(\bzero)}
    \sum_{i=1}^{c'N}
    H_N(\brho^i)
    \geq
    \Big(\frac{1}{c'N^2}
    \sum_{i=1}^{c'N}
    H_{N,\ge 2}(\wh\brho^i)\Big)
    -\frac{\eps}{4}
    \geq
    \cQ(\xi^{\gamma_1 \leftarrow 0})-\eps.
    \qedhere
    \end{align*}
\end{proof}

\begin{remark}
\label{rem:lower-tail-existing-results}
Given a weaker version of Corollary~\ref{cor:many-orthogonal} with $k_N\leq e^{o(N)}$, one finds $|A|\geq\Omega(\log k_N)$ in Lemma~\ref{lem:exact-orthogonalization}, which implies speed $\Omega(N\log k_N)$.
In particular, Chatterjee's ``multiple peaks'' property \cite{chatterjee2014superconcentration} suffices to obtain super-linear speed for the lower tail.
At positive temperature, the super-linearity in the lower tail was essentially predicted by Dotsenko-Franz-M{\'e}zard in \cite{dotsenko1994partial} and proved in \cite{talagrand2007large,jagannath2017approximate} using a similar ``orthogonal structures'' idea; see also \cite{chen2023gaussian} which derives it from superconcentration.
The main new feature of Theorem~\ref{thm:LDP-speed} is the quadratic rate, which is best possible when $\gamma_2\neq 0$.
It is natural to conjecture that $N^{p_{\min}}$ is the correct lower tail speed for $p_{\min}=\min\{p~:~\gamma_p>0\}$.
\end{remark}

\subsubsection*{Acknowledgement}
We thank Wei-Kuo Chen and Peter Mottishaw for helpful comments.
B.H. was supported by a Google PhD Fellowship, NSF awards CCF-1940205 and DMS-1940092, and NSF-Simons grant DMS-2031883.

% \fontsize{10}{12}
\footnotesize
% \small
% \newpage
\conditionalpagebreak

\bibliographystyle{alpha}
\bibliography{bib}

\end{document}